\documentclass[11pt]{article}

\usepackage[margin=1in]{geometry} 
\usepackage{amsmath,amsthm,amssymb,comment}
\usepackage{algorithmic}
\usepackage{appendix}
\usepackage{tikz}
\usepackage{graphicx}
\usepackage{amsfonts}
\usepackage{hyperref}
\usepackage{float}
\usepackage{subfigure}

\usepackage{algorithm,float}
\usepackage{lipsum}

\makeatletter
\newenvironment{breakablealgorithm}
{
		\begin{center}
			\refstepcounter{algorithm}
			\hrule height.8pt depth0pt \kern2pt
			\renewcommand{\caption}[2][\relax]{
				{\raggedright\textbf{\ALG@name~\thealgorithm} ##2\par}%
				\ifx\relax##1\relax 
				\addcontentsline{loa}{algorithm}{\protect\numberline{\thealgorithm}##2}%
				\else 
				\addcontentsline{loa}{algorithm}{\protect\numberline{\thealgorithm}##1}%
				\fi
				\kern2pt\hrule\kern2pt
			}
		}{
		\kern2pt\hrule\relax
	\end{center}
}
\makeatother

\newtheorem{theorem}{Theorem}

\newtheorem{lemma}{Lemma}
\newtheorem{claim}[]{Claim}

\newtheorem{conjecture}{Conjecture}

\begin{document}
	
	
	\title{3-Coloring $P_t$-Free Graphs With Only One Prescribed Induced Odd Cycle Length}
	\author{Yidong Zhou\footnote{College of Computer Science, Nankai University, Tianjin 300350, China} , Mingxian Zhong\footnote{Lehman College and the Graduate Center, CUNY, Bronx, NY 10468, USA}   and Shenwei Huang\footnote{School of Mathematical Sciences and LPMC, Nankai University, Tianjin 300071, China. The corresponding author: shenweihuang@nankai.edu.cn. Sponsored by CCF-Huawei Populus Grove Fund.}}
	
	\maketitle
	\begin{abstract}
		A graph is $P_t$-free if it contains no induced subgraph isomorphic to a $t$-vertex path. 
		A graph is not bipartite if and only if it contains an induced subgraph isomorphic to a $k$-vertex cycle, where $k$ is odd. 
		We focus on the 3-coloring problem for $P_t$-free graphs that have only one prescribed induced odd cycle length. 
		For any integer $t$ and any odd integer $k$, let $\mathcal{G}_{t,k}$ be the class of graphs that are $P_{t}$-free and all their induced odd cycles must be $C_k$. 
		In this paper, we present a polynomial-time algorithm that solves the 3-coloring problem for any graph in $\mathcal{G}_{10,7}$.

	\end{abstract}
	
	
	\section{Introduction}
	All graphs in this paper are finite and simple. 
	For general graph theory notations, we follow \cite{BA08}. 
	A {\em $k$-coloring} of a graph $G=(V(G),E(G))$ is a function $f:V\rightarrow \{1,2,\ldots, k\}$ such that $f(u)\neq f(v)$ whenever $uv\in E(G)$. 
	We call $G$ {\em $k$-colorable} if $G$ admits a $k$-coloring. 
	A well-known problem surrounding graph colorings is the {\em $k$-coloring problem}, which determines whether a given graph $G$ is $k$-colorable or not, and obtains a $k$-coloring of $G$ if it is $k$-colorable.         
	For $k\leq 2$, the problem is solvable in linear time. 
	For $k\geq 3$, the problem is one of Karp's 21 NP-complete problems \cite{KR72}.
	
	Because of the notorious hardness of the $k$-coloring problem, efforts were made to understand the problem on restricted graph classes. 
	Some of the most prominent classes are the classes of $H$-free graphs, which is defined as the class of all graphs not containing $H$ as an induced subgraph. 
	Kami\'{n}ski and Lozin \cite{KM07} and independently Kr\'{a}l, Kratochv\'{i}l, Tuza, and Woeginger \cite{KD01} proved that for any fixed $k,g\geq 3$, the $k$-coloring problem is NP-complete for the class of graphs containing no cycle of length less than $g$.
	As a consequence, if the graph $H$ contains a cycle, then $k$-coloring problem is NP-complete for $k\geq 3$ for the class of $H$-free graphs. 
	Moreover, if $H$ is a forest with a vertex of degree at least 3, then $k$-colouring problem is NP-complete for $H$-free graphs and $k\geq 3$ \cite{HI81,LD83}. 
	Combined, these two results only leave open the complexity of the $k$-coloring problem for the class of $H$-free graphs where $H$ is a disjoint union of paths. 
	See for \cite{HP13} as a survey by Hell and Huang on the complexity of coloring graphs without paths and cycles of certain lengths. 
	We denote a path and a cycle on $t$ vertices by $P_t$ and $C_t$, respectively. 
	
	Huang \cite{HS13} proved that 4-coloring problem for $P_7$-free graphs is NP-complete, and that 5-coloring problem for $P_6$-free graphs is also NP-complete. 
	On the other hand, Ho\`{a}ng, Kami\'{n}ski, Lozin, Sawada and Shu \cite{HC10} prove that $k$-coloring problem for $P_5$-free graphs can be solved in polynomial time for any fixed $k\geq 5$, and Chudnovsky, Spirkl and Zhong \cite{CM24,CM241} showed that $4$-coloring problem for $P_6$-free graphs can be also solved in polynomial time. 
	Combined with these results, we have a complete classification of the complexity of $k$-colouring $P_t$-free graphs for any fixed $k\geq 4$. 
	For $k=3$, the strongest known results related to our work are due to Bonomo, Chudnovsky, Maceli, Schaudt, Stein and Zhong \cite{BF18} who prove that 3-coloring problem for $P_7$-free graphs can be solved in polynomial time.
	
	On the other hand, a graph $G$ is bipartite if and only if $G$ has no odd cycles and $G$ is 2-colorable if and only if $G$ is bipartite. 
	Therefore, a natural question is to ask what  the chromatic numbers are if we lift the restriction on cycles. 
	Bollob\'{a}s and Erd\H{o}s \cite{EP90} conjectured that $\chi(G)\leq 2r+2$ for every graph $G$ which contains at most $r$ different odd cycle lengths. 
	This is proved by Gy\'{a}rf\'{a}s \cite{GA92} and is best possible considering $G=K_{2r+2}$. 
	Mih\'{o}k and Schiermeyer \cite{MP04} showed $\chi(G)\leq 2s+3$ for every graph $G$ which contains at most $s$ different even cycle lengths.   
	When it comes to induced cycle length, a few results are given. 
	Randerath and Schiermeyer \cite{RB01} gave the chromatic number of graphs with two prescribed induced cycle lengths. 
	Moreover, they presented a polynomial-time algorithm to 3-color all triangle-free graphs of this kind. 
	Motivated by this line of the research, we focus on the 3-coloring problem for graphs whose induced odd cycles must have the same length.               
	For any integer $t$ and any odd integer $k$, let $\mathcal{G}_{t,k}$ be the class of graphs that are $P_{t}$-free and all its induced odd cycles must be $C_k$. 
	
	\begin{conjecture}
		For any fixed $t,k$, 3-coloring problem for $\mathcal{G}_{t,k}$ can be solved in polynomial time. 
	\end{conjecture}
	Note that $\mathcal{G}_{t,3}$ is a subclass of odd-hole-free graphs. 
	By The Strong Perfect Theorem \cite{CM06}, for any odd-hole-free $G$ without $K_4$, either $G$ contains an odd anti-hole with length 7 or $G$ is 3-colorable. 
	Thus, 3-coloring problem for $\mathcal{G}_{t,3}$ can be solved in polynomial time. 
	Due to the work of Bonomo, Chudnovsky, Maceli, Schaudt, Stein and Zhong \cite{BF18}, 3-coloring problem for $\mathcal{G}_{t,k}$ can be solved in polynomial time for any $t\leq 7$. 
	For $t=8$, Chudnovsky and Stacho \cite{CM18} proved that all graphs in $\mathcal{G}_{8,7}$ are 3-colorable. 
	By a similar analysis, one can prove that $\mathcal{G}_{t,t-1}$ are 3-colorable for every even integer $t$. 
	A result given by Alberto Rojas Anr\'{i}quez and Maya Stein \cite{AA23} showed that 3-coloring problem for $\mathcal{G}_{9,7}$ can be solved in $O(|V(G)|+|E(G)|)$ time.
	In this paper, we prove that  
	\begin{theorem}[Main Theorem]\label{Main}
		Given a graph $G\in \mathcal{G}_{10,7}$, it can be decided in $O(|V(G)|^{21}(|V(G)|+|E(G)|))$ time whether $G$ admits a $3$-coloring and find one if it exists.  
	\end{theorem}

	We call non-adjacent vertices $u,v$ {\em comparable}, if $N(u)\subseteq N(v)$ or $N(v)\subseteq N(u)$, and call $G$ has no {\em comparable pair} if $G$ has no such $u$ and $v$. It is safe to assume that $G$ has no comparable pair by Lemma~\ref{comparable}. By a structural analysis of an induced odd cycle $C$ and its attachments, we first partition $V(G)$ into different sets based on their relations with $C$, and prove that all these vertices lie at most  distance 2 from $C$. 
	We  could then reduce $G$ to a subgraph $G'$ such that each component in $V(G')\setminus N[C]$ is either an isolated vertex or a complete bipartite graph. 
	This is done in Section $3$ by applying this step for every induced $C_7$ of $G$.

	For the non-trivial components, we can partition them into multiple cases depending on their size and neighbors in $N(C)$. 
	For each case, we identify a finite number of vertices that dominate all components in this case, except for some vertices that can be safely reduced. 	
	This is done in Section 4.

	The main idea to handle isolated vertices lies in Lemma~\ref{wonderful}, which essentially says that if $G$ is 3-colorable, then either it has  some ``good" colorings or some coloring satisfies a ``mono" condition. 
	In the former case, we can find a bounded set of vertices which dominates the isolated vertices. 
	In the latter case, we can reduce the problem to a variant of the $2$-SAT problem. To prove this Lemma we start with an arbitrary $3$-coloring of $G$ that is not ``good" and obtain a $3$-coloring which satisfies the ``mono" condition. 
	This is one of our key technical contributions and is done in Section 5.  
	
	To prove Theorem~\ref{Main}, we divide the possible $3$-colorings of $C_7$ into three types. 
	In Section 6, we start with guessing the first type of coloring for some $C_7$ and prove the main theorem under such assumption by combining results proved in previous sections. 
	Then we deal with the case where there is no first type coloring but some $C_7$ are colored as the second type and lastly, we consider a relatively easier case where no $C_7$ is colored as the first or the second type.

	\section{Notation and Preliminaries}

	We start by establishing some notation and preliminary results. 
	Let $G$ be a graph. 
	For any $x\in V(G)$, let $N(x)$ denote the set of all neighbors of $x$ in $G$ and $N[x]:=N(x)\cup \{x\}$.
	For $A,B\subseteq V(G)$, we say that $A$ is {\em complete} to $B$ if every vertex in $A$ is adjacent to every vertex in $B$.
	Let $k$ be a positive integer. Set $[k]:=\{1,2,\ldots,k\}$.
	
	If $u$ and $v$ are comparable, say $N(u)\subseteq N(v)$, then $G$ is $3$-colorable if and only if $G\setminus u$ is colorable since we can extend the color of $v$ on $u$. 
	Therefore, we can obtain the following folklore result.                             
	\begin{lemma}\label{comparable}
		Let $G$ be a graph.  
		There is an algorithm to obtain an induced subgraph $G'$ of $G$ in $O(|V(G)|+|E(G)|)$ time such that $G'$ has no comparable pair and $G$ is 3-colorable if and only if $G'$ is 3-colorable. 
		Moreover, any 3-coloring of $G'$ can be extended to a 3-coloring of $G$ in $O(|V(G)|+|E(G)|)$ time. 
	\end{lemma}

	In the context of 3-coloring, we call a family $\mathcal{L}$ of lists $L:=\{L(v): v\in V(G)\}$, where $L(v)\subseteq \{1,2,3\}$ for each $v\in V(G)$, a {\em palette} of the graph $G$. 
	A palette $L'$ is a {\em subpalette} of $L$ if $L'(v)\subseteq L(v)$ for each $v\in V(G)$. 
	Given a graph and a palette $L$, we say that a 3-coloring $c$ of $G$ is a {\em coloring} of $(G,L)$ if $c(v)\in L(v)$ for all $v\in V(G)$. 
	We say that $(G,L)$ is {\em colorable} if there exists a coloring of $(G,L)$. 
	Let $L_0$ be the palette such that $L_0(v)=\{1,2,3\}$ for all $v\in V(G)$. 
	Then $G$ is 3-colorable if and only if $(G,L_0)$ is colorable. 
	We denote by $(G,\mathcal{L})$ a graph and a collection $\mathcal{L}$ of palettes of $G$. 
	We say that $(G,\mathcal{L})$ is {\em colorable} if $(G,L)$ is {\em colorable} for some $L\in \mathcal{L}$. 
	
	An {\em update} of the list of a vertex $v$ {\em from} $w$ means we delete an entry from the list of $v$ that appears as the unique entry of the list of a neighbor $w$ of $v$. 
	Clearly, such an update does not change the colorability of the graph. 
	If a palette $L'$ is obtained from a palette $L$ by updating repeatedly until for every vertex $v$, if $v$ has a neighbor $u$ with $L'(u)= i$, then $i\notin L'(v)$, we say we obtained $L'$ from $L$ by {\em updating} and $L'$ is {\em updated}.
	Note that updating can be carried out in $O(|V(G)|+|E(G)|)$ time. 
	
	By reducing to an instance of 2-SAT, which can be solved in linear time in the number of variables and clauses \cite{AB79}, several authors \cite{EK86,EP79,VV76} independently proved the following. 
	\begin{lemma}\label{2Sat}
		Let $L$ be a palette of a graph $G$ such that $|L(v)|\leq 2$ for all $v\in V(G)$. 
		Then a coloring of $(G,L)$, or a determination that none exists, can be obtained in $O(|V(G)|+|E(G)|)$ time. 
	\end{lemma}
	Let $G$ be a graph. 
	A subset $S$ of $V(G)$ is called {\em monochromatic} with respect to a given coloring $c$ of $G$ if $c(u)=c(v)$ for all $u,v\in S$.  
	Let $L$ be a palette of $G$, and $Z$ a set of subsets of $V(G)$. 
	We say that $(G,L,Z)$ is {\em colorable} if there is a coloring $c$ of $(G,L)$ such that $S$ is monochromatic with respect to $c$ for all $S\in Z$. 
	
	A triple $(G',L',Z')$ is a {\em restriction} of $(G,L,Z)$ if 
	\begin{itemize}
		\item[(a)] $G'$ is an induced subgraph of $G$,
		\item[(b)] the palette $L'$ is a subpalette of $L$ restricted to the set $V(G)$,
		\item[(c)] $Z'$ is a set of subsets of $V(G')$ such that if $S\in Z$ then $S\cap V(G')\subseteq S'$ for some $S'\in Z'$.
	\end{itemize}
	Let $\mathcal{R}$ be a set of restrictions of $(G,L,Z)$. 
	We say that $\mathcal{R}$ is {\em colorable} if at least one element of $\mathcal{R}$ is colorable. 
	If $\mathcal{L}$ is a set of palettes of $G$, we write $(G,\mathcal{L},Z)$ to mean the set of restrictions $(G,L',Z)$ where $L'\in \mathcal{L}$. 
	Under this assumption, Lemma \ref{2Sat} can be easily modified to obtain a generalized result as follows \cite{BF18}. 
	\begin{lemma}\label{improve 2sat}
		If a palette $L$ of a graph $G$ is such that $|L(v)|\leq 2$ for all $v\in V(G)$, and $Z$ is a set of mutually disjoint subsets of $V(G)$, then a coloring of $(G,L,Z)$, or a determination that none exists, can be obtained in $O(|V(G)|+|E(G)|)$ time. 
	\end{lemma}

	\section{Cleaning of $G$}
	Throughout the paper, for an induced odd cycle $C:=v_1-v_2-v_3-v_4-v_5-v_6-v_7-v_1$ and $i\in [7]$ in a graph $G\in \mathcal{G}_{10,7}$ , we define   
	\begin{itemize}
		\item $A_i^C=\{x|N_C(x)=\{v_i\}\}$; 
		\item $B_i^C=\{x|N_C(x)=\{v_{i-1},v_{i+1}\}\}$. 
	\end{itemize}
	Let $A^C=\bigcup A_i^C$, $B^C=\bigcup B_i^C$. Note that $A^C\cup B^C= N(C)$ since $G$ is $C_5$-free and $C_3$-free. For simplicity, we omit the superscript for these sets and write $A,B,A_i,B_i$ when there is no risk of confusion.
	
	We call a connected graph $G$ in $\mathcal{G}_{10,7}$ with no comparable pair {\em cleaned} if for every induced odd cycle $C$, $V(G)=N[N[C]]$ and every component in $G\setminus N[C]$ is either an isolated vertex or a complete bipartite graph. 
	In this section, our aim is to reduce a connected graph $G$ to be cleaned in polynomial time. 
	Formally, we prove the following Lemma. 
	
	\begin{lemma}\label{cleaned}
		Let $G$ be a connected graph in $\mathcal{G}_{10,7}$.  
		There is an algorithm to obtain a cleaned induced subgraph $G'$ of $G$ in $O(|V(G)|^7(|V(G)|+|E(G)|))$ time such that $G$ is 3-colorable if and only if $G'$ is 3-colorable. 
		Moreover, any 3-coloring of $G'$ can be extended to a 3-coloring of $G$ in $O(|V(G)|+|E(G)|)$ time. 
	\end{lemma}
	
	The corresponding Algorithm for Lemma~\ref{cleaned} is as follows.

	\begin{breakablealgorithm}
		\caption{Algorithm of Cleaning}
		\label{alg:clean}

		\begin{algorithmic}[1]
			\REQUIRE A connected graph $G$ in $\mathcal{G}_{10,7}$.   
			\ENSURE A cleaned graph or determine that $G$ is bipartite.   
			\IF{$G$ has no induced $C_7$}
			\RETURN $G$ is bipartite. 
			\ELSE
			\WHILE{$G$ has two non-adjacent vertices $u,v$ such that $N(u)\subseteq N(v)$}
			\STATE Delete $u$ and let $G=G\setminus \{u\}$.
			\ENDWHILE
			\FOR{each induced $C_7$ $C\in G$}
			\STATE Delete all vertices at distance from  $C$ at least 3 and all their neighbors.
			\FOR{each non-complete bipartite component in $G\setminus N[C]$}
			\STATE Delete all vertices in this component except for an edge.
			\ENDFOR
			\STATE Let $G$ be the remaining graph.
			\ENDFOR
			\ENDIF
			\RETURN A cleaned graph $G$.
		\end{algorithmic}
	\end{breakablealgorithm}
	
	\begin{proof}[Proof of Lemma \ref{cleaned}.]
		For an induced odd cycle $C:=v_1-v_2-v_3-v_4-v_5-v_6-v_7-v_1$ of $G$, we classify some sets as follows. 
		Define   
		\begin{itemize} 
			\item $D_1=\{x|d(x,C)=2\}$;
			\item $D_2=\{x|d(x,C)=3\}$; 
			\item $D'_1=\{x|x\in D_1,x\text{ has no neighbors in }D_2\}$;
			\item $D''_1=\{x|x\in D_1,x\text{ has a neighbor in }D_2\}$.
		\end{itemize}
		Note that $D_1=D'_1\cup D''_1$. 
		
		\begin{claim}\label{basic}
			Here are several useful properties of $G$:
			\begin{itemize}
				\item[(a)] $V(G)=V(C)\cup A\cup B\cup D_1\cup D_2$; 
				\item[(b)] for every vertex $x\in D''_1$, $x$ has no neighbors in $A$;
				\item[(c)] $D_2$ is a stable set;
				\item[(d)] $D_1''$ is a stable set and has no neighbors in $D'_1$. 
			\end{itemize} 
		\end{claim}                                           
		\begin{proof}                        
			$(a)$  
			Since $G$ has no induced $C_3$ and $C_5$, it is clear that $N(C)=A\cup B$.                                                                                       
			Thus, to prove (a), it is equivalent to prove that there is no vertex $v$ such that $d(v,C)=4$. 
			Suppose to the contrary that such a $v$ exists. 
			Let $v'\in N(C),v''\in D_1,v'''\in D_2$ such that $v'-v''-v'''-v$ is an induced $P_4$ in $G$ and we may assume that $v'$ is adjacent to $v_1$ and possibly $v_6$ . 
			Since $G$ has no comparable pair, $v$ has a neighbor $u$, which is not adjacent to $v''$. 
			Then $u-v-v'''-v''-v'-v_1-v_2-v_3-v_4-v_5$ is an induced $P_{10}$.  
			
			$(b)$
			Suppose to the contrary that there is a vertex $x\in D''_1$ such that $x$ has a neighbor $a\in A$. 
			By the definition of $D''_1$, $x$ has a neighbor $d$ in $D_2$. 
			Without loss of generality, we may assume that $a$ is adjacent to $v_1$. 
			Since $G$ has no comparable pair, $d$ has a neighbor $d'$ not adjacent to $a$. 
			Then $d'-d-x-a-v_1-v_2-v_3-v_4-v_5-v_6$ is an induced $P_{10}$. 
			
			$(c)$
			Suppose to the contrary that there is an edge $d-d'$ in $D_2$. 
			Let $v$ be a neighbor of $d$ in $D_1$, $v'$ be a neighbor of $v$ in $V(C)$. 
			By symmetry, we may assume that $v'$ is adjacent to $v_1$ but $v_3$. 
			Since $G$ has no comparable pair, $d'$ has a neighbor $d''$ not adjacent to $v$. 
			Since $d''-d'-d-v-v'-v_1-v_2-v_3-v_4-v_5$ is not an induced $P_{10}$ in $G$, $d''$ is adjacent to $v'$, and so, $v'-d''-d'-d-v-v'$ is an induced $C_5$. 
			
			$(d)$
			Suppose to the contrary that there is an edge $x-x'$ such that $x\in D''_1,x'\in D_1$. 
			Let $d$ be a neighbor of $x$ in $D_2$, $v$ be a neighbor of $x'$ in $N(C)$. 
			By symmetry, we may assume that $v$ is adjacent to $v_1$ but $v_3$. 
			Since $G$ has no comparable pair, $d$ has a neighbor $d'$ not adjacent to $x'$. 
			Since $d'-d-x-x'-v-v_1-v_2-v_3-v_4-v_5$ is not an induced $P_{10}$ in $G$, $d'$ is adjacent to $v$, and so, $d'-d-x-x'-v-d'$ is an induced $C_5$. 
		\end{proof}
		\begin{claim}\label{X_1 neighbor in B}
			For every $x\in D_1$, there is an index $i\in [7]$ such that $N(x)\cap B\subseteq B_{i-1}\cup B_{i+1}$. 
		\end{claim}
		\begin{proof} 
			Assume that $x$ has a neighbor $b$ in $B_{j}$, where $j\in [7]$. 
			If $x$ also has a neighbor in $B_{j\pm 1}$, say $b'\in B_{j+1}$, then $x-b-v_{j+1}-c_{j}-b-x$ is an induced $C_5$. 
			If $x$ also has a neighbor in $B_{j\pm 3}$, say $b''\in B_{j+3}$, then $x-b-v_{i+1}-v_{i+2}-b''-x$ is an induced $C_5$. 
			So, if $x$ has some neighbors in $B\setminus B_j$, they must in $B_{j\pm 2}$. 
			And also, by replacing $j$ with $j+2$, either $N(x)\cap B_{j+2}=\emptyset$ or $N(x)\cap B_{j-2}=\emptyset$. 
			Thus, we prove Claim \ref{X_1 neighbor in B} by setting $i=j+1$ or $i=j-1$. 
		\end{proof}
		\vspace{0.2cm}
		For a vertex $x$ in $D_1''$ such that $N(x)\cap B\subseteq B_{i-1}\cup B_{i+1}$ for some $i\in [7]$, define a set $T(x)$ be the {\em sign list} of $x$ such that
		$$T(x)=
		\begin{cases}
			\{i\}& \text{if }N(x)\cap B_{i-1}\neq\emptyset \text{ and } N(x)\cap B_{i+1}\neq\emptyset;\\ 
			\{i-2,i\}& \text{if }N(x)\cap B_{i+1}=\emptyset;\\
			\{i,i+2\}& \text{if }N(x)\cap B_{i-1}=\emptyset.
		\end{cases}
		$$
		By Claim \ref{X_1 neighbor in B}, $T(x)$ is well-defined. 
		We call a vertex $x\in D_1''$ {\em good} if $|T(x)|=1$, and {\em bad} otherwise
	. 
		
		\begin{claim}\label{bad to good}
			Let $d\in D_2$ with two neighbors $x_1,x_2$ in $D''_1$ such that $x_1$ has a neighbor in $B_i$ and $x_2$ has a neighbor in $B_{i+3}$ for some $i$. 
			Then $T(x_1)=\{i+1\}$ and $T(x_2)=\{i+2\}$. 
		\end{claim}
		\begin{proof}
			Without loss of generality, we may assume that $i=7$. 
			Let $b_7\in B_7,b_3\in B_3$ such that $b_7x_1,b_3x_2\in E(G)$. 
			Note that $v_7-v_6-b_7-x_1-d-x_2-b_3-v_4-v_3$ is an induced $P_9$. 
			Since $G$ has no comparable pair, let $w$ be a neighbor of $v_7$ that is not adjacent to $b_7$. 
			Since $w-v_7-v_6-b_7-x_1-d-x_2-b_3-v_4-v_3$ is not an induced $P_{10}$, $wx_2\in E(G)$. 
			By Claim \ref{basic}(b), $w\in B_1$, and hence $T(x_2)=\{2\}$. 
			By a similar analysis, $T(x_1)=\{1\}$. 
			This completes the proof of Claim \ref{bad to good}. 
		\end{proof}
		By Claim \ref{bad to good}, if a vertex $d$ in $D_2$ has a good neighbor $x$ in $D_1''$ with $T(x)=\{i\}$, then each good neighbor of $d$ in $D_1''$ has sign list either $\{i-1\}$ or $\{i+1\}$. 
		Thus, for every vertex $d\in D_2$, there are no three good neighbors of $d$ in $D''_1$ with pairwise different sign lists. 
		In particular, there is an index $i\in [7]$ such that for each good neighbor of $d$, its sign list is either $\{i\}$ or $\{i+1\}$. 
		
		\begin{claim}\label{nb}
			If $d$ has at least one good neighbor or has two bad neighbors with different sign lists, then all its bad neighbors has no neighbors in $D_2$ except $d$. 
			Moreover, 
			\begin{itemize}
				\item[(a)] $d$ has no bad neighbors if $d$ has two good neighbors with different sign lists; 
				\item[(b)] all bad neighbors of $d$ have the same sign list if $d$ has a good neighbor $x_1$ and all good neighbors of $d$ have the same sign list; 
				\item[(c)] there is an index $i\in 7$ such that for each bad neighbor of $d$, it has list either $\{i-1,i+1\}$ or $\{i,i+2\}$. 
			\end{itemize}
		\end{claim}
		\begin{proof}
			Assume first that $d$ has two good neighbors $x_1,x_2$ in $D_1''$ with $T(x_1)=\{i\}$, $T(x_2)=\{i+1\}$. 
			We then assume that $x$ be a bad neighbor of $d$. 
			By Claim \ref{bad to good}, either $T(x)=\{i-1,i+1\}$ or $T(x)=\{i,i+2\}$, say $T(x)=\{i,i+2\}$ by symmetry. 
			Let $b_{i-1}$ be a neighbor of $x_1$ in $B_{i-1}$, $b_i$ be a neighbor of $x_2$ in $B_{i}$. 
			Then $x-d-x_1-b_{i-1}-v_{i-2}-v_{i-3}-v_{i+3}-v_{i+2}-v_{i+1}-b_i$ is an induced $P_{10}$. 
			Thus, we have that $d$ has no bad neighbors if $d$ has two good neighbors in $D_1''$ with different sign lists.

			We then assume that	$d$ has a good neighbor $x_1$ and all good neighbors of $d$ have the same sign list $\{i\}$.  
			Let $x$ be a bad neighbor of $d$. 
			By Claim \ref{bad to good} again, $T(x)$ is $\{i-2,i\}$, $\{i-1,i+1\}$ or $\{i,i+2\}$. 
			If $T(x)=\{i-2,i\}$ or $T(x)=\{i,i+2\}$, say $T(x)=\{i,i+2\}$ by symmetry, then let $b$ be a neighbor of $x_1$ in $B_{i-1}$ and let $b'$ be a neighbor of $x$ in $B_{i+1}$. 
			Since $d-x_1-b-v_{i-2}-v_{i-3}-v_{i+3}-v_{i+2}-b'-x-d$ is not an induced $C_9$, $x_1b'\in E(G)$. 
			Hence, $N(x)\cap B\subseteq N(x_1)$ by the arbitrariness of $b'$. 
			Since $G$ has no comparable pair, $x$ has a neighbor $d'$ in $D_2$ not adjacent to $x_1$. 
			Then $d'-x-d-x_1-b-v_{i-2}-v_{i-3}-v_{i+3}-v_{i+2}-v_{i+1}$ is an induced $P_{10}$.  
			Hence, $T(x)=\{i-1,i+1\}$. 
			Let $x_1$ be a good neighbor of $d$ with a neighbor $b_{i-1}\in B_{i-1}$, $b_i$ be a neighbor of $x$ in $B_i$. 
			If $x$ has another neighbor $d'$ in $D_2$, then since $d'-x-d-x_1-b_{i-1}-v_{i-2}-v_{i-3}-v_{i+3}-v_{i+2}-v_{i+1}$ is not an induced $P_{10}$, we have $d'x_1\in E(G)$. 
			Hence, by the arbitrariness of $x$ and $x_1$, $d,d'$ have the same neighborhood by the arbitrariness of $x$ and $x_1$, which is a contradiction since $G$ has no comparable pair. 
			
			We finally assume that $d$ has two bad neighbors with different sign lists and has no good neighbors in $D_1''$. 
			Let $x_1$ be a bad neighbor of $d$ with $T(x_1)=\{i-1,i+1\}$ and $x_2$ be another bad neighbor of $d$ with $T(x_2)\neq \{i-1,i+1\}$. 
			Let $b$ be a neighbor of $x_1$ in $B$, $b'$ be a neighbor of $x_2$ in $B$.  
			By Claim \ref{bad to good}, $T(x_2)$ is $\{i-3,i-1\}$, $\{i-2,i\}$, $\{i,i+2\}$ or $\{i+1,i+3\}$. 
			If $T(x_2)=\{i-3,i-1\}$ or $T(x_2)=\{i+1,i+3\}$, say $T(x_2)=\{i+1,i+3\}$ by symmetry, then $d-x_1-b-v_{i-1}-v_{i-2}-v_{i-3}-v_{i+3}-b'-x_2-d$ is an induced $C_9$. 
			So either $T(x_2)=\{i-2,i\}$ or $T(x_2)=\{i,i+2\}$, say $T(x_2)=\{i,i+2\}$ by symmetry. 
			If $x$ has another neighbor $d'$ in $D_2$, then since $d'-x_1-d-x_2-b'-v_{i+2}-v_{i+3}-v_{i-3}-v_{i-2}-v_{i-1}$ is not an induced $P_{10}$, we have $d'x_1\in E(G)$. 
			Hence, by the arbitrariness of $x$ and $x_1$, $d$ and $d'$ have the same neighborhood, which is a contradiction since $G$ has no comparable pair. 
			This completes the proof of Claim \ref{nb}. 
		\end{proof}
		
		Let $G_1'$ be the remaining graph by deleting $D_1''\cup D_2$ from $G$ and let $G_1$ be the graph obtained by Lemma \ref{comparable} from $G_1'$. 
		So $G_1$ has no comparable pair and can be obtained in $O(|V(G)|+|E(G)|)$ time. 
		
		\begin{claim}\label{P_4 neighbor}
			For every induced $P_3$ in $G_1\setminus N[C]$, the ends of the path have the same neighborhood in $A$ and for every induced $P_4$ in $G[D_1]$, the first vertex and the third vertex of this path have the same neighborhood in $B$.  
		\end{claim}
		\begin{proof}
			We first assume that $a-b-c$ is an induced $P_3$ in $G_1\setminus N[C]$ such that $a$ has a neighbor $v\in A$ adjacent to $v_i$. 
			Since $v_{i+2}-v_{i+3}-v_{i-3}-v_{i-2}-v_{i-1}-v_i-v-a-b-c$ is not an induced $P_{10}$, $vc\in E(G)$. 
			By symmetry and the arbitrariness of $v$, $N(a)\cap A= N(c)\cap A$. 
			We then assume that $a-b-c-d$ is an induced $P_4$ in $G_1\setminus N[C]$ such that $a$ has a neighbor $u\in B$ adjacent to $v_i$ but $v_{i-2}$. 
			Since $v_{i+3}-v_{i-3}-v_{i-2}-v_{i-1}-v_i-u-a-b-c-d$ is not an induced $P_{10}$, $uc\in E(G)$. 
			By the arbitrariness of $u$, $N(a)\cap B\subseteq N(c)\cap B$. 
			On the other hand, let $u'\in B$ be a neighbor of $c$ adjacent to $v_i$ but $v_{i-2}$. 
			Since $G_1$ has no comparable pair, $a$ has a neighbor $a'$ not adjacent to $c$. 
			Hence, $a'\in G_1\setminus N[C]$ and $a'-a-b-c$ is an induced $P_4$ in $G_1\setminus N[C]$. 
			Thus, $N(c)\cap B\subseteq N(a)\cap B$.
			So, $N(a)\cap B= N(c)\cap B$.
			This completes the proof of Claim \ref{P_4 neighbor}.
		\end{proof}
		
		\begin{claim}\label{bipartite}
			$G_1\setminus N[C]$ is a bipartite graph. 
		\end{claim}
		\begin{proof}
			Suppose to the contrary that $G_1\setminus N[C]$ is not a bipartite graph. 
			Let $C'$ be an odd hole in $G_1\setminus N[C]$. 
			By applying Claim \ref{P_4 neighbor} to every $P_4$ in $C'$, we generate a $C_3$ in $G_1$. 
			This completes the proof of Claim \ref{bipartite}.  
		\end{proof}
		\begin{claim}\label{p_4 reduction}
			For any non-complete bipartite component $H:=\{H_1,H_2\}$ of $G_1\setminus N[C]$, all vertices in $H_i$ have the same neighborhood in $N(C)$, where $i\in \{1,2\}$.   
		\end{claim}
		\begin{proof}
			By symmetry, we only need to prove this Claim for $i=1$. 
			Since $H$ is not complete, there is an induced $P_4:=a-b-c-d$ in $H$ with $a,c\in H_1$. 
			Then by the Claim \ref{P_4 neighbor}, $a$ and $c$ have the same neighborhood in $N(C)$. 
			Suppose to the contrary that $v$ is a vertex in $H_1\setminus \{a,c\}$ such that $a$ and $v$ have different neighborhoods in $N(C)$. 
			It is clearly that $v$ and $a$ have the same neighborhood in $N(C)$ if $d(v,a)\geq 4$ or $d(v,c)\geq 4$. 
			So, we may assume that both $d(v,a)$ and $d(v,c)$ are equal to 2. 
			Let $v'\in H_2$ be a common neighbor of $a$ and $v$. 
			Suppose first that $v$ is not adjacent to $b$. 
			Then $v-v'-a-b$ is an induced $P_4$ in $H$, and hence, $v$ and $a$ have the same neighborhood in $N(C)$ by Claim \ref{P_4 neighbor}. 
			Thus, $vb\in E(G)$. 
			We then suppose that $v$ is not adjacent to $d$. 
			Then $v-b-c-d$ is an induced $P_4$ in $H$ and hence, $v$ and $c$ have the same neighborhood in $N(C)$ by Claim \ref{P_4 neighbor}. 
			Thus, $v$ is adjacent to both $b$ and $d$. 
			Then $a-b-v-d$ is an induced $P_4$ in $H$. 
			By Claim \ref{P_4 neighbor}, $v$ and $a$ have the same neighborhood in $N(C)$ . 
			This completes the proof of Claim \ref{p_4 reduction}. 
		\end{proof}
		
		By Claim \ref{basic}, we have $V(G)=V(C)\cup A\cup B\cup D'_1\cup D_1''\cup D_2$. 
		Let $G_2$ be the remaining graph from $G_1$ by deleting all vertices in $H$ except an edge for each non-complete bipartite component $H:=\{H_1,H_2\}$ in $G_1\setminus N[C]$. 
		And let $G_3$ be the graph by Lemma \ref{comparable} from $G_2$. 
		So, for every $v\in V(G_3)$, $d(v,C)\leq 2$, $G_3$ has no non-complete bipartite component in $G_3\setminus N[C]$ and $G_3$ has no comparable pair. 
		Moreover, it follows that $G_3$ can be constructed in $O(|V(G)|+|E(G)|)$ time. 
		We then prove that for $G$ is 3-colorable if and only if $G_3$ is 3-colorable. 
		Since $G_3$ is an induced subgraph of $G$, $G_3$ is 3-colorable if $G$ is 3-colorable. 
		We then assume that $G_3$ is 3-colorable and let $c_3$ be a 3-coloring of $G_3$. 
		By Lemma \ref{comparable}, we obtain a 3-coloring $c_2$ of $G_2$ in $O(|V(G)|+|E(G)|)$ time. 
		Define $c_1$ by setting $c_1(v)=c_2(u)$, for every pair of $u,v$ such that $v\in G_1\setminus G_2$, $u\in G_2$ and $u,v$ in the same part of a non-complete bipartite component in $G_1\setminus N[C]$ and leaving $c_1(w)=c_2(w)$ for all $w\in G_2$. 
		Then by the definition of $G_2$, for each part of each non-complete bipartite component in $G_2\setminus N[C]$, only one vertex left in $G_2$. 
		So $c_1$ is a well-defined 3-coloring of $G_1$ and can be obtained in $O(|V(G)|+|E(G)|)$ time. 
		Let $c_1'$ be the 3-coloring of $G_1'$ obtained by Lemma \ref{comparable}.   
		Define $c$ by setting $c(v)=c_1'(v_{T(v)})$ for every good vertex $v\in D''_1$, $c(u)=c_1'(v_{i})$ if $u$ is a bad vertex in $D''_1$ and $T(u)=\{i,i+2\}$, $c(d)$ be the color different from its neighbors assigned in $c_1'$ restricted in $G_1'\setminus D_2$ for all $d\in D_2$ and leaving $c(w)=c_1'(w)$ for all $w\in G_1'$. 
		By Claim \ref{nb}, $c$ is a well-defined 3-coloring of $G$ and could be obtained from $c_1'$ in $O(|V(G)|+|E(G)|)$ time. 
		 
		By applying this step for every induced odd cycle in $G$, we obtain $G'$ as the remaining graph in $O((|V(G)|^7)(|V(G)|+|E(G)|))$ time.
		This completes the proof of Lemma \ref{cleaned}.
	\end{proof}

	\section{Non-trivial Component in $G\setminus N[C]$}
	Let $G$ be a cleaned graph in $\mathcal{G}_{10,7}$ with an induced odd cycle $C:=v_1-v_2-v_3-v_4-v_5-v_6-v_7-v_1$ and $(G,L,Z)$ be a restriction of $(G,L_0,\emptyset)$, where $L_0$ is the palette such that $L_0(v)=\{1,2,3\}$ for all $v\in V(G)$. 
	We call a vertex $v\in G\setminus N[C]$ {\em reducible} for $L$ if $|L(v)|=3$ and there is a color $i$ such that $i\notin L(u)$ for every $u\in N(v)$. 
	Similarly, we call a complete bipartite component $K=(U_1,U_2)$ of $G\setminus N[C]$ {\em reducible} for $L$ if  there are two different colors $i,j$ with $i\notin L(v)$ for every $v\in N(U_1)\cap N(C)$, $i\in L(a)$ for every $a\in N(U_1)$, $j\notin L(u)$ for every $u\in N(U_2)\cap N(C)$ and $j\in L(b)$ for every $b\in U_2$.
	We call a palette $L$ {\em non-reducible} if there are no reducible components and vertices in $G\setminus N[C]$ for $L$. 
	It follows that one can reduce $L$ to obtain a non-reducible updated subpalette $L'$ of $L$ in $O(|V(G)|+|E(G)|)$ time. 
	
	In this section, we assume that $L$ is non-reducible updated and every vertex has list length at most 2 except those in a non-trivial component in $G\setminus N[C]$. 
	Our aim is to obtain a set $\mathcal{R}$ of restrictions of $(G,L,Z)$ in polynomial time such that for every element of $\mathcal{R}$, no vertex has  list length 3. 
	Formally, we prove the following Lemma.  
	
	\begin{lemma}\label{Y}
		Let $G$ be cleaned graph in $\mathcal{G}_{10,7}$  with an induced odd cycle $C:=v_1-v_2-v_3-v_4-v_5-v_6-v_7-v_1$. 
		Let $(G,L,Z)$ be a restriction of $(G,L_0,\emptyset)$ such that $L$ is non-reducible updated and  every vertex in $V(C)$ has list length 1.  
		Then there is a set $\mathcal{R}$ of restrictions of $(G,L,Z)$ with size in $O(1)$ such that 
		\begin{itemize}
			\item[(a)] $(G,L,Z)$ is colorable if and only if $\mathcal{R}$ is colorable; 
			\item[(b)] for each element $(G',L',Z')$ of $\mathcal{R}$ and for every $v$ in a non-trivial component of $G\setminus N[C]$, $|L'(v)|\leq 2$; 
			\item[(c)] $\mathcal{R}$ can be constructed in $O(|V(G)|+|E(G)|)$ time;
			\item[(d)] any coloring of $\mathcal{R}$ can be extended to a coloring of $(G,L,Z)$ in $O(|V(G)|)$ time.   
		\end{itemize} 
	\end{lemma}
	To give an understanding of the structure of $G\setminus V(C)$, we first analyze the edges between the neighbors of an edge in $G\setminus N[C]$. 
	\begin{claim}\label{edges in nieghbors of D}
		Let $d_1d_2$ be an edge in $G\setminus N[C]$. 
		If $d_1$ has a neighbor $a_i$ in $A_i$, then all neighbors of $d_2$ in $A_{i\pm 1}\cup A_{i\pm 3}\cup B_{i\pm 2}$ are complete to $a_i$. 
		If $d_1$ has a neighbor $b_i$ in $B_i$, then all neighbors of $d_2$ in $A_{i\pm 2}\cup B_{i\pm 1}$ are complete to $b_i$. 
	\end{claim}
	The proof of Claim \ref{edges in nieghbors of D} immediately follows from the fact that $G$ has no induced $C_3$, $C_5$, $C_9$ and $P_{10}$. 
	
	Let $\mathcal{K}_1$ be the set of components in $G\setminus N[C]$ that have order 2.
	We partition $\mathcal{K}_1$  as follows:   
	\begin{itemize}
		\item $W_1=\{d_1d_2\in \mathcal{K}_1:N(d_1)\cap A_i\neq \emptyset, N(d_2)\cap A_{i+2}\neq \emptyset\text{ for some }i\in [7]\}$; 
		\item $W_2=\{d_1d_2\in \mathcal{K}_1\setminus W_1:N(d_1)\cap B_i\neq \emptyset, N(d_2)\cap B_{i+3}\neq \emptyset\text{ for some }i\in [7]\}$; 
		\item $W_3=\{d_1d_2\in \mathcal{K}_1\setminus (W_1\cup W_2):N(d_1)\cap B_i\neq \emptyset, N(d_2)\cap A_{i\pm 3}\neq \emptyset\text{ for some }i\in [7]\}$; 
		\item $W_4=\{d_1d_2\in \mathcal{K}_1\setminus (W_1\cup W_2\cup W_3):N(d_1)\cap B_i\neq \emptyset, N(d_2)\cap B_{i+1}\neq \emptyset\text{ for some }i\in [7]\}$; 
		\item $W_5=\mathcal{K}_1\setminus (W_1\cup W_2\cup W_3\cup W_4)$.   
	\end{itemize}

	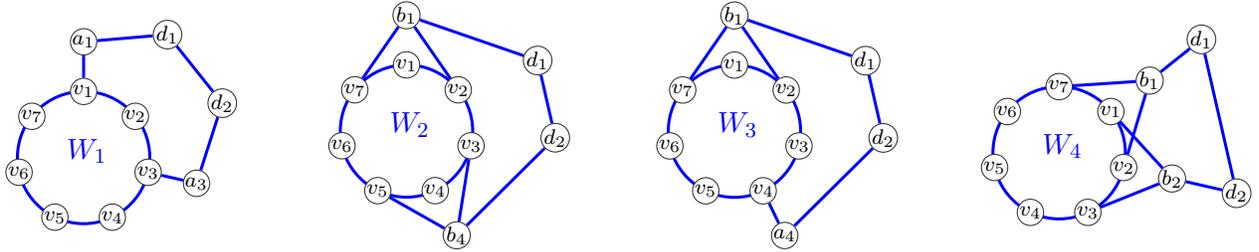
\begin{figure}[htbp]
		\centering		
		\begin{minipage}{.2\linewidth}
			\centering
			\begin{tikzpicture}[scale=.44]
				\draw[fill=none, line width=0.4mm, blue](0,0) circle (2) {};
				\draw[rotate=0] (0,2) node[circle, draw=black!80, inner sep=0mm, minimum size=3.5mm, fill=white] (a_{1}){\scriptsize $v_1$};
				\draw[rotate=-360/7] (0,2) node[circle, draw=black!80, inner sep=0mm, minimum size=3.5mm, fill=white] (a_{2}){\scriptsize $v_2$};
				\draw[rotate=-360/7*2] (0,2) node[circle, draw=black!80, inner sep=0mm, minimum size=3.5mm, fill=white] (a_{3}){\scriptsize $v_3$};
				\draw[rotate=-360/7*3] (0,2) node[circle, draw=black!80, inner sep=0mm, minimum size=3.5mm, fill=white] (a_{4}){\scriptsize $v_4$};
				\draw[rotate=-360/7*4] (0,2) node[circle, draw=black!80, inner sep=0mm, minimum size=3.5mm, fill=white] (a_{5}){\scriptsize $v_5$};
				\draw[rotate=-360/7*5] (0,2) node[circle, draw=black!80, inner sep=0mm, minimum size=3.5mm, fill=white] (a_{6}){\scriptsize $v_6$};
				\draw[rotate=-360/7*6] (0,2) node[circle, draw=black!80, inner sep=0mm, minimum size=3.5mm, fill=white] (a_{7}){\scriptsize $v_7$};
				
				
				\draw[rotate=0] (0,3.5) node[circle, draw=black!80, inner sep=0mm, minimum size=3.5mm, fill=white] (x_{1}){\scriptsize $a_1$};
				\draw[rotate=-240/7] (0,4.5) node[circle, draw=black!80, inner sep=0mm, minimum size=3.5mm, fill=white] (x_{2}){\scriptsize $d_1$};
				\draw[rotate=-240/7*2] (0,4.5) node[circle, draw=black!80, inner sep=0mm, minimum size=3.5mm, fill=white] (x_{3}){\scriptsize $d_2$};
				\draw[rotate=-240/7*3] (0,3.5) node[circle, draw=black!80, inner sep=0mm, minimum size=3.5mm, fill=white] (x_{4}){\scriptsize $a_3$};
				
				
				\draw[line width=0.4mm, blue] (a_{1}) to  (x_{1});
				\draw[line width=0.4mm, blue] (a_{3}) to  (x_{4});
				\draw[line width=0.4mm, blue] (x_{1}) to  (x_{2});
				\draw[line width=0.4mm, blue] (x_{2}) to  (x_{3});
				\draw[line width=0.4mm, blue] (x_{3}) to  (x_{4});
				
				\node[text=blue] at (.1,.2) {$W_1$};

			\end{tikzpicture}
			
		\end{minipage}
		\centering
		\hfill
		\begin{minipage}{.2\linewidth}

			\centering
			\begin{tikzpicture}[scale=.44]
				\draw[fill=none, line width=0.4mm, blue](0,0) circle (2) {};
				\draw[rotate=0] (0,2) node[circle, draw=black!80, inner sep=0mm, minimum size=3.5mm, fill=white] (a_{1}){\scriptsize $v_1$};
				\draw[rotate=-360/7] (0,2) node[circle, draw=black!80, inner sep=0mm, minimum size=3.5mm, fill=white] (a_{2}){\scriptsize $v_2$};
				\draw[rotate=-360/7*2] (0,2) node[circle, draw=black!80, inner sep=0mm, minimum size=3.5mm, fill=white] (a_{3}){\scriptsize $v_3$};
				\draw[rotate=-360/7*3] (0,2) node[circle, draw=black!80, inner sep=0mm, minimum size=3.5mm, fill=white] (a_{4}){\scriptsize $v_4$};
				\draw[rotate=-360/7*4] (0,2) node[circle, draw=black!80, inner sep=0mm, minimum size=3.5mm, fill=white] (a_{5}){\scriptsize $v_5$};
				\draw[rotate=-360/7*5] (0,2) node[circle, draw=black!80, inner sep=0mm, minimum size=3.5mm, fill=white] (a_{6}){\scriptsize $v_6$};
				\draw[rotate=-360/7*6] (0,2) node[circle, draw=black!80, inner sep=0mm, minimum size=3.5mm, fill=white] (a_{7}){\scriptsize $v_7$};
				
				
				\draw[rotate=0] (0,3.5) node[circle, draw=black!80, inner sep=0mm, minimum size=3.5mm, fill=white] (x_{1}){\scriptsize $b_1$};
				\draw[rotate=-360*6/35] (0,4.5) node[circle, draw=black!80, inner sep=0mm, minimum size=3.5mm, fill=white] (x_{2}){\scriptsize $d_1$};
				\draw[rotate=-360*9/35] (0,4.5) node[circle, draw=black!80, inner sep=0mm, minimum size=3.5mm, fill=white] (x_{3}){\scriptsize $d_2$};
				\draw[rotate=-360/7*3] (0,3.5) node[circle, draw=black!80, inner sep=0mm, minimum size=3.5mm, fill=white] (x_{4}){\scriptsize $b_4$};
				
				
				\draw[line width=0.4mm, blue] (a_{7}) to  (x_{1});
				\draw[line width=0.4mm, blue] (a_{2}) to  (x_{1});
				\draw[line width=0.4mm, blue] (a_{3}) to  (x_{4});
				\draw[line width=0.4mm, blue] (a_{5}) to  (x_{4});
				\draw[line width=0.4mm, blue] (x_{1}) to  (x_{2});
				\draw[line width=0.4mm, blue] (x_{2}) to  (x_{3});
				\draw[line width=0.4mm, blue] (x_{3}) to  (x_{4});
				
				\node[text=blue] at (.1,.2) {$W_2$};

			\end{tikzpicture}
			
		\end{minipage}
		\centering
		\hfill
		\begin{minipage}{.2\linewidth}

			\centering
			\begin{tikzpicture}[scale=.44]
				\draw[fill=none, line width=0.4mm, blue](0,0) circle (2) {};
				\draw[rotate=0] (0,2) node[circle, draw=black!80, inner sep=0mm, minimum size=3.5mm, fill=white] (a_{1}){\scriptsize $v_1$};
				\draw[rotate=-360/7] (0,2) node[circle, draw=black!80, inner sep=0mm, minimum size=3.5mm, fill=white] (a_{2}){\scriptsize $v_2$};
				\draw[rotate=-360/7*2] (0,2) node[circle, draw=black!80, inner sep=0mm, minimum size=3.5mm, fill=white] (a_{3}){\scriptsize $v_3$};
				\draw[rotate=-360/7*3] (0,2) node[circle, draw=black!80, inner sep=0mm, minimum size=3.5mm, fill=white] (a_{4}){\scriptsize $v_4$};
				\draw[rotate=-360/7*4] (0,2) node[circle, draw=black!80, inner sep=0mm, minimum size=3.5mm, fill=white] (a_{5}){\scriptsize $v_5$};
				\draw[rotate=-360/7*5] (0,2) node[circle, draw=black!80, inner sep=0mm, minimum size=3.5mm, fill=white] (a_{6}){\scriptsize $v_6$};
				\draw[rotate=-360/7*6] (0,2) node[circle, draw=black!80, inner sep=0mm, minimum size=3.5mm, fill=white] (a_{7}){\scriptsize $v_7$};
				
				
				\draw[rotate=0] (0,3.5) node[circle, draw=black!80, inner sep=0mm, minimum size=3.5mm, fill=white] (x_{1}){\scriptsize $b_1$};
				\draw[rotate=-360*6/35] (0,4.5) node[circle, draw=black!80, inner sep=0mm, minimum size=3.5mm, fill=white] (x_{2}){\scriptsize $d_1$};
				\draw[rotate=-360*9/35] (0,4.5) node[circle, draw=black!80, inner sep=0mm, minimum size=3.5mm, fill=white] (x_{3}){\scriptsize $d_2$};
				\draw[rotate=-360/7*3] (0,3.5) node[circle, draw=black!80, inner sep=0mm, minimum size=3.5mm, fill=white] (x_{4}){\scriptsize $a_4$};
				
				
				\draw[line width=0.4mm, blue] (a_{7}) to  (x_{1});
				\draw[line width=0.4mm, blue] (a_{2}) to  (x_{1});
				\draw[line width=0.4mm, blue] (a_{4}) to  (x_{4});
				\draw[line width=0.4mm, blue] (x_{1}) to  (x_{2});
				\draw[line width=0.4mm, blue] (x_{2}) to  (x_{3});
				\draw[line width=0.4mm, blue] (x_{3}) to  (x_{4});
				
				\node[text=blue] at (.1,.2) {$W_3$};

			\end{tikzpicture}
			
		\end{minipage}
		\centering
		\hfill
		\begin{minipage}{.2\linewidth}

			\centering
			\begin{tikzpicture}[scale=.44]
				\draw[fill=none, line width=0.4mm, blue](0,0) circle (2) {};
				\draw[rotate=-360/7] (0,2) node[circle, draw=black!80, inner sep=0mm, minimum size=3.5mm, fill=white] (a_{1}){\scriptsize $v_1$};
				\draw[rotate=-360/7*2] (0,2) node[circle, draw=black!80, inner sep=0mm, minimum size=3.5mm, fill=white] (a_{2}){\scriptsize $v_2$};
				\draw[rotate=-360/7*3] (0,2) node[circle, draw=black!80, inner sep=0mm, minimum size=3.5mm, fill=white] (a_{3}){\scriptsize $v_3$};
				\draw[rotate=-360/7*4] (0,2) node[circle, draw=black!80, inner sep=0mm, minimum size=3.5mm, fill=white] (a_{4}){\scriptsize $v_4$};
				\draw[rotate=-360/7*5] (0,2) node[circle, draw=black!80, inner sep=0mm, minimum size=3.5mm, fill=white] (a_{5}){\scriptsize $v_5$};
				\draw[rotate=-360/7*6] (0,2) node[circle, draw=black!80, inner sep=0mm, minimum size=3.5mm, fill=white] (a_{6}){\scriptsize $v_6$};
				\draw[rotate=-360/7*7] (0,2) node[circle, draw=black!80, inner sep=0mm, minimum size=3.5mm, fill=white] (a_{7}){\scriptsize $v_7$};
				
				
				\draw[rotate=-360/7] (0,3.5) node[circle, draw=black!80, inner sep=0mm, minimum size=3.5mm, fill=white] (x_{1}){\scriptsize $b_1$};
				\draw[rotate=-360/7] (0,5.5) node[circle, draw=black!80, inner sep=0mm, minimum size=3.5mm, fill=white] (x_{2}){\scriptsize $d_1$};
				\draw[rotate=-360/7*2] (0,5.5) node[circle, draw=black!80, inner sep=0mm, minimum size=3.5mm, fill=white] (x_{3}){\scriptsize $d_2$};
				\draw[rotate=-360/7*2] (0,3.5) node[circle, draw=black!80, inner sep=0mm, minimum size=3.5mm, fill=white] (x_{4}){\scriptsize $b_2$};
				
				
				\draw[line width=0.4mm, blue] (a_{7}) to  (x_{1});
				\draw[line width=0.4mm, blue] (a_{2}) to  (x_{1});
				\draw[line width=0.4mm, blue] (a_{1}) to  (x_{4});
				\draw[line width=0.4mm, blue] (a_{3}) to  (x_{4});
				\draw[line width=0.4mm, blue] (x_{1}) to  (x_{2});
				\draw[line width=0.4mm, blue] (x_{2}) to  (x_{3});
				\draw[line width=0.4mm, blue] (x_{3}) to  (x_{4});
				
				\node[text=blue] at (.1,.2) {$W_4$};

			\end{tikzpicture}
			
		\end{minipage}
		\caption{Illustration of $\mathcal{K}_1$ when $i=1$.}
		\label{fig1}
	\end{figure}

	\begin{claim}\label{K_2}
		We claim that 
		\begin{itemize}
			\item[(a)] there is a set $T_1^1$ with $|T_1^1|\leq 7$ such that $T_1^1$ dominates all components in $W_1$; 
			\item[(b)] there is a set $T_1^2$ with $|T_1^2|\leq 28$ such that $T_1^2$ dominates all components in $W_2$;
			\item[(c)] there is a set $T_1^3$ with $|T_1^3|\leq 42$ such that $T_1^3$ dominates all components in $W_3$;
			\item[(d)] $W_4=\emptyset$; 
			\item[(e)] $W_5=\emptyset$.   
		\end{itemize}
		Moreover, all the sets $T_1^1$, $T_1^2$ and $T_1^3$ can be constructed in $O(|V(G)|+|E(G)|)$ time. 
	\end{claim}
	\begin{proof}
		(a) 
		Let $d_1d_2$ be an element in $W_1$ such that $d_1$ has a neighbor $a$ in $A_i$ and $d_2$ has a neighbor $a'$ in $A_{i+2}$.  
		If there is a vertex $a''$ in $A_i$ not adjacent to $d_1$, then $a''-v_i-a-d_1-d_2-a'-v_{i+2}-v_{i+3}-v_{i-3}-v_{i-2}$ is an induced $P_{10}$. 
		Hence, $d_1$ is complete to $A_i$ and by symmetry, $d_2$ is complete to $A_{i+2}$. 
		Let $T_1^1$ be a set consisting of an arbitrary vertex in $A_i$ for every $i\in[7]$.  
		Thus, $T_1^1$ dominates all components in $W_1$ and can be constructed in $O(|V(G)|+|E(G)|)$ time. 
		This completes the proof of (a). 
		
		(b) 
		For each $i\in [7]$, let $W_2^i$ be the set of the components in $W_2$ such that one of its vertices has a neighbor in $B_{i}$ and another has a neighbor in $B_{i+3}$. 
		Among all elements in $W_2^i$, let $d_1d_2$ be the component such that $N(d_1)\cap B_{i}$ is minimal and $d_1'd_2'$ be the component such that $N(d_2')\cap B_{i+3}$ is minimal. 
		Let $b_i\in B_i$ be a neighbor of $d_1$, $b_{i+3}\in B_{i+3}$ be a neighbor of $d_2$, $b_i'\in B_i$ be a neighbor of $d'_1$ and $b'_{i+3}\in B_{i+3}$ be a neighbor of $d_2'$. 
		Since $G$ has no comparable pair, there is a vertex $w_1$ adjacent to $v_{i+3}$ but $b_{i+3}$ and a vertex $w_2$ adjacent to $v_{i}$ but $b'_{i}$. 
		So $\{b_i,b'_{i+3},w_1,w_2\}$ can be constructed in $O(|V(G)|+|E(G)|)$ time. 
		We claim that $\{b_i,b'_{i+3},w_1,w_2\}$ dominates all components in $W_2^i$. 
		Suppose to the contrary that there is a vertex in a component $x_1x_2$, say $x_1$, in $W_2^i$ such that $x_1$ is not dominated by $\{b_i,b'_{i+3},w_1,w_2\}$. 
		By the minimality of $d_1$, $x_1$ has a neighbor $b_i''$ in $B_i$ not adjacent to $d_1$. 
		Since $x_2-x_1-b_i''-v_{i+1}-b_i-d_1-d_2-b_{i+3}-v_{i-3}-v_{i+3}$ is not an induced $P_{10}$, $x_2b_{i+3}\in E(G)$. 
		Since $x_1-b_i''-v_{i+1}-b_i-d_1-d_2-b_{i+3}-v_{i-3}-v_{i+3}-w_1$ is not an induced $P_{10}$, we have $w_1d_1\in E(G)$. 
		Then $v_{i-2}-v_{i-1}-b_i''-x_1-x_2-b_{i+3}-v_{i+2}-v_{i+3}-w_1-d_1$ is an induced $P_{10}$.  
		Thus, $\{b_i,b'_{i+3},w_1,w_2\}$ dominates all components of $W_2^i$. 
		For each $i\in [7]$, we obtain such four vertices and let $T_1^2$ be the union of these vertices. 
		It follows that $|T_1^2|\leq 28$, $T_1^2$ dominates all components of $W_2$ and could be constructed in $O(|V(G)|+|E(G)|)$ time.  
		This completes the proof of (b). 
		
		(c) 
		For each $i\in [7]$, let $W_3^{i,+}$ (or $W_3^{i,-}$, respectively) be the sets of the components in $W_3$ such that one of its vertices has a neighbor in $B_{i}$ and another has a neighbor in $A_{i+3}$ (or $A_{i-3}$, respectively).  
		Among all elements in $W_3^{i,+}$, let $d_1-d_2$ be the component such that $N(d_1)\cap B_{i}$ is minimal and $d_1'-d_2'$ be the component such that $N(d_2')\cap A_{i+3}$ is minimal. 
		Let $b_i\in B_i$ be a neighbor of $d_1$, $a_{i+3}\in A_{i+3}$ be a neighbor of $d_2$ and $a_{i+3}'\in A_{i+3}$ be a neighbor of $d_2'$. 
		Since $G$ has no comparable pair, there is a vertex $w_1$ adjacent to $v_i$ but $b_i$. 
		Thus, $\{b_i,a'_{i+3},w_1\}$ could be obtained in $O(|V(G)|+|E(G)|)$ time. 
		We claim that $\{b_i,a'_{i+3},w_1\}$ dominates all elements in $W_3^{i,+}$. 
		Suppose to the contrary that there is a vertex in a component $x_1-x_2$ in $W_3^{i,+}$ that is not dominated. 
		Assume first that $x_1$ is not dominated. 
		By the minimality of $d_1$, $x_1$ has a neighbor $b_i'$ in $B_i$ not adjacent to $d_1$. 
		Then $x_1-b_i'-v_{i+1}-b_i-d_1-d_2-a_{i+3}-v_{i+3}-v_{i-3}-v_{i-2}$ is an induced $P_{10}$. 
		Thus, we may assume that $x_2$ is not dominated and both $d_1'$ and $x_1$ are adjacent to $b_i$. 
		By the minimality of $d_2'$, $x_2$ has a neighbor $a''_{i+3}$ in $A_{i+3}$ that is not adjacent to $d_2'$. 
		Note that if $w_1\in B$, then $w_1$ is not adjacent to $d_2'$ and $x_2$. 
		Since $w_1-v_i-v_{i+1}-b_i-d_1'-d_2'-a_{i+3}'-v_{i+3}-a''_{i+3}-x_2'$ is not an induced $P_{10}$, $w_1\notin B_{i-1}$. 
		Thus, either $w_1\in A_i$ or $w_1\in B_{i+1}$. 
		If $w_1\in B_{i+1}$, then since $w_1-v_i-v_{i+1}-b_i-d_1'-d_2'-a_{i+3}'-v_{i+3}-v_{i-3}-v_{i-2}$ is not an induced $P_{10}$, $w_1a'_{i+3}\in E(G)$. 
		By symmetry, $w_1a''_{i+3}\in E(G)$, and hence, $x_2'-a''_{i+3}-w_1-a'_{i+3}-d_2'-d_1'-b_i-v_{i-1}-v_{i-2}-v_{i-3}$ is an induced $P_{10}$.  
		So, we may assume that $w_1\in A_i$. 
		Since $w_1-v_i-v_{i+1}-b_i-x_1-x_2-a_{i+3}''-v_{i+3}-v_{i-3}-v_{i-2}$ is not an induced $P_{10}$, $w_1a''_{i+3}\in E(G)$. 
		Since $w_1-v_i-v_{i+1}-b_i-d_1'-d_2'-a_{i+3}'-v_{i+3}-v_{i-3}-v_{i-2}$ is not an induced $P_{10}$, $w_1a'_{i+3}\in E(G)$ or $w_1d_2'\in E(G)$. 
		If $w_1a'_{i+3}\in E(G)$, then $x_2'-a''_{i+3}-w_1-a'_{i+3}-d_2'-d_1'-b_i-v_{i-1}-v_{i-2}-v_{i-3}$ is an induced $P_{10}$. 
		So, $w_1d_2'\in E(G)$.  
		Then $w_1-d_2'-a'_{i+3}-v_{i+3}-a''_{i+3}-w_1$ is an induced $C_5$. 
		Thus, we show that $\{b_i,a'_{i+3},w_1\}$ dominates all elements of $W_3^{i,+}$. 
		By symmetry, there are another three vertices that dominate all elements of $W_3^{i,-}$.
		For each $i\in [7]$, we obtain such 6 vertices, and let $T_1^3$ be the union of these vertices. 
		It follows that $|T_1^3|\leq 42$, $T_1^3$ dominates all elements $W_3$ and could be constructed in $O(|V(G)|+|E(G)|)$ time.
		This completes the proof of (c).  
		
		(d) 
		By the definition of $W_4$, $W_4\cap (W_1\cup W_2\cup W_3)=\emptyset$. 
		Thus, $d_1$ has no neighbors in $A_{i-2}\cup A_{i-3}\cup B_{i-2}\cup B_{i-3}$. 
		By Claim \ref{X_1 neighbor in B}, $N(d_1)\cap B\subseteq B_{i}\cup B_{i+2}$. 
		On the other hand, since $G$ has no induced $C_3,C_5,C_9$, $N(d_1)\cap A\subseteq A_{i+1}$ or $N(d_1)\cap A\subseteq A_{i-1}\cup A_{i+3}$. 
		Thus, either $N(d_1)\cap N(C)\subseteq A_{i-1}\cup B_{i}\cup B_{i+2}\cup A_{i+3}$ or $N(d_1)\cap N(C)\subseteq B_{i}\cup A_{i+1}\cup B_{i+2}$.
		If $N(d_1)\cap N(C)\subseteq A_{i-1}\cup B_{i}\cup B_{i+2}\cup A_{i+3}$, then by the Claim \ref{edges in nieghbors of D}, $d_1$ and $b_{i+1}$ are comparable, where $b_{i+1}$ is a neighbor of $d_2$ in $B_{i+1}$. 
		Thus, $N(d_1)\cap N(C)\subseteq B_{i}\cup A_{i+1}\cup B_{i+2}$, and $N(d_2)\cap N(C)\subseteq B_{i-1}\cup A_{i}\cup B_{i+1}$ by symmetry. 
		Hence, $d_1-d_2$ is a reducible component for $L$, a contradiction. 
		This completes the proof of (d).  
		
		(e)
		Let $d_1d_2$ be an element of $W_5$. 
		Suppose first that $d_1$ has a neighbor $b$ in $B_i$. 
		Then by the definition of $W_5$, $N(d_2)\cap N(C)\subseteq A_{i}\cup A_{i\pm 2}$. 
		If $d_2$ has a neighbor in $A_{i+2}$, then $N(d_2)\cap N(C)\subseteq A_{i\pm 2}$, and hence, by Claim \ref{edges in nieghbors of D}, $b$ and $d_2$ are comparable. 
		If $N(d_2)\cap N(C)\subseteq A_{i}$, then either $N(d_1)\cap N(C)\subseteq A_{i+1}$ or $N(d_1)\cap N(C)\subseteq A_{i-1}$. 
		Then $d_1d_2$ is a reducible component. 
		Thus, we may assume that both $d_1$ and $d_2$ has no neighbors in $B$. 
		Let $a\in A_i$ be a neighbor of $d_1$. 
		Then $N(d_2)\cap N(C)\subseteq A_{i-3}\cup A_{i+1}$ or $N(d_2)\cap N(C)\subseteq A_{i-3}\cup A_{i+1}$, say $N(d_2)\cap N(C)\subseteq A_{i-3}\cup A_{i+1}$. 
		Hence, $a$ and $d_2$ are comparable by Claim \ref{edges in nieghbors of D}. 
		This completes the proof of (e).  
	\end{proof}
	
	Let $\mathcal{K}_2$ be the set of components in $G\setminus N[C]$ that has order at least 3. 
    We distinguish $\mathcal{K}_2$ as following:  
		
		\begin{itemize}
			\item $Q_1^i:=\{(U_1,U_2)\in \mathcal{K}_2:\text{there is }x,y\in U_1\text{ such that }N(x)\cap B_i\neq \emptyset\text{ and }N(y)\cap B_{i+1}\neq \emptyset\}$; 
			\item $Q_2^i:=\{(U_1,U_2)\in \mathcal{K}_2\setminus (\bigcup_{i\in[7]} Q_1^i):\text{there is }x,y\in U_1\text{ such that }N(x)\cap B_i\neq \emptyset\text{ and }N(y)\cap B_{i+3}\neq \emptyset\}$; 
			\item $Q_3^i:=\{(U_1,U_2)\in \mathcal{K}_2\setminus ((\bigcup_{i\in[7]} (Q_1^i\cup Q_2^i)):\text{there is }x,y\in U_1\text{ such that }N(x)\cap B_i\neq \emptyset\text{ and }N(y)\cap B_{i+2}\neq \emptyset\}$; 
			\item $Q_4=\mathcal{K}_2\setminus (\bigcup_{i\in[7]} (Q_1^i\cup Q_2^i\cup Q_3^i))$.  
		\end{itemize}

		Note that by Claim \ref{P_4 neighbor}, for any two non-adjacent vertices in a element of $\mathcal{K}_2$, their neighborhoods in $A$ are the same. 
		Let $Q_1=\bigcup_{i\in[7]} Q_1^i$, $Q_2=\bigcup_{i\in[7]} Q_2^i$ and $Q_3=\bigcup_{i\in[7]} Q_3^i$. 
		For an element $(U_1,U_2)$ of $Q_1\cup Q_2\cup Q_3$, we call $U_1$ the {\em signed part} and $U_2$ the {\em unsigned part} of this element.

		\begin{claim}\label{Q_1}
			There is a set $T_2^1$ with $|T_2^1|\leq 21$ such that $T_2^1$ dominates all elements of $Q_1$, and can be obtained in $O(|V(G)|+|E(G)|)$ time. 
		\end{claim}
		\begin{proof}
			Let $K=(U_1,U_2)$ be an element of $Q_1^i$ with $x,y\in U_1$ such that $x$ has a neighbor $b_i\in B_i$ and $y$ has a neighbor $b_{i+1}\in B_{i+1}$. 
			Let $z$ be a vertex in $U_2$ with a neighbor $v\in N(C)$. 
			If there is a vertex $u\in U_1$ that is not adjacent to $b_i$ and $b_{i+1}$, then $u-z-x-b_i-v_{i-1}-v_{i-2}-v_{i-3}-v_{i+3}-v_{i+2}-b_{i+1}$ is an induced $P_{10}$. 
			
			\vspace{0.2cm}
			\noindent (1) All vertices in $B_i$ that is not adjacent to $x$ are complete to $b_{i+1}$ and all vertices in $B_{i+1}$ that is not adjacent to $y$ are complete to $b_i$. 
			\vspace{0.2cm}
			
			By symmetry, we may assume that there is a vertex $b'$ in $B_i$ that is not adjacent to $x$ and $b_{i+1}$. 
			Then $x-z-y-b_{i+1}-v_{i+2}-v_{i+3}-v_{i-3}-v_{i-2}-v_{i-1}-b'$ is an induced $P_{10}$, a contradiction.
			This completes the proof of (1). 
			
			\vspace{0.2cm}
			
			Let $K'=(U_1',U_2')$ be another element of $Q_1^i$ such that there are $x',y'\in U'_1$ and $z'\in U_2'$ with $x'$ has a neighbor in $B_i$ and $y'$ has a neighbor in $B_{i+1}$. 
			
			\vspace{0.2cm}
			\noindent (2) The neighborhoods of $x$ and $x'$ in $B_i$ are either the same or disjoint. 
			If the former, then the neighborhoods of $y$ and $y'$ in $B_{i+1}$ are also the same, and if the later, then the neighborhoods of $y$ and $y'$ in $B_{i+1}$ are also disjoint. 
			\vspace{0.2cm}
			
			If the neighborhoods of $x$ and $x'$ in $B_{i}$ are different, say $b_i$ is not adjacent to $x'$, then all neighbors of $y'$ in $B_{i+1}$ are complete to $b_1$ by (1).  
			So the neighborhoods of $y$ and $y'$ in $B_{i+1}$ are disjoint. 
			And hence, the neighborhoods of $x$ and $x'$ in $B_i$ are disjoint by symmetry. 
			This completes the proof of (2). 
			
			\vspace{0.2cm}
			\noindent (3) The neighborhoods of $U_1$ and $U_1'$ in $B_i$ are the same.
			\vspace{0.2cm}
			
			Suppose to the contrary that the neighborhoods of $U_1$ and $U_1'$ in $B_i$ are disjoint. 
			Let $b_i'$ be a neighbor of $x'$ in $B_i$ and $b_{i+1}'$ be a neighbor of $y'$ in $B_{i+1}$. 
			By (1), $b_i'b_{i+1},b_ib_{i+1}'\in E(G)$. 
			Assume first that $v\in B$.  
			It follows that $v\in B_{i-3}$. 
			If $z'$ is not adjacent to $v$, then $y'-z'-x'-b_i'-v_{i+1}-b_i-x-z-v-v_{i+3}$ is an induced $P_{10}$. 
			So, $z'$ is adjacent to $v$. 
			Then $y-z-v-z'-x'-b_i'-v_{i-1}-b_i-b_{i+1}'-v_{i+2}$ is an induced $P_{10}$. 
			So we may assume that $v\in A$. 
			It follows that $v\in A_{i-2}\cup A_{i-3}\cup A_{i+3}$. 
			If $v\in A_{i-3}$, then $x'-b_i'-v_{i-1}-b_i-x-z-v-v_{i-3}-v_{i+3}-v_{i+2}$ is an induced $P_{10}$. 
			So, $v$ is in $A_{i+3}\cup A_{i-2}$, say $A_{i+3}$. 
			Then $x'-b_i'-v_{i+1}-b_i-x-z-v-v_{i+3}-v_{i-3}-v_{i-2}$ is an induced $P_{10}$. 
			Thus, $z$ has no neighbors in $N(C)$, a contradiction. 
			This completes the proof of (3). 
			
			\vspace{0.2cm}
			\noindent (4) The neighborhoods of $U_2$ and $U_2'$ in $N(C)$ are the same.
			\vspace{0.2cm}
			
			By symmetry, we only need to prove $N(z)\cap N(C)\subseteq N(z')\cap N(C)$. 
			We assume first that $z$ has a neighbor $b$ in $B$ that is not adjacent to $z'$. 
			It follows that $b\in B_{i-3}$. 
			Then $v_{i}-b_{i+1}-y'-z'-x'-b_i-x-z-b-v_{i+3}$ is an induced $P_{10}$.  
			So we may assume that $z$ has a neighbor $a$ in $A$ that is not adjacent to $z'$. 
			It follows that $a\in A_{i+3}\cup A_{i-3}\cup A_{i-2}$. 
			If $a\in A_{i-3}$, then $b_i-x'-z'-y'-b_{i+1}-y-z-a-v_{i-3}-v_{i+3}$ is an induced $P_{10}$. 
			If $a$ is in $A_{i+3}\cup A_{i-2}$, say $A_{i+3}$, then $y'-z'-x'-b_i-x-z-a-v_{i+3}-v_{i-3}-v_{i-2}$ is an induced $P_{10}$. 
			This completes the proof of (4). 
			
			\vspace{0.2cm}
			By the arbitrariness of $K'$, we have that $\{b_i,b_{i+1},v\}$ dominates all elements of $Q_1^i$. 
			For each $i\in [7]$, we obtain such three vertices, and let $T_2^1$ be the union of these vertices for all $i\in [7]$. 
			So, it follows that $|T_2^1|\leq 21$ and $T_2^1$ dominates all elements of $Q_1$.  
			This completes the proof of Claim \ref{Q_1}.  
		\end{proof}

		\begin{claim}\label{Q_2}
			There is a set $T_2^2$ with $|T_2^2|\leq 14$ such that $T_2^2$ dominates all elements of $Q_2$, and can be obtained in $O(|V(G)|+|E(G)|)$ time. 
		\end{claim}
		\begin{proof}
			Let $K=(U_1,U_2)$ be an element of $Q_2^i$ such that there are $x,y\in U_1,z\in U_2$ with $x$ has a neighbor $b_i$ in $B_i$ and $y$ has a neighbor $b_{i+3}$ in $B_{i+3}$. 
			Since $G$ has no comparable pair, $v_i$ has a neighbor $v$ not adjacent to $b_i$.

			\vspace{0.2cm}
			\noindent (1) $v\in A_i$ and $v$ is complete to $U_2$.  
			\vspace{0.2cm}

			Since $v-v_{i}-v_{i-1}-b_i-x-z-y-b_{i+3}-v_{i-3}-v_{i+3}$ is not an induced $P_{10}$, $v$ is adjacent to $y$ or $z$. 
			Note that $v\in B_{i-1}\cup A_{i}\cup B_{i+1}$. 
			If $v\in B_{i-1}$, by Claim \ref{X_1 neighbor in B}, $vz\in E(G)$, and hence, $x-z-v-v_{i-2}-v_{i-3}-v_{i+3}-v_{i+2}-v_{i+1}-b_i-x$ is an induced $C_9$. 
			If $v\in B_{i+1}$, then $vy\in E(G)$, and hence, $K\in Q_1$. 
			So, we may assume that $v\in A_i$. 
			If $v$ is adjacent to $y$, then $vx\in E(G)$, and hence, $x-v-v_i-v_{i-1}-b_i-x$ is an induced $C_5$. 
			So, by the arbitrariness of $z$, $v$ is complete to $U_2$.
			This completes the proof of (1).

			\vspace{0.2cm}
			\noindent (2) $|Q_2^i|=1$. 
			\vspace{0.2cm}
			
			Suppose to the contrary that there is another element $K'=(U_1',U_2')$ of $Q_2^i$ such that there are $x',y'\in U'_1$ and $z'\in U_2'$ with $x'$ has a neighbor in $B_i$ and $y'$ has a neighbor in $B_{i+3}$. 
			Assume first that the neighborhoods of $x$ and $x'$ in $B_{i}$ are incomparable. 
			So, we may assume that $b_i$ is not adjacent to $x'$ and $x'$ has a neighbor $b_i'\in B_i$ not adjacent to $x$. 
			Then $z'-x'-b_i'-v_{i-1}-b_i-x-z-y-b_{i+3}-v_{i-3}$ is an induced $P_{10}$. 
			Thus, the neighborhoods of $x$ and $x'$ in $B_{i}$ are comparable. 
			By symmetry, we may assume that $x'-b_i$ and $y'-b_{i+3}$. 
			Then $x'-b_i-v_{i-1}-v_i-v-z-y-b_{i+3}-v_{i-3}-v_{i+3}$ is an induced $P_{10}$. 
			This completes the proof of (2). 
			
			\vspace{0.2cm}
			
			So, $\{x,z\}$ dominates all elements of $Q_2^i$ by (2).  
			For each $i\in [7]$, we obtain such two vertices, and let $T_2^2$ be the union of these vertices for all $i\in [7]$. 
			Then $|T_2^2|\leq 14$ and $T_2^2$ dominates all elements of $Q_2$.  
			This completes the proof of Claim \ref{Q_2}.  
		\end{proof}

		\begin{claim}\label{Q_3}
			One can reduce $L$ to obtain a subpalette $L'$ of $L$ in $O(|V(G)|+|E(G)|)$ time such that $|L'(v)|\leq 2$ for every $v$ in a element of $Q_3$. Moreover, $(G,L,Z)$ is colorable if and only if $(G,L',Z)$ is colorable. 
		\end{claim}
		\begin{proof}
			Note that for each element of $Q_3^i$, all neighbors of its signed part in $B$ are in $B_i\cup B_{i+2}$. 
			
			\vspace{0.2cm}
			\noindent (1) For each element of $Q_3^i$, all vertices in its signed part have the same neighborhood in $B_i$ or $B_{i+2}$. 
			\vspace{0.2cm}
			
			Let $K=(U_1,U_2)\in Q_3^i$ such that there are $x,y\in U_1$ with $x$ has a neighbor $b_i$ in $B_i$ and $y$ has a neighbor $b_{i+2}$ in $B_{i+2}$. 
			Since $G$ has no comparable pair, we may assume that $x$ and $y$ have different neighborhoods in $B$. 
			Without loss of generality, we may assume that $b_i$ is not adjacent to $y$. 
			Since $G$ has no induced $C_9$, $b_{i+2}x\in E(G)$. 
			By the arbitrariness of $b_{i+2}$, $N(y)\cap B_{i+2}\subseteq N(x)\cap B_{i+2}$. 
			Since $G$ has no comparable pair again, $y$ has a neighbor $b_i'$ in $B_{i}$ that is not adjacent to $x$. 
			By symmetry, $N(x)\cap B_{i+2}\subseteq N(y)\cap B_{i+2}$. 
			Suppose to the contrary that there is a vertex $w\in U_1$ such that $N(w)\cap B_{i+2}\neq N(x)\cap B_{i+2}$. 
			We first assume that $N(x)\cap B_{i+2}\subseteq N(w)\cap B_{i+2}$
			Then since $G$ has no comparable pair, $w$ has a neighbor $w'$ in $B_{i+2}$ not adjacent to $x$. 
			By symmetry, $w$ is complete tox  $N(x)\cap B_i$ and $N(y)\cap B_i$, which is a contradiction since $N(x)\subseteq N(w)$. 
			So we may assume that there is a vertex in $N(x)\cap B_{i+2}$ that is not adjacent to $w$. 
			By symmetry, $N(x)\cap B_i=N(w)\cap B_i=N(y)\cap B_i$, a contradiction.  
			This completes the proof of (1). 
			
			\vspace{0.2cm}
			
			Let $Q_3^{-,i}$(or $Q_3^{+,i}$, respectively) be the set of elements in $Q_3$ such that all neighbors of its signed part have the same neighborhood in $B_i$(or $B_{i+2}$, respectively). 
			Let $K=(U_1,U_2)\in Q_3^{+,i}$ with $z\in U_2$. 
			Note that $N(U_2)\cap N(C)\subseteq A_i\cup A_{i-3}\cup B_{i+1}\cup B_{i+3}$, $N(U_2)\cap N(C)\subseteq A_{i+2}\cup B_{i+1}\cup B_{i+3}$, $N(U_2)\cap N(C)\subseteq A_i\cup B_{i-1}\cup B_{i+1}$ or $N(U_2)\cap N(C)\subseteq A_{i-2}\cup A_{i+2}\cup B_{i-1}\cup B_{i+1}$. 
			If $N(U_2)\cap N(C)\subseteq A_i\cup A_{i-3}\cup B_{i+1}\cup B_{i+3}$, then by Claim \ref{edges in nieghbors of D}, $z$ and $b_{i+2}$ are comparable. 
			
			\vspace{0.2cm}
			\noindent (2) If $L(v_{i+1})\neq L(v_{i+3})$, then one can reduce $L$ to obtain a subpalette $L'$ of $L$ in $O(|V(G)|+|E(G)|)$ time such that $|L'(v)|\leq 2$ for every $v\in K$. Moreover, $(G,L,Z)$ is colorable if and only if $(G,L',Z)$ is colorable. 
			\vspace{0.2cm}
			
			Without loss of generality, we may assume that $L(v_{i+1})=\{1\}$ and $L(v_{i+3})=\{2\}$. 
			Since $L$ is updated, $L(v)=\{3\}$ for all $v\in B_{i+2}$, and hence $|L(v)|\leq 2$ for every $v$ in $U_1$. 
			We first assume that $N(U_2)\cap N(C)\subseteq A_{i+2}\cup B_{i+1}\cup B_{i+3}$. 
			Since $L(v_{i+2})=\{3\}$, $3\notin L(v)$ for every $v$ in the neighborhood of every vertex in $U_2$, a contradiction. 
			We then assume that $N(U_2)\cap N(C)\subseteq A_i\cup B_{i-1}\cup B_{i+1}$ or $N(U_2)\cap N(C)\subseteq A_{i-2}\cup A_{i+2}\cup B_{i-1}\cup B_{i+1}$. 
			Let $b_i$ be a vertex in $B_i$ that has a neighbor $x$ in $U_1$. 
			For any 3-coloring $c$ of $(G,L,Z)$, either $c(b_{i})=2$ or $c(b_{i})=3$. 
			If $c(b_{i})=3$, then by Claim \ref{edges in nieghbors of D}, $3\notin L(v)$ for every $v$ in the neighborhood of every vertex in $U_2$. 
			Hence, there is a coloring of $(G,L,Z)$ such that $c'(v)=3$ for every $v\in U_2$ and $c'(v)=c(v)$ for every remaining vertex $v$. 
			If$c(b_{i})=2$, then $c(x)=1$, and hence, $c(v)\neq 1$ for every $v\in U_2$. 
			Define $L'$ by setting $L'(v)=L(v)\setminus \{1\}$ for every $v\in U_2$ and leaving $L'(u)=L(u)$ for every remaining vertex $u$. 
			Then $|L'(v)|\leq 2$ for every $v\in K$, can be obtained in $O(|V(G)|+|E(G)|)$ time and $(G,L,Z)$ is colorable if and only if $(G,L',Z)$ is colorable.
			This completes the proof of (2).

			\vspace{0.2cm}
			\noindent (3) If $L(v_{i+1})=L(v_{i+3})$, then one can reduce $L$ to obtain a subpalette $L'$ of $L$ in $O(|V(G)|+|E(G)|)$ time such that $|L'(v)|\leq 2$ for every $v\in K$. Moreover, $(G,L,Z)$ is colorable if and only if $(G,L',Z)$ is colorable. 
			\vspace{0.2cm}
			
			Without loss of generality, we may assume that $L(v_{i+1})=L(v_{i+3})=\{1\}$ and $L(v_{i-1})\neq \{3\}$. 
			Thus, either $L(v_{i-1})=\{1\}$ or $L(v_{i-1})=\{2\}$
			Suppose first that $L(v_{i-1})=\{1\}$. 
			If $N(U_2)\cap N(C)\subseteq A_{i+2}\cup B_{i+1}\cup B_{i+3}$ or $N(U_2)\cap N(C)\subseteq A_i\cup B_{i-1}\cup B_{i+1}$, then $K$ is a reducible component for $L$, a contradiction.  
			So we may assume that $N(U_2)\cap N(C)\subseteq A_{i-2}\cup A_{i+2}\cup B_{i-1}\cup B_{i+1}$.
			If $L(v_{i-2})=L(v_{i+2})=\{1\}$, then $K$ is also a reducible component for $L$, a contradiction.  
			So we may assume that $L(v_{i-2})\neq L(v_{i+2})$. 
			Let $z$ be a vertex in $U_2$. 
			Note that for any coloring $c$ of $(G,L,Z)$, $c(z)\neq 1$. 
			Define $L'$ by setting $L'(v)=1$ for every $v\in U_1$, $L'(v)=L(v)\setminus \{1\}$ for every $v\in U_2$ and leaving $L'(v)=L(v)$ for every remaining vertex $v$. 
			Then $|L'(v)|\leq 2$ for every $v\in K$, can be obtained in $O(|V(G)|+|E(G)|)$ time and $(G,L,Z)$ is colorable if and only if $(G,L',Z)$ is colorable.  
			
			We then suppose that $L(v_{i-1})=\{2\}$. 
			Since $L$ is updated, $L(v)=\{3\}$ for all $v\in B_i$, and hence $c(v)\neq 3$ for all $v\in U_1$.
			If $N(U_2)\cap N(C)\subseteq A_i\cup B_{i-1}\cup B_{i+1}$, $N(U_2)\cap N(C)\subseteq A_{i-2}\cup A_{i+2}\cup B_{i-1}\cup B_{i+1}$ or $N(U_2)\cap N(C)\subseteq A_{i+2}\cup B_{i+1}\cup B_{i+3}$ and $L(v_{i+2})=\{3\}$, then by Claim \ref{edges in nieghbors of D}, $3\notin L(v)$ for all $v$ in the neighborhood of $U_2$ and hence, every vertex in $U_2$ is a reducible vretex for $L$, a contradiction. 
			So we may assume that $N(U_2)\cap N(C)\subseteq A_{i+2}\cup B_{i+1}\cup B_{i+3}$ and $L(v_{i+2})=\{2\}$. 
			Let $z\in U_2$. 
			Note that $3\notin L(v)$ for every $v\in N(z)\cap A_{i+2}\cup B_{i+1}$ by Claim \ref{edges in nieghbors of D}. 
			It follows that if $z$ has no neighbors in $B_{i+3}$, $z$ is a reducible vertex, a contradiction. 
			Let $b\in B_{i+3}$ be a neighbor of $z$. 
			Note that for any coloring $c$ of $(G,L,Z)$, either $c(b)=1$ or $c(b)=3$. 
			If $c(b)=3$, then all neighbors of $U_1$ in $B_{i+2}$ are colored 2, and hence, $c(v)=1$ for every $v\in U_1$. 
			Thus, no matter $c(b)=1$ or $c(b)=3$, $c(z)\neq 1$. 
			Define $L'$ by setting $L'(v)=L(v)\setminus \{1\}$ for every $v\in U_2$ and leaving $L'(v)=L(v)$ for every remaining vertex $v$. 
			Then $|L'(v)|\leq 2$ for every $v\in K$, can be obtained in $O(|V(G)|+|E(G)|)$ time and $(G,L,Z)$ is colorable if and only if $(G,L',Z)$ is colorable.  
			This completes the proof of (3). 
			
			\vspace{0.2cm}
			
			By applying (2) and (3) for every element of $Q_3$, we complete the proof of Claim \ref{Q_3}. 
		\end{proof}

		Note that for each element $K=(U_1,U_2)\in Q_4$, if $|U_j|\neq 1$, $N(U_j)\cap B\subseteq B_i$ for some $i\in [7]$, where $j=1,2$. 
		Without loss of generality, we may assume that $|U_1|\geq |U_2|$ and call $U_1$ its {\em signed part}, $U_2$ its {\em unsigned part}  as well.

		\begin{claim}\label{Q_4}
			There is a set of restrictions $\mathcal{R}$ of $(G,L,Z)$ with size in $O(1)$, such that $(G,L,Z)$ is colorable if and only if $\mathcal{R}$ is colorable, $\mathcal{R}$ can be obtained in $O(|V(G)|+|E(G)|)$ time and there are no vertices in every element of $Q_4$ have list length 3. 
		\end{claim}
		\begin{proof}
			Let $Q_4^i$ be the set of elements of $Q_4$ such that its signed part has neighbors in $B_i$, $Q_4^{i,1,+}$ (or $Q_4^{i,1,-}$, respectively) be the set of elements of $Q_4^i$ such that its unsigned part has neighbors in $A_{i+3}$ (or $A_{i-3}$, respectively). 
			Set $Q_4^{i,1}=Q_4^{i,1,+}\cup Q_4^{i,1,-}$. 
			
			\vspace{0.2cm}
			\noindent (1) There exist 8 vertices that dominate every element of $Q_4^{i,1}$, and these vertices can be obtained in $O(|V(G)|+|E(G)|)$ time. 
			\vspace{0.2cm}
			
			By symmetry, we only need to prove there are 4 vertices that dominate all elements of $Q_4^{i,1,+}$. 
			Suppose that there are two elements $K:=(U_1,U_2),K':=(U_1',U_2')$ of $Q_4^{i,1,+}$ such that the neighborhoods of $U_2$ and $U_2'$ in $A_{i+3}$ are incomparable. 
			Let $x,y\in U_1$, $z\in U_2$, $x',y'\in U_1'$, $z'\in U_2'$ with $x$ has a neighbor $b_i$ in $B_i$ not adjacent to $y$, $y$ has a neighbor $b_i'$ in $B_i$ not adjacent to $x$, $z$ has a neighbor $a$ in $A_{i+3}$. 
			Since the neighborhoods of $U_2$ and $U_2'$ in $A_{i+3}$ are incomparable, we may assume that $a$ is not adjacent to $z'$ and $z'$ has a neighbor $a'\in A_{i+3}$ not adjacent to $z$. 
			Then since $x'-z'-a'-v_{i+3}-a-z-x-b_i-v_{i-1}-v_{i-2}$ is not an induced $P_{10}$, $x'b_i\in E(G)$. 
			By the arbitrariness of $b_i$, $N(x)\cap B_i\subseteq N(x')$. 
			Then by the symmetry of $x$ and $x'$, $x$ and $y$, $N(x)\cap B_i=N(x')\cap B_i=N(y)\cap B_i$, a contradiction. 
			Thus, all neighborhoods of the vertices in unsigned parts of different elements of $Q_4^{i,1,+}$ are comparable in $A_{i+3}$. 
			So, there is a vertex $a$ in $A_{i+3}$ such that $a$ adjacent to all vertices in the unsigned parts of all elements of $Q_4^{i,1,+}$. 
			
			We then assume that there are two elements $K:=(U_1,U_2),K':=(U_1',U_2')$ of $Q_4^{i,1,+}$ such that there are two vertices $x\in U_1$, $x\in U_1'$ with the neighborhoods of $x$ and $x'$ in $B_i$ are incomparable. 
			Let $z\in U_2$ and $b_i,b_i'$ in $B_i$ such that $b_i$ is adjacent to $x$ but $x'$, $b_i'$ is adjacent to $x'$ but $x$. 
			Then $x'-b_i'-v_{i+1}-b_i-x-z-a-v_{i+3}-v_{i-3}-v_{i-2}$ is an induced $P_{10}$. 
			Thus, all neighborhoods of the vertices in signed parts of different elements of $Q_4^{i,1,+}$ are comparable in $B_{i}$.
			So, there is a vertex $b$ in $B_{i}$ such that $b$ adjacent to all vertices in signed parts of all elements of $Q_4^{i,1,+}\setminus K_1$ except one, say $K_1$. 
			Let $x_1y_1$ be an edge in $K_1$. 
			Then $\{a,b,x_1,y_1\}$ dominates all elements of $Q_4^{i,1,+}$ and can be obtained in $O(|V(G)|+|E(G)|)$ time. 
			This completes the proof of (1).

			\vspace{0.2cm}
			
			Let $Q_4^{i,2,+}$ (or $Q_4^{i,2,-}$, respectively) be the set of all elements of $Q_4^i\setminus Q_4^{i,1}$ such that its unsigned part has neighbors in $B_{i+3}$ (or $B_{i-3}$, respectively). 
			Set $Q_4^{i,2}=Q_4^{i,2,+}\cup Q_4^{i,2,-}$.
			
			\vspace{0.2cm}
			\noindent (2) There exist 12 vertices that dominate $Q_4^{i,2}$, and these 12 vertices can be found in $O(|V(G)|+|E(G)|)$ time. 
			\vspace{0.2cm}
			
			By symmetry, we only need to prove there are 6 vertices that dominate all elements of $Q_4^{i,2,+}$. 
			Suppose that there are two elements $K:=(U_1,U_2),K':=(U_1',U_2')$ of $Q_4^{i,2,+}$ such that there are two vertices $z\in U_2,z'\in U_2'$ with incomparable neighborhoods in $B_{i+3}$. 
			Let $x\in U_1,x'\in U_1',b_{i}\in B_{i}$ and $b_{i+3},b_{i+3}'\in B_{i+3}$ with $xb_i\in E(G)$ and $b_{i+3}$ is adjacent to $z$ but $z'$, $b_{i+3}'$ is adjacent to $z'$ but $z$. 
			Then since $v_{i-2}-v_{i-1}-b_i-x-z-b_{i+3}-v_{i+2}-b_{i+3}'-z'-x'$ is not an induced $P_{10}$, $x'b_{i}\in E(G)$. 
			By the arbitrariness of $b_i$, $N(x)\cap B_i\subseteq N(x')$. 
			Then by the symmetry of $x$ and $x'$, $x$ and $y$, $N(x)\cap B_i=N(x')\cap B_i=N(y)\cap B_i$, a contradiction. 
			Thus, all neighborhoods of the vertices in unsigned parts of different elements of $Q_4^{i,2,+}$ are comparable in $B_{i+3}$.
			So, there is a vertex $b$ in $B_{i+3}$ such that $b$ is adjacent to all vertices in the unsigned parts of all elements of $Q_4^{i,2,+}$ except one, say $K_1$. 
			
			We then assume that there are two elements $K:=(U_1,U_2),K':=(U_1',U_2')$ of $Q_4^{i,2,+}\setminus K_1$ such that there are two vertices $x\in U_1,x'\in U_1'$ with the neighborhoods of $x$ and $x'$ in $B_{i}$ are incomparable. 
			Let $z\in U_2,z'\in U_2'$ and $b_{i},b_{i}'\in B_{i}$ with $b_i$ is adjacent to $x$ but $x'$, $b_i'$ is adjacent to $x'$ but $x$. 
			Since $v_{i+3}$ and $b$ are not comparable pair, there is a vertex $w$ adjacent $v_{i+3}$ but $b$. 
			Then since $x'-b_{i}'-v_{i-1}-b_i-x-z-b-v_{i+2}-v_{i+3}-w$ is not an induced $P_{10}$, $w\in A_{i+3}$ and either $wx\in E(G)$ or $wx'\in E(G)$. 
			If $w$ is adjacent to exactly one of $x$ and $x'$, say $x$, then $v_{i+3}-w-x-z-b-z'-x'-b_i'-v_{i-1}-v_i$ is an induced $P_{10}$. 
			Thus, $w$ is adjacent to both $x$ and $x'$. 
			Let $K_2$ be the element of $Q_4^{i,2,+}\setminus K_1$ such that there is a vertex $v$ in its signed part with the neighborhood of $v$ in $B_i$ is minimal among all vertices in the signed parts of all elements of $Q_4^{i,2,+}\setminus K_1$. 
			Let $v'$ be a neighbor of $v$ in $B_i$, $v''$ be a vertex in the unsigned part of $K_2$ and $x_1y_1$ be an edge in $K_1$. 
			It follows that $\{b,w,v',v'',x_1,y_1\}$ dominates all elements of $Q_4^{i,2,+}$ and can be obtained in $O(|V(G)|+|E(G)|)$ time. 
			This completes the proof of (2). 
			
			\vspace{0.2cm}
			
			Let $Q_4^{i,3}=Q_4^i\setminus (Q_4^{i,1}\cup Q_4^{i,2})$.  
			It follows that for each element $K=(U_1,U_2)\in Q_4^{i,3}$, for every $v\in U_2$, $N(v)\cap N(C)\subseteq A_{i-2}\cup B_{i-1}\cup A_{i}\cup B_{i+1}\cup A_{i+2}$.  
			Then for all $v\in U_2$, either $N(v)\cap N(C)\subseteq B_{i-1}\cup A_{i}\cup B_{i+1}$ or $N(v)\cap N(C)\subseteq A_{i-2}\cup B_{i-1}\cup B_{i+1}\cup A_{i+2}$. 
			Let $Q_4^{i,3,1}$ (or $Q_4^{i,3,2}$, respectively) be the set of elements in $Q_4^{i,3}$ such that the neighborhoods of its unsigned part are in $B_{i-1}\cup A_{i}\cup B_{i+1}$ (or $A_{i-2}\cup B_{i-1}\cup B_{i+1}\cup A_{i+2}$, respectively). 
			So, $Q_4^{i,3}=Q_4^{i,3,1}\cup Q_4^{i,3,2}$. 
			
			For any element $K=(U_1,U_2)$ of $Q_4^{i,3,1}$ such that $U_1$ has no neighbors in $A_{i-3}$ and $A_{i+3}$, either $N(U_1)\cap N(C)\subseteq A_{i-1}\cup B_i$ or $N(U_1)\cap N(C)\subseteq A_{i+1}\cup B_i$. 
			So, $K$ is a reducible component for $L$, a contradiction. 
			Thus, for every element $K=(U_1,U_2)$ of $Q_4^{i,3,1}$, $U_1$ has neighbors in $A_{i-3}\cup A_{i+3}$. 
			Let $Q_4^{i,3,1,-}$ (or $Q_4^{i,3,1,+}$, respectively) be the set of elements of $Q_4^{i,3,1}$ such that its signed part has neighbors in $A_{i-3}$ (or $A_{i+3}$, respectively). 
			
			\vspace{0.2cm}
			\noindent (3) There exist 10 vertices that dominate all elements of $Q_4^{i,3,1}$, and these vertices can be found in $O(|V(G)|+|E(G)|)$ time. 
			\vspace{0.2cm}
			
			By symmetry, we only need to prove there are 5 vertices that dominate all elements of $Q_4^{i,3,1,+}$. 
			If there is a vertex $v$ in the unsigned part of an element of $Q_4^{i,3,1,+}$ such that $v$ has no neighbors in $B_{i-1}$, then by Claim \ref{edges in NC}, $v$ and $a$ are comparable, where $a\in A_{i+3}$ is a neighbor of a vertex in the signed part of this element. 
			Thus, we may assume that for every $v$ in the unsigned parts of all elements of $Q_4^{i,3,1,+}$, $N(v)\cap B_{i-1}\neq \emptyset$. 
			Suppose that there are two elements $K:=(U_1,U_2),K':=(U_1',U_2')$ of $Q_4^{i,3,1,+}$ such that there are two vertices $z\in U_2,z'\in U_2'$ with incomparable neighborhoods in $B_{i-1}$.  
			Let $x\in U_1,b_{i-1},b_{i-1}'\in B_{i-1}, a_{i+3}\in A_{i+3}$ such that $xa_{i+3}\in E(G)$ and $b_{i-1}$ is adjacent to $z$ but $z'$, $b'_{i-1}$ is adjacent to $z'$ but $z$. 
			Then $z'-b_{i-1}'-v_{i-2}-b_{i-1}-z-x-a_{i+3}-v_{i+3}-v_{i+2}-v_{i+1}$ is an induced $P_{10}$. 
			Thus, all neighborhoods of the vertices in the unsigned parts of different elements of $Q_4^{i,3,1,+}$ are comparable in $B_{i-1}$.
			So, there is a vertex $b$ in $B_{i-1}$ such that $b$ is adjacent to all vertices in unsigned parts of all elements of $Q_4^{i,3,1,+}$ except one, say $K_1$. 
			Since $G$ has no comparable pair, there is a vertex $w$ adjacent to $v_{i-1}$ but $b$. 
			
			We then assume that there are two elements $K:=(U_1,U_2),K':=(U_1',U_2')$ of $Q_4^{i,3,1,+}\setminus K_1$ such that there are two vertices $x\in U_1,x'\in U_1'$ with incomparable neighborhoods in $A_{i+3}$.  
			Let $z\in U_2$ and $a_{i+3},a_{i+3}'\in A_{i+3}$ with $a_i$ is adjacent to $x$ but $x'$, $a_i'$ is adjacent to $x'$ but $x$. 
			Then since $w-v_{i-1}-v_i-b-z-x-a_{i+3}-v_{i+3}-a_{i+3}'-x'$ is not an induced $P_{10}$, $w\notin B_{i}$. 
			If $w\in B_{i-2}$, either $wa_{i+3}\in E(G)$ or  $wa_{i+3}'\in E(G)$, say $wa_{i+3}'$, and hence, $v_{i+1}-v_{i+2}-v_{i+3}-a_{i+3}'-w-v_{i-1}-v_{i-2}-b-z-x$ is an induced $P_{10}$. 
			If $w\in A_{i-1}$, then $wa_{i+3}\in E(G)$, $wa_{i+3}'\in E(G)$, $wx\in E(G)$ or $wx'\in E(G)$. 
			Note that if $wa_{i+3}\in E(G)$, then $w$ is not adjacent to $x'$ and if $wa_{i+3}'\in E(G)$, then $w$ is not adjacent to $x$. 
			Thus, if $wa_{i+3}\in E(G)$ or $wa_{i+3}'\in E(G)$, say $wa_{i+3}'\in E(G)$, then $v_{i+1}-v_{i+2}-v_{i+3}-a_{i+3}'-w-v_{i-1}-v_{i-2}-b-z-x$ is also an induced $P_{10}$. 
			So we may assume that either $wx\in E(G)$ or $wx'\in E(G)$. 
			If $w$ is not adjacent to $x'$, then $x-w-v_{i-1}-v_i-b-z'-x'-a_{i+3}'-v_{i+3}-v_{i+2}$ is an induced $P_{10}$. 
			Thus, $w$ is adjacent to both $x$ and $x'$. 
			Let $v$ be a vertex in the signed parts of a element of $Q_4^{i,3,1,+}\setminus K_1$ such that $v$ has a minimal neighborhood in $A_{i+3}$ among all vertices the signed parts of all elements of $Q_4^{i,3,1,+}\setminus K_1$. 
			Let $v'$ be a neighbor of $v$ in $A_{i+3}$, $x_1y_1$ be an edge in $K_1$. 
			Then it follows that $\{b,w,v',x_1,y_1\}$ dominate all elements of $Q_4^{i,3,1,+}$ and can be obtained in $O(|V(G)|+|E(G)|)$ time. 
			This completes the proof of (3). 
			
			\vspace{0.2cm}
			
			For any element $K=(U_1,U_2)\in Q_4^{i,3,2}$, either $N(U_1)\cap N(C)\subseteq A_{i-1}\cup B_{i}\cup A_{i+3}$ or $N(U_1)\cap N(C)\subseteq A_{i-3}\cup B_{i}\cup A_{i+1}$. 
			By symmetry, we may assume that $N(U_1)\cap N(C)\subseteq A_{i-1}\cup B_{i}\cup A_{i+3}$.
			Since $G$ has no comparable pair, it follows that $U_2$ has no neighbors in $A_{i+2}$. 
			We assume first that $U_2$ has a neighbor $b_{i+1}$ in $B_{i+1}$. 
			Since $b_{i+1}$ and $x$ are not comparable, where $x\in U_1$, $|U_2|\geq 2$ and hence, $N(U_2)\cap N(C)\subseteq A_{i-2}\cup B_{i+1}$. 
			Then it follows that $U_1$ has no neighbors in $A_{i-1}$.
			If $L(v_{i-1})\neq L(v_{i+1})$, by symmetry, we may assume that $L(v_{i-1})=\{1\}$ and $L(v_{i+1})=\{2\}$, then $L(B_i)=\{3\}$, and hence, $3\notin L(v)$ for every $v$ in neighborhood of $U_2$. 
			Thus, every vertex in $U_2$ is a reducible vertex for $L$, a contradiction. 
			If $L(v_{i})\neq L(v_{i+2})$, then by a similar analysis, every vertex in $U_1$ is a reducible vertex for $L$, a contradiction.
			So, we may assume that $L(v_{i-1})=L(v_{i+1})=\{1\}$ and $L(v_{i})=L(v_{i+2})=\{2\}$. 
			If $L(v_{i-2})=\{2\}$ and $L(v_{i+3})=\{1\}$, then $K$ is a reducible component, a contradiction. 
			
			We then assume that $L(v_{i-2})=L(v_{i+3})=\{3\}$. 
			Then for any coloring $c$ of $(G,L,Z)$, either $c(a_{i-2})=1$ or $c(a_{i-2})=2$ and either $c(a_{i+3})=1$ or $c(a_{i+3})=2$, where $a_{i-2}$ is a vertex of $U_2$ in $A_{i-2}$ and $a_{i+3}$ is a vertex of $U_1$ in $A_{i+3}$. 
			If for every $v\in N(U_1)\cap A_{i+3}$, $c(v)\neq 1$ and for every $v\in N(U_2)\cap A_{i-2}$, $c(v)\neq 2$, then there is also a coloring $c'$ of $(G,L,Z)$ such that $c'(v)=1$ for every $v\in U_1$, $c'(v)=2$ for every $v\in U_2$ and $c'(v)=c(v)$ for every remaining vertex $v$. 
			If there is a vertex $v$ in $N(U_1)\cap A_{i+3}$ with $c(v)=1$ or in $N(U_2)\cap A_{i-2}$ with $c(v)=2$, say in $N(U_2)\cap A_{i-2}$ with $c(v)=2$, then by Claim \ref{edges in nieghbors of D}, $c(u)=1$ for every $u\in U_1$ and $c(u)=3$ for every $u\in U_2$. 
			Thus, for any coloring $c$ of $(G,L,Z)$, $c(u)\neq 2$ for every $u\in U_1$ and $c(u)\neq 1$ for every $u\in U_2$. 
			Define $L'$ by setting $L'(v)=L(v)\setminus \{2\}$ for every $v\in U_1$, $L'(v)=L(v)\setminus \{1\}$ for every $v\in U_2$ and leaving $L'(v)=L(v)$ for every remaining vertex $v$. 
			Then $|L'(v)|\leq 2$ for every $v\in K$, can be obtained in $O(|V(G)|+|E(G)|)$ time and $(G,L,Z)$ is colorable if and only if $(G,L',Z)$ is colorable.   
			
			So we may assume that exactly one of $L(v_{i-2})$ and $L(v_{i+3})$ is equal to $\{3\}$, say $L(v_{i-2})=\{3\}$ and $L(v_{i+3})=\{1\}$. 
			By a similarly analysis above, define $L'$ by setting $L'(v)=L(v)\setminus \{2\}$ for every $v\in U_1$, $L'(v)=L(v)\setminus \{1\}$ for every $v\in U_2$ and leaving $L'(v)=L(v)$ for every remaining vertex $v$. 
			Then $|L'(v)|\leq 2$ for every $v\in K$, can be obtained in $O(|V(G)|+|E(G)|)$ time and $(G,L,Z)$ is colorable if and only if $(G,L',Z)$ is colorable.   
			The proof of the case that $U_2$ has no neighbors in $B_{i+1}$ is the same. 
			
			So, we could obtain a subpalette $L'$ of $L$ and a set $T_2^4$ with $|T_2^4|\leq 12\times 7+12\times 7+10\times 7=238$ in $O(|V(G)|+|E(G)|)$ time such that $(G,L,Z)$ is colorable if and only if $(G,L',Z)$ is colorable and $T_2^4$ dominates every vertex in a element of $Q_4$ with list length 3 for $L'$.  
			Thus, by classifying all possible colorings of $T_2^4$, there is a set $\mathcal{R}$ of restrictions of $(G,L,Z)$ with size in $O(1)$, such that $(G,L,Z)$ is colorable if and only if $\mathcal{R}$ is colorable, $\mathcal{R}$ can be obtained in $O(|V(G)|+|E(G)|)$ time and there are no vertices in a element of $Q_4$ have list length 3.  
			This completes the proof of Claim \ref{Q_4}.
		\end{proof}
		
		Now we are ready to prove Lemma \ref{Y}. 
	\begin{proof}[Proof of Lemma \ref{Y}.]
		We first apply Claim \ref{Q_4} to obtain a set $\mathcal{R}_1$ of restrictions of $(G,L,Z)$. 
		For each element of $\mathcal{R}_1$, we apply Claim \ref{Q_3}, and let $\mathcal{R}_2$ be the set of all restrictions we obtained from this way. 
		Then for each element of $\mathcal{R}_2$, by classifying all possible colorings of $T_2^1\cup T_2^2\cup T_1^1\cup T_1^2\cup T_1^3$, we obtain a set of restrictions of this element. 
		Let $\mathcal{R}_3$ be the set of all restrictions we obtained this way. 
		Then by Claim \ref{K_2}, \ref{Q_1}, \ref{Q_2}, \ref{Q_3} and \ref{Q_4}, $\mathcal{R}_3$ is the set we seek.
		This completes the proof of Lemma \ref{Y}.
	\end{proof}

	\section{Isolated Vertex in $G\setminus N[C]$}
	Let $G$ be a cleaned graph in $\mathcal{G}_{10,7}$ with an induced odd cycle $C:=v_1-v_2-v_3-v_4-v_5-v_6-v_7-v_1$ and $(G,L,Z)$ be a restriction of $(G,L_0,\emptyset)$. 
	Let $D=V(G)\setminus N[C]$, $X\subseteq D$ be set of vertices that have no neighbors in $D$, and $Y\subseteq D$ be the set of vertices that have some neighbors in $D$. 
	For each $i\in [7]$, let $X_{i}$ be the set of vertices that has some neighbors in $A_{i-2}\cup B_{i-1}$ and some neighbors in $A_{i+2}\cup B_{i+1}$. 
	Let $v$ be a vertex in $X$. 
	Then since $G$ has no induced $C_3,C_5,C_9$, there is an index $i\in [7]$ such that either $N(v)\subseteq A_{i-2}\cup B_{i-1}\cup B_{i+1}\cup A_{i+2}$ or $N(v)\subseteq B_{i-1}\cup A_i\cup B_{i+1}$. 
	If $N(v)\subseteq B_{i-1}\cup A_i\cup B_{i+1}$, then $N(v)\subseteq N(v_i)$ which is a contradiction since $G$ has no comparable pair. 
	So we may assume that for every $v\in X$,  $N(v)\subseteq A_{i-2}\cup B_{i-1}\cup B_{i+1}\cup A_{i+2}$ for some $i\in [7]$. 
	Thus, $X=\bigcup_{i\in [7]} X_{i}$.
	Furthermore, when we define an induced $C_7$ in the graph $G$, we automatically define the corresponding $D$, $X_i$ and $Y$ as described above.

	In this section, our aim is to obtain a set $\mathcal{R}$ of restrictions of $(G,L,\emptyset)$ in polynomial time such that for a fixed index $i\in [7]$, every vertex in $X_i$ has list length at most 2. 
	Formally, we prove the following Lemma.  
	
	\begin{lemma}\label{one set}
		Let $G$ be a cleaned graph in $\mathcal{G}_{10,7}$, $C$ be an induced $C_7$ in $G$ and $L$ be an updated subpalette of $L_0$ with $|L(v)|=1$ for every $v\in V(C)$. 
		Then there is a set $\mathcal{R}$ of restrictions of $(G,L,\emptyset)$ with size $O(|V(G)|^{12})$ such that 
		\begin{itemize}
			\item[(a)] $(G,L,\emptyset)$ is colorable if and only if $\mathcal{R}$ is colorable; 
			\item[(b)] for each element $(G',L',Z')$ of $\mathcal{R}$, $|L(u)|\leq 2$ for every non-reducible vertex $u\in X_i\cap V(G')$ and a fixed index $i\in [7]$; 
			\item[(c)] $\mathcal{R}$ can be constructed in $O(|V(G)|^{11}(|V(G)|+|E(G)|))$ time.  
					\item[(d)] any coloring of $\mathcal{R}$ can be extended to a coloring of  $(G,L,\emptyset)$  in $O(|V(G)|)$ time. 
		\end{itemize} 
	\end{lemma}
	We start with the following Claim.

	\begin{claim}\label{edges in NC}
		Let $v$ be a vertex in $N(C)$ with a neighbor $u\in N(C)$. 
		If $v\in A_i$ for some $i$, then $u\notin A_{i}\cup A_{i\pm 2}\cup B_{i\pm 1}\cup B_{i\pm 3}$. 
		If $v\in B_i$ for some $i$, then $u\notin A_{i\pm 1}\cup A_{i\pm 3}\cup B_i\cup B_{i\pm 2}\cup B_{i+3}$. 
		Moreover, if $v$ has a neighbor $d$ in $G\setminus N[C]$ such that $N(d)\cap N(v_{i+4})\neq \emptyset$, then $u\notin A_{i+4}\cup B_{i+5}$. 
	\end{claim}
	\begin{proof}
		We first assume that $v\in A_i$. 
		If $u\in A_i$, then $v-u-v_i-v$ is a $C_3$. 
		If $u\in A_{i\pm 2}$, say $A_{i+2}$, then $v-v_{i}-v_{i+1}-v_{i+2}-u-v$ is an induced $C_5$. 
		If $u\in B_{i\pm 1}$, say $B_{i+1}$, then $v-u-v_i-v$ is a $C_3$. 
		If $u\in B_{i\pm 3}$, say $B_{i+3}$, then $v-v_{i}-v_{i+1}-v_{i+2}-u-v$ is an induced $C_5$. 
		The proof for  $v\in B_i$ is similar. 
		We then assume that $v$ has a neighbor $d$ in $G\setminus N[C]$ such that there is a neighbor $d'$ of $d$ in $N(v_{i+4})$ and $u\in A_{i+4}\cup B_{i+5}$, say $A_{i+4}$. 
		Since $G$ has no $C_3$, $ud\notin E(G)$. 
		Then $v-d-d'-v_{i+4}-u-v$ is an induced $C_5$. 
		This completes the proof of Claim \ref{edges in NC}. 
	\end{proof}

	We call a coloring $c$ of $G$ good with respect to $(C,i)$ if $c(v_{i-2}), c(v_{i-1}), c(v_{i+2})$ are pairwise different and $c(v_{i+1})=c(v_{i-1})$. 
	Next we distinguish between several types of good coloring of $G$. 
	We call a good coloring $c$ of $G$  
	\begin{itemize}
		\item[(A)] a type $A$ coloring with respect to $(C,i)$ if 
		
		\begin{itemize}
			\item[$\bullet$] Case 1: there is a set of vertices $T=\{a_1,a_2,a_2',w_1,w_2,w_3\}$ with $a_{1}\in A_{i-2}\cup B_{i-1}$, $w_1,w_2\in X_{i}$, $a_{2},a_{2}'\in A_{i+2}\cup B_{i+1}$ and $w_3\in D$ such that $a_{1}-w_1-a_{2}-w_2-a_{2}'-w_3$ is an induced $P_6$ in $G$ with $c(a_{1})=c(v_{i+2})$, $c(w_1)=c(w_2)=c(w_3)=c(v_{i+1})$ and $c(a_{2})=c(a_{2}')=c(v_{i-2})$, or the symmetric case with the roles of $+$ and $-$ reversed;
			\item[$\bullet$] Case 2: there is a set of vertices $T=\{a_1,a_2,w_1\}$ with $a_{1}\in A_{i-2}\cup B_{i-3}$, $w_1\in D$ and $a_{2}\in A_{i+1}\cup B_{i+2}$ such that $a_{1}-w_1-a_{2}$ is an induced $P_3$ in $G$ with $c(a_1)=c(v_{i+2})$, $c(a_{2})=c(v_{i-2})$ and $c(w_1)=c(v_{i+1})$, or the symmetric case with the roles of $+$ and $-$ reversed; 
			\item[$\bullet$] Case 3: there is a set of vertices $T=\{a_1,a_1',a_2,w_1,w_2\}$ with $a_{1},a_1'\in A_{i-2}\cup B_{i-1}$, $w_1,w_2\in D$ and $a_{2}\in A_{i+3}$ such that $a'_{1}-w_1-a_1-w_2-a_{2}$ is an induced $P_5$ in $G$ with $c(a_{2})=c(w_1)=c(v_{i-2})$, $c(w_2)=c(a_1')=c(v_{i+1})$ and $c(a_1)=c(v_{i+2})$, or the symmetric case with the roles of $+$ and $-$ reversed; 
			\item[$\bullet$] Case 4: there is a set of vertices $T=\{a_1,a_1',a_2,w_1\}$ with $a_{1},a_1'\in A_{i-2}\cup B_{i-1}$, $w_1\in D$ and $a_{2}\in A_{i-3}$ such that $a'_{1}-w_1-a_1-a_{2}$ is an induced $P_4$ in $G$ with $c(w_1)=c(v_{i-2})$, $c(a_{2})=c(a_1')=c(v_{i+1})$ and $c(a_1)=c(v_{i+2})$, or the symmetric case with the roles of $+$ and $-$ reversed; 
			\item[$\bullet$] Case 5: there is a set of vertices $T=\{a_1,a_1',a_1'',w_1,w_2,w_3\}$ with $a_{1},a_1',a_1''\in A_{i-2}\cup B_{i-1}$, $w_1,w_2\in D$ and $w_3\in X_{i-2}$ such that $a_1-w_1-a_1'-w_2-a_1''-w_3$ is an induced $P_5$ in $G$ with $c(w_1)=c(v_{i-2})$, $c(w_1)=c(w_2)=c(a_1)=c(v_{i+1})$ and $c(a_1')=c(a_1'')=c(v_{i+2})$, or the symmetric case with the roles of $+$ and $-$ reversed; 
			\item[$\bullet$] Case 6: there is a set of vertices $T=\{a_1,a_1',a_1'',a_2,w_1,w_2,w_3,w_4\}$ with $a_{1},a_1',a_1''\in A_{i-2}\cup B_{i-1}$, $w_1,w_2,w_3,w_4\in D$ and $a_{2}\in A_{i+2}\cup B_{i+1}$ such that $a_1-w_1$ and $w_1-a_1'-w_2-a_1''-w_3-w_4-a_{2}-w_1$ is an induced $C_7$ in $G$ with $c(w_1)=c(v_{i-2})$, $c(a_1)=c(a_{2})=c(w_3)=c(w_2)=c(v_{i+1})$ and $c(a_1')=c(a_1'')=c(v_{i+2})$, or the symmetric case with the roles of $+$ and $-$ reversed; 
		\end{itemize}

		\item[(B)] a type $B$ coloring with respect to $(C,i)$ if 
		
		\begin{itemize}
			\item[$\bullet$] Case 1: there is a set of vertices $T=\{a_1,w_1,w_2\}$ with $a_{1}\in A_{i+2}\cup B_{i+1}$, $w_1,w_2\in D$ such that $w_2-w_1-a_{1}$ is an induced $P_3$ in $G$ with $c(a_{1})=c(v_{i-2})$, $c(w_1)=c(v_{i+1})$ and $c(w_2)=c(v_{i+2})$, or the symmetric case with the roles of $+$ and $-$ reversed; 
			\item[$\bullet$] Case 2: there is a set of vertices $T=\{a_1,w_1,w_2\}$ with $a_{1}\in A_{i+2}\cup B_{i+1}$, $w_1,w_2\in D$ such that $w_2-w_1-a_{1}$ is an induced $P_3$ in $G$ with $c(a_{1})=c(w_2)=c(v_{i-2})$, $c(w_1)=c(v_{i+1})$ and $w_2$ has no neighbors in $A_{i-2}\cup B_{i-1}$, or the symmetric case with the roles of $+$ and $-$ reversed; 
			\item[$\bullet$] Case 3: there is a set of vertices $T=\{a_1,w_1,w_2\}$ with $a_{1}\in A_{i+2}\cup B_{i+1}$, $w_1\in A_{i+3}$ and $w_2\in D$ such that $w_2-w_1-a_{1}$ is an induced $P_3$ in $G$ with $c(a_{1})=c(v_{i-2})$, $c(w_1)=c(v_{i+1})$ and $c(w_2)=c(v_{i+2})$, or the symmetric case with the roles of $+$ and $-$ reversed; 
			\item[$\bullet$] Case 4: there is a set of vertices $T=\{a_1,a_1',a_2,w_1,w_2,w_3\}$ with $a_{1},a_1'\in A_{i-2}\cup B_{i-1}$, $a_{2}\in A_{i+2}$ and $w_1,w_2,w_3\in D$ such that $a_{2}x'\in E(G)$ and $a'_{1}-w_1-a_1-w_2-w_3$ is an induced $P_5$ in $G$ with $c(w_1)=c(v_{i-2})$, $c(w_2)=c(a_1')=c(a_{2})=c(v_{i+1})$ and $c(a_1)=c(v_{i+2})$, or the symmetric case with the roles of $+$ and $-$ reversed. 
		\end{itemize}
	\end{itemize}
	
		In Lemma \ref{A type}, we prove that if $G$ has a type $A$ coloring for some index $i\in [7]$, then there is a set with finite order such that if we fix the colors of all vertices in this set, then every vertex in $X_i$ will have list length at most 2 after updating. 
		In Lemma \ref{B type}, we prove that if $G$ has a type $B$ coloring for some index $i\in [7]$, then there is a set with finite order such that if we fix the colors of all vertices in this set, there is no $2P_3$ in the set of vertices with list length 3, where $2P_3$ is a special structure which we will define later. 
	Moreover, we call a coloring type $(A,j)$ coloring if it is from the case $j$ of type $A$ coloring defined above and type $(B,j)$ coloring if it is from the case $j$ of type $B$ coloring defined above.
	
	\begin{lemma}\label{A type}
		Suppose that $G$ has a type $A$ coloring $c$ with respect to $(C,i)$. 
		Let $L$ be an updated subpalette of $L_0$ with $L(v)=\{c(v)\}$ for every $v\in T\cup \{v_{i-2},v_{i-1},v_{i+1},v_{i+2}\}$.
		Then $|L(x)|\leq 2$ for every vertex $x$ in $X_{i}$. 
	\end{lemma}
	\begin{proof}
		Suppose to the contrary there is a vertex $x\in X_{i}$ with $|L(x)|=3$. 
		Let $\bar{a}_{i-2}$ be a neighbor of $x$ in $A_{i-2}\cup B_{i-1}$ and $\bar{a}_{i+2}$ be a neighbor of $x$ in $A_{i+2}\cup B_{i+1}$. 
		If case 1, then by Claim \ref{edges in NC}, $\bar{a}_{i+2}-x-\bar{a}_{i-2}-v_{i-2}-a_1-w_1-a_{2}-w_2-a_{2}'-w_3$ is an induced $P_{10}$. 
		If case 2, then by Claim \ref{edges in NC}, $\bar{a}_{i+2}-x-\bar{a}_{i-2}-v_{i-2}-a_1-w_1-a_{2}-v_{i+1}-v_{i+2}-\bar{a}_{i+2}$ is an induced $C_{9}$.
		If case 3, then since $\bar{a}_{i-2}-x-\bar{a}_{i+2}-v_{i+2}-v_{i+3}-a_{2}-w_2-a_1-w_1-a_1'$ is not an induced $P_{10}$, $\bar{a}_{i-2}w_1\in E(G)$. 
		Hence, $\bar{a}_{i+2}-x-\bar{a}_{i-2}-w_1-a_1-w_2-a_{2}-v_{i+3}-v_{i+2}-\bar{a}_{i+2}$ is an induced $C_{9}$. 
		If case 4, then since $\bar{a}_{i-2}-x-\bar{a}_{i+2}-v_{i+2}-v_{i+3}-v_{i+4}-a_{2}-a_1-w_1-a_1'$ is not an induced $P_{10}$, $\bar{a}_{i-2}w_1\in E(G)$. 
		Hence, $\bar{a}_{i+2}-x-\bar{a}_{i-2}-w_1-a_1-a_{2}-v_{i+4}-v_{i+3}-v_{i+2}-\bar{a}_{i+2}$ is an induced $C_{9}$. 
		If case 5, then since $\bar{a}_{i+2}-x-\bar{a}_{i-2}-v_{i-2}-a_1-w_1-a_1'-w_2-a_1''-w_3$ is not an induced $P_{10}$, $\bar{a}_{i-2}w_1\in E(G)$. 
		On the other hand, since $w_3\in X_{i-2}$, $w_3$ has a neighbor $a_{i-4}$ in $A_{i-4}\cup B_{i-3}$. 
		Then $x-\bar{a}_{i-2}-w_1-a_1'-w_2-a_1''-w_3-a_{i-4}-v_{i-4}-v_{i-3}$ is an induced $P_{10}$. 
		If case 6, then since $\bar{a}_{i+2}-x-\bar{a}_{i-2}-v_{i-2}-a_1''-w_3-w_4-a_{2}-v_{i+2}-\bar{a}_{i+2}$ is not an induced $C_{9}$, $\bar{a}_{i+2}w_4\in E(G)$. 
		On the other hand, since $a_1-w_1-a_1'-w_2-a_1''-w_3-w_4-\bar{a}_{i+2}-x-\bar{a}_{i-2}$ is not an induced $P_{10}$, $\bar{a}_{i-2}w_1\in E(G)$. 
		Then $w_1-a_1'-w_2-a_1''-w_3-w_4-\bar{a}_{i+2}-x-\bar{a}_{i-2}-w_1$ is an induced $C_9$.
	\end{proof}
	Let $U$ be a subset of $X_{i}$ for some $i\in [7]$. 
	A {\em $kP_3$ matching} with respect to $U$ is an induced $kP_3$ such that for each $P_3$, its interior vertex belongs to $U$, and its two ends belong to $A_{i-2}\cup B_{i-1}$ and $A_{i+2}\cup B_{i+1}$, respectively. 
	
	\begin{lemma}\label{B type}
		Suppose that $G$ has a type $B$ coloring $c$ with respect to $(C,i)$. 
		Let $L$ be an updated subpalette of $L_0$ with $L(v)=\{c(v)\}$ for every $v\in T\cup \{v_{i-2},v_{i-1},v_{i+1},v_{i+2}\}$ and $\bar{X}_{i}$ be the set of those vertices in $X_{i}$ have list length $3$. 
		Then there is no $2P_3$ matching with respect to $\bar{X}_{i}$. 
	\end{lemma}
	\begin{proof}
		Suppose to the contrary that there are induced $u_1-x_1-u_1', u_2-x_2-u_2'$ with $u_1,u_2\in A_{i-2}\cup B_{i-1}$ and $u_1',u_2'\in A_{i+2}\cup B_{i+1}$ as a 2$P_3$ matching with respect to $\bar{X}_{i}$. 
		If case 1 ,2 or 3, then $w_2-w_1-a_{1}-v_{i+2}-u_2'-x_2-u_2-v_{i-2}-u_1-x_1$ is an induced $P_{10}$.
		If case 4, then since $w_3-w_2-a_1-v_{i-2}-u_1-x_1-u_1'-v_{i+2}-u_2'-x_2$ is not an induced $P_{10}$, either $w_3u_1'\in E(G)$ or $w_3u_2'\in E(G)$, say $w_3u_1'\in E(G)$. 
		On the other hand, since $x_2-u_2-v_{i-2}-a_1'-w_1-a_{i+2}-v_{i+2}-u_1'-w_3-w_2$ is not an induced $P_{10}$, $u_1w_1\in E(G)$. 
		Then $x_2-u_2-w_1-a_1-w_2-w_3-u_1'-v_{i+2}-v_{i+3}-v_{i+4}$ is an induced $P_{10}$. 
	\end{proof}

	For a 3-coloring $c$ of $G$, $x\in X_i$ satisfies {\em mono condition} with respect to $(C,c)$ if all neighbors of $x$ in $A_{i-2}\cup B_{i-1}$ are assigned the same color in $c$, and all neighbors of $x$ in $A_{i+2}\cup B_{i+1}$ are also assigned the same color in $c$. 
	Our aim is to find a coloring $c$ such that every vertex in $X_i$ satisfies mono condition with respect to $(C,c)$ for a fixed index $i\in [7]$. 
	However, this may be not attained for $C$. 
	Hence, we construct a new induced odd cycle $C'$ to make every vertex in $X_i\cap X_i'$ satisfy mono condition with respect to $(C',c)$. We use $A',B', A_i'$ and $B'$ to denote $A^{C'}$, $B^{C'}$, $A_i^{C'}$ and $B_i^{C'}$.
	Formally, we prove the following Lemma. 
	
	\begin{lemma}\label{wonderful}
		Let $G$ be a cleaned graph in $\mathcal{G}_{10,7}$ with $C=v_1-v_2-v_3-v_4-v_5-v_6-v_7-v_1$, $C'=v_i'-v_{i+1}'-v_{i+2}-v_{i+3}-v_{i-3}-v_{i-2}-v_{i-1}'-v_i'$ be two induced $C_7$ in $G$ with $v_{i-1}'\in A_{i-2}\cup B_{i-1}$, $v_i'\in X_i$, $v_{i+1}'\in A_{i+2}\cup B_{i+1}$ for some $i\in [7]$. 
		Suppose that $G$ has no type $A$ coloring or type $B$ coloring with respect to $(C',i)$. 
		Let $c$ be a good 3-coloring of $G$, $L$ be an updated subpalette of $L_0$ with $L(v)=\{c(v)\}$ for every $v\in V(C)\cup V(C')$ and $\bar{X}_i$ be the set of those vertices in $X_{i}\cap X'_{i}$ that have list length 3.
		Then there is a 3-coloring $c'$ of $G$ such that for each $v\in V(C)\cup V(C')$, $c'(v)=c(v)$ and for each $x\in \bar{X}_i$, $x$ satisfies mono condition with respect to $(C',c')$. 
	\end{lemma}
	\begin{proof}
		Without loss of generality, let $c(v_{i-2})=1,c(v'_{i-1})=c(v'_{i+1})=2$ and $c(v_{i+2})=3$. 
		First, we may assume that there is a vertex $x$ in $\bar{X}_i$ which does not satisfy mono condition with respect to $(C',c)$. 
		So, by the definition of mono condition, there are two neighbors of $x$ in $A'_{i-2}\cup B'_{i-1}$ or in $A'_{i+2}\cup B'_{i+1}$, say in $A'_{i-2}\cup B'_{i-1}$, assigned two different colors in $c$. 
		Let $r$ be a neighbor of $x$ in $A'_{i-2}\cup B'_{i-1}$ and $a'_{i+2}$ be a vertex in $N(x)\cap (A'_{i+2}\cup B'_{i+1})$ with $c(r)=c(a'_{i+2})=2$.  
		Define $A^x_{i-2}=\{v\in (A'_{i-2}\cup B'_{i-1})\cap N(x)|c(v)=3\}$ and $X^{x}_{i}=\{v\in X'_{i}\cap N(A^x_{i-2})|c(v)=2\}$. 
		For every vertex $v\in X^x_{i}$, let
		\begin{itemize}
			\item $S_{1.1}(v)$ be the set of those vertices in $(A'_{i+2}\cup B'_{i+1})\cap N(v)$ with color 1 in $c$; 
			\item $S_{1.2}(v)$ be the set of those vertices in $(X'_i\cap N(S_{1.1}(v)))\setminus X^x_{i}$ that are assigned color 2 in $c$; 
			\item $S_{1.3}(v)$ be the set of those vertices in $((A'_{i+2}\cup B'_{i+1})\cap N(S_{1.2}(v)))\setminus S_{1.1}(v)$ with color 1 in $c$; 
			\item $S_{2.1}(v)$ be the set of those vertices in $((A'_{i-2}\cup B'_{i-1})\cap N(v))\setminus A^x_{i-2}$ with color 3 in $c$; 
			\item $S_{2.2}(v)$ be the set of those vertices in $(X'_i\cap N(S_{2.1}(v)))\setminus X^x_{i}$ that are assigned color 2 in $c$. 
		\end{itemize} 
		Let $S_1(v)=S_{1.1}(v)\cup S_{1.2}(v)\cup S_{1.3}(v)$, $S_2(v)=S_{2.1}(v)\cup S_{2.2}(v)$ and $T(C)=\{v_{i-1},v_{i},v_{i+1}\}$. 
		We then construct a 3-coloring $c_x$ of $G$ from $c$. 
		To illustrate $c_x$, we use the Table \ref{table:1} to represent how we obtain $c_x$ from $c$ and Figure \ref{fig:table set} to represent some sets in Table \ref{table:1}. 
		\begin{table}[h]
			\centering
			\begin{tabular}{c |c |c} 
				\hline
				Set Name & Definition & Color assigned in $c_x$ \\ 
				\hline
				$A^x_{i-2}$ & $\{v\in (A'_{i-2}\cup B'_{i-1})\cap N(x)|c(v)=3\}$ & 2 \\ 
				\hline 
				$X^x_{i-2}$ & $\{v\in X'_{i-2}\cap N(A^x_{i-2})|c(v)=2\}$ & 1 \\ 
				\hline
				$X^x_{i+3}$ & $\{v\in X'_{i+3}\cap N(A^x_{i-2})|c(v)=2\}$ & 1 \\ 
				\hline
				$X^{x,+}_{i}$ & $\{v\in X'_{i}\cap N(A^x_{i-2})|c(v)=2,S_1(v)\cap T=\emptyset\}$ & 1 \\ 
				\hline
				$X^{x,-}_{i}$ & $\{v\in X'_{i}\cap N(A^x_{i-2})|c(v)=2,S_1(v)\cap T\neq \emptyset,S_2(v)\cap T=\emptyset\}$ & 3 \\ 
				\hline
				$A^{x,+}_{i+2}$ & $\{v\in (A'_{i+2}\cup B'_{i+1})\cap N(X^{x,+}_{i})|c(v)=1\}$ & 2 \\  
				\hline
				$A^{x,+}_{i+3}$ & $\{v\in A'_{i+3}\cap N(A^{x,+}_{i+2})|c(v)=2\}$ & 3 \\  
				\hline
				$X^{x,+}_{i-3}$ & $\{v\in X'_{i-3}\cap N(A^{x,+}_{i+2})|c(v)=2\}$ & 3 \\
				\hline
				$X^{x,+}_{i+2}$ & $\{v\in X'_{i+2}\cap N(A^{x,+}_{i+2})|c(v)=2\}$ & 3 \\
				\hline
				$Y^{x,+}$ & $\{v\in Y'\cap N(A^{x,+}_{i+2})|c(v)=2\}$ & 3 \\
				\hline
				$X^{x,++}_{i}$ & $\{v\in (X'_{i}\cap N(A^{x,+}_{i+2}))\setminus (X^{x,+}_{i}\cup X^{x,-}_{i})|c(v)=2\}$ & 1 \\
				\hline
				$A^{x,++}_{i+2}$ & $\{v\in ((A'_{i+2}\cup B'_{i+1})\cap N(X^{x,++}_{i}))\setminus A^{x,+}_{i+2}|c(v)=1\}$ & 2 \\ 
				\hline
				$A^{x,-}_{i-2}$ & $\{v\in ((A'_{i-2}\cup B'_{i-1})\cap N(X^{x,-}_{i}))\setminus A^x_{i-2}|c(v)=3\}$ & 2 \\  
				\hline
				$X^{x,--}_{i}$ & $\{v\in (X'_{i}\cap N(A^{x,-}_{i-2}))\setminus (X^{x,+}_{i}\cup X^{x,-}_{i}\cup X^{x,++}_{i})|c(v)=2\}$ & 3 \\
				\hline
				$X^{x,-}_{i+3}$ & $\{v\in (X'_{i+3}\cap N(A^{x,-}_{i-2}))\setminus X^{x}_{i+3}|c(v)=2\}$ & 1 \\
				\hline
			\end{tabular}
			\caption{Table to illustrate $c_x$ different from $c$}
			\label{table:1}
		\end{table}
		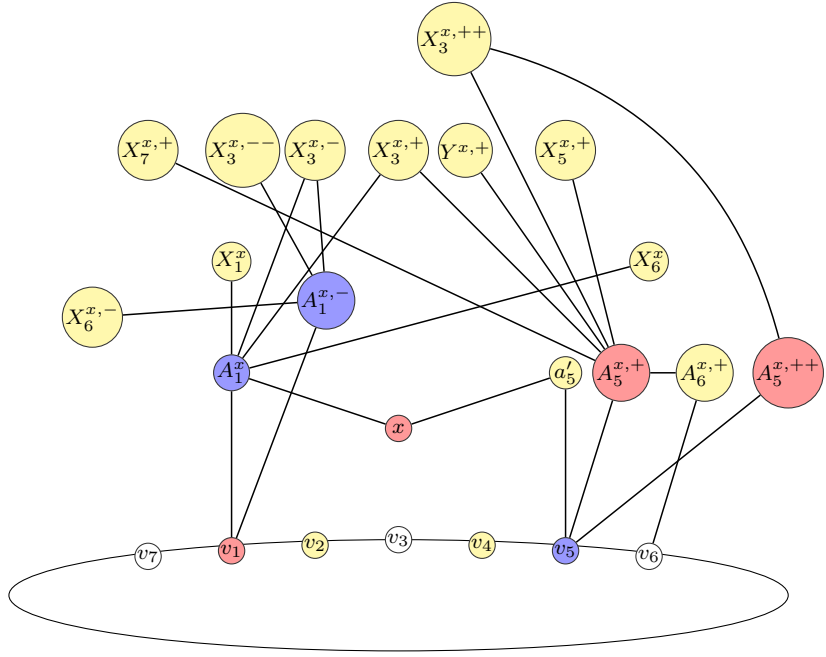
\begin{figure}[h!]
			\centering
			\begin{tikzpicture}[scale=.74]

				\draw (0,0) ellipse (7 and 1);
				
				\node[circle, draw=black!80, inner sep=0mm, minimum size=3.5mm, fill =white] (a_{3}) at (0,1){\scriptsize $v_3$};
				\node[circle, draw=black!80, inner sep=0mm, minimum size=3.5mm, fill=yellow!40] (a_{4}) at (1.5,.9){\scriptsize $v_4$};
				\node[circle, draw=black!80, inner sep=0mm, minimum size=3.5mm, fill=blue!40] (a_{5}) at (3,.8){\scriptsize $v_5$};
				\node[circle, draw=black!80, inner sep=0mm, minimum size=3.5mm, fill=white] (a_{6}) at (4.5,.7){\scriptsize $v_6$};
				\node[circle, draw=black!80, inner sep=0mm, minimum size=3.5mm, fill=white] (a_{7}) at (-4.5,.7){\scriptsize $v_7$};
				\node[circle, draw=black!80, inner sep=0mm, minimum size=3.5mm, fill=red!40] (a_{1}) at (-3,.8){\scriptsize $v_1$};
				\node[circle, draw=black!80, inner sep=0mm, minimum size=3.5mm, fill=yellow!40] (a_{2}) at (-1.5,.9){\scriptsize $v_2$};
				
				\node[circle, draw=black!80, inner sep=0mm, minimum size=3.5mm, fill=blue!40] (a_{8}) at (-3,4){\scriptsize $A^x_{1}$};
				\node[circle, draw=black!80, inner sep=0mm, minimum size=3.5mm, fill=yellow!40] (a_{11}) at (-3,6){\scriptsize $X^x_{1}$};
				\node[circle, draw=black!80, inner sep=0mm, minimum size=3.5mm, fill=yellow!40] (a_{12}) at (4.5,6){\scriptsize $X^x_{6}$};
				\node[circle, draw=black!80, inner sep=0mm, minimum size=3.5mm, fill=yellow!40] (a_{13}) at (0,8){\scriptsize $X^{x,+}_{3}$};
				\node[circle, draw=black!80, inner sep=0mm, minimum size=3.5mm, fill=red!40] (a_{14}) at (4,4){\scriptsize $A^{x,+}_{5}$};
				\node[circle, draw=black!80, inner sep=0mm, minimum size=3.5mm, fill=yellow!40] (a_{15}) at (5.5,4){\scriptsize $A^{x,+}_{6}$};
				\node[circle, draw=black!80, inner sep=0mm, minimum size=3.5mm, fill=yellow!40] (a_{16}) at (-4.5,8){\scriptsize $X^{x,+}_{7}$};
				\node[circle, draw=black!80, inner sep=0mm, minimum size=3.5mm, fill=yellow!40] (a_{17}) at (3,8){\scriptsize $X^{x,+}_{5}$};
				\node[circle, draw=black!80, inner sep=0mm, minimum size=3.5mm, fill=yellow!40] (a_{18}) at (1.2,8){\scriptsize $Y^{x,+}$};
				\node[circle, draw=black!80, inner sep=0mm, minimum size=3.5mm, fill=yellow!40] (a_{19}) at (1,10){\scriptsize $X^{x,++}_{3}$};
				\node[circle, draw=black!80, inner sep=0mm, minimum size=3.5mm, fill=red!40] (a_{20}) at (7,4){\scriptsize $A^{x,++}_{5}$};
				\node[circle, draw=black!80, inner sep=0mm, minimum size=3.5mm, fill=yellow!40] (a_{9}) at (3,4){\scriptsize $a_5'$};
				\node[circle, draw=black!80, inner sep=0mm, minimum size=3.5mm, fill=red!40] (a_{10}) at (0,3){\scriptsize $x$};
				
				\node[circle, draw=black!80, inner sep=0mm, minimum size=3.5mm, fill=yellow!40] (a_{21}) at (-1.5,8){\scriptsize $X^{x,-}_{3}$};
				\node[circle, draw=black!80, inner sep=0mm, minimum size=3.5mm, fill=blue!40] (a_{22}) at (-1.3,5.3){\scriptsize $A^{x,-}_{1}$};
				\node[circle, draw=black!80, inner sep=0mm, minimum size=3.5mm, fill=yellow!40] (a_{23}) at (-2.8,8){\scriptsize $X^{x,--}_{3}$};
				\node[circle, draw=black!80, inner sep=0mm, minimum size=3.5mm, fill=yellow!40] (a_{24}) at (-5.5,5){\scriptsize $X^{x,-}_{6}$};
				
				\draw[line width=0.2mm, black] (a_{8}) to  (a_{1});
				\draw[line width=0.2mm, black] (a_{8}) to  (a_{10});
				\draw[line width=0.2mm, black] (a_{9}) to  (a_{10});
				\draw[line width=0.2mm, black] (a_{9}) to  (a_{5});
				
				\draw[line width=0.2mm, black] (a_{8}) to  (a_{11});
				\draw[line width=0.2mm, black] (a_{8}) to  (a_{12});
				\draw[line width=0.2mm, black] (a_{8}) to  (a_{13});
				
				\draw[line width=0.2mm, black] (a_{13}) to  (a_{14});
				
				\draw[line width=0.2mm, black] (a_{14}) to  (a_{5});
				\draw[line width=0.2mm, black] (a_{14}) to  (a_{15});
				\draw[line width=0.2mm, black] (a_{14}) to  (a_{16});
				\draw[line width=0.2mm, black] (a_{14}) to  (a_{17});
				\draw[line width=0.2mm, black] (a_{14}) to  (a_{18});
				\draw[line width=0.2mm, black] (a_{14}) to  (a_{19});

				\draw[line width=0.2mm, black] (a_{15}) to  (a_{6});
				\draw[line width=0.2mm, black] (a_{20}) to  (a_{5});
				\draw[bend left=30, line width=0.2mm, black] (a_{19}) to  (a_{20});
				
				\draw[line width=0.2mm, black] (a_{8}) to  (a_{21});
				\draw[line width=0.2mm, black] (a_{21}) to  (a_{22});
				\draw[line width=0.2mm, black] (a_{1}) to  (a_{22});
				\draw[line width=0.2mm, black] (a_{22}) to  (a_{23});
				\draw[line width=0.2mm, black] (a_{22}) to  (a_{24});


			\end{tikzpicture}
			\caption{Illustration of sets in Table \ref{table:1} when $i=3$ (every vertex in red set is colored 1 in $c$, every vertex in yellow set is colored 2 in $c$, every vertex in blue set is colored 3 in $c$ and the color of every white vertex is uncertain).}

			\label{fig:table set}
		\end{figure}
		Thus, by the definition of $c_x$ and all neighbors of $x$ in $A'_{i-2}\cup B'_{i-1}$ and $A'_{i+2}\cup B'_{i+1}$ are assigned color 2 in $c_x$. 
		Hence, $x$ satisfies mono condition with respect to $(C',c_x)$. 
		On the other hand, if there is a vertex $v$ in $V(C)\cup V(C')$ such that $c_x(v)\neq c(v)$, then $v\in T(C)\cap (A_{i-2}^x\cup X_{i-2}^x\cup X_{i+3}^x\cup X_i^{x,+}\cup X_i^{x,-})$. 
		If there is a vertex $v\in T(C)\cap A_{i-2}^x$, then $v=v_{i-1}$, $L(v_{i-1})=\{3\}$ and hence, $|L(x)|\leq 2$, a contradiction. 
		If there is a vertex $v\in T(C)\cap (X_{i-2}^x\cup X_{i+3}^x\cup X_i^{x,+}\cup X_i^{x,-})$, then $v=v_{i}$, $v\in X^{x,+}_{i}\cup X^{x,-}_{i}$ and $L(v)=\{2\}$. 
		Let $v'$ be a neighbor of $v$ in $A_{i-2}^x$. 
		Since $L$ is updated,  $L(v')=\{3\}$ and hence, $|L(x)|\leq 2$, a contradiction. 
		It implies that, for each $v\in V(C)\cup V(C')$, we have $c_x(v)=c(v)$.

		\begin{claim}\label{proper}
			$c_x$ is a 3-coloring of $G$. 
		\end{claim}
		\begin{proof}
			Let $S$ be the set of all vertices that have different colors in $c$ and $c_x$. 
			Suppose to the contrary that $c_x$ is a coloring. 
			So there is a vertex $u\in S$ with a neighbor $v$ such that $c_x(u)=c_x(v)$. 
			
			\vspace{0.2cm}
			\noindent (1) $v\notin S$. 
			\vspace{0.2cm}
			
			If $c_x(u)=c_x(v)=1$, then both $u$ and $v$ are in $X'$, contrary to the fact that $u$ and $v$ are adjacent. 
			If $c_x(u)=c_x(v)=2$, then $u,v\in A'_{i-2}\cup B'_{i-1}\cup A'_{i+2}\cup B'_{i+1}$. 
			By Claim \ref{edges in NC}, both $u$ and $v$ in $A_{i-2}^x$, and hence, $x-u-v-x$ is an induced $C_3$. 
			So we may assume that $c_x(u)=c_x(v)=3$. 
			Thus, $u,v\in X_i^{x,+}\cup A_{i+3}^{x,+}\cup X_{i-3}^{x,+}\cup X_{i+2}^{x,+}\cup Y^{x,+}\cup X_i^{x,--}$. 
			If $u$ or $v$ in $X_i^{x,+}\cup X_{i-3}^{x,+}\cup X_{i+2}^{x,+}\cup X_i^{x,--}$, say $u$, then $u$ has no neighbors in $X_i^{x,+}\cup A_{i+3}^{x,+}\cup X_{i-3}^{x,+}\cup X_{i+2}^{x,+}\cup Y^{x,+}\cup X_i^{x,--}$. 
			Thus, both $u,v$ are in $A_{i+3}^{x,+}\cup Y^{x,+}$. 
			If $u$ or $v$ in $A_{i+3}^{x,+}$, say $u\in A_{i+3}^{x,+}$, then $v\in Y^{x,+}$. 
			Hence, let $v'$ be a neighbor of $v$ in $A_{i+2}^{x,+}$. 
			Then $v-v'-v_{i+2}-v_{i+3}-u-v$ is an induced $C_5$. 
			So, we may assume that both $u$ and $v$ are in $Y^{x,+}$. 
			By the definition of $Y^{x,+}$, let $u',v'$ be the neighbors of $u,v$ in $A_{i+2}^{x,+}$, respectively. 
			Then $u-u'-v_{i+2}-v'-v-u$ is an induced $C_5$ or contains an induced $C_3$ if $u'=v'$. 
			This completes the proof of (1).

			\vspace{0.2cm}
			\noindent (2) $u\notin X^{x,--}_{i}\cup X^{x,-}_{i+3}$. 
			\vspace{0.2cm}
			
			Suppose that $u\in X^{x,--}_{i}\cup X^{x,-}_{i+3}$. 
			If $u\in X^{x,-}_{i+3}$, then $c_x(u)=c_x(v)=1$, and hence, $v\in A'_{i+1}\cup B'_{i+2}$ since $c_x(v_{i-2})=1$.  
			Then $c_x$ is a type $(A,2)$ coloring with respect to $(C',i)$. 
			So we may assume that $u\in X^{x,--}_{i}$. 
			Hence, $c_x(u)=c_x(v)=3$ and $v\in A_{i-2}\cup B_{i-1}$. 
			Let $a^{x,-}_{i-2}\in A^{x,-}_{i-2}$, $x^{x,-}_i\in X^{x,-}_i$ and $a_{i-2}^x\in A_{i-2}^x$ such that $x-a_{i-2}^x-x^{x,-}_i-a^{x,-}_{i-2}-u$ is an induced $P_5$ in $G$. 
			By Claim \ref{edges in NC}, since $v_{i-3}-v_{i+3}-v_{i+2}-a'_{i+2}-x-a_{i-2}^x-x^{x,-}_i-a^{x,-}_{i-2}-u-v$ is not an induced $P_{10}$, either $vx\in E(G)$ or $vx^{x,-}_i\in E(G)$. 
			Hence, $v\in S$, which is contrary to (1).   
			This completes the proof of (2). 
			
			\vspace{0.2cm}
			\noindent (3) $u\notin A^{x,-}_{i-2}$. 
			\vspace{0.2cm}
			
			Suppose that $u\in A^{x,-}_{i-2}$. 
			So $c(v)=c_x(v)=c_x(u)=2$. 
			Let $x^{x,-}_i\in X^{x,-}_i$ and $a_{i-2}^x\in A_{i-2}^x$ such that $x-a_{i-2}^x-x^{x,-}_i-u$ is an induced $P_4$ in $G$. 
			Since $c(v)=2$, $v\notin S$ and by Claim \ref{edges in NC}, $v\in A'_{i-3}\cup X'_{i-2}\cup Y'$.  
			We first assume that $v\in A'_{i-3}$. 
			Since $x-a_{i-2}^x-x^{x,-}_i-u-v-v_{i-3}-v_{i+3}-v_{i+2}-a'_{i+2}-x$ is not an induced $C_9$, $va_{i-2}^x\in E(G)$, and hence $c_x$ is a type $(A,4)$ coloring with respect to $(C',i)$.
			We then assume that $v\in X'_{i-2}$. 
			Then $c_x$ is a type $(A,5)$ coloring with respect to $(C',i)$. 
			So we may assume that $v\in Y'$. 
			By the definition of $Y'$, there is a neighbor $s$ of $v$ in $Y'$. 
			Since $s-v-u-x^{x,-}_i-a_{i-2}^x-x-a'_{i+2}-v_{i+2}-v_{i+3}-v_{i-3}$ is not an induced $P_{10}$, $sa'_{i+2}\in E(G)$. 
			Then $c_x$ is a type $(A,6)$ coloring with respect to $(C',i)$.
			This completes the proof of (3). 
			
			\vspace{0.2cm}
			\noindent (4) $u\notin A^{x,++}_{i+2}$
			\vspace{0.2cm}
			
			Suppose that $u\in A^{x,++}_{i+2}$. 
			So $c(v)=c_x(v)=c_x(u)=2$. 
			Let $x_i^{x,++}\in X_i^{x,++}$, $a_{i+2}^{x,+}\in A_{i+2}^{x,+}$, $x_i^{x,+}\in X_i^{x,+}$ and $a_{i-2}^x\in A_{i-2}^x$ such that $x-a_{i-2}^x-x_i^{x,+}-a_{i+2}^{x,+}-x_i^{x,++}-u$ is an induced $P_6$ in $G$. 
			Since $c_x(v)=2$ and by Claim \ref{edges in NC}, $v\in A'_{i+3}\cup D'$.   
			If $v\in D'$, then $c_x$ is a type $(A,1)$ coloring with respect to $(C',i)$. 
			So we may assume that $v\in A'_{i+3}$. 
			Then, $v_{i-3}-v_{i+3}-v-u-x_i^{x,++}-a_{i+2}^{x,+}-x_i^{x,+}-a_{i-2}^x-x-r$ is an induced $P_{10}$. 
			This completes the proof of (4). 
			
			\vspace{0.2cm}
			\noindent (5) $u\notin X^{x,++}_{i}\cup Y^{x,+}\cup X^{x,+}_{i+2}\cup X^{x,+}_{i-3}\cup A^{x,+}_{i+3}$. 
			\vspace{0.2cm}
			
			Suppose that $u\in X^{x,++}_{i}\cup Y^{x,+}\cup X^{x,+}_{i+2}\cup X^{x,+}_{i-3}\cup A^{x,+}_{i+3}$. 
			Let $a_{i+2}^{x,+}\in A_{i+2}^{x,+}$, $x_i^{x,+}\in X_i^{x,+}$ and $a_{i-2}^x\in A_{i-2}^x$ such that $x-a_{i-2}^x-x_i^{x,+}-a_{i+2}^{x,+}-u$ is an induced $P_5$ in $G$. 
			We first assume that $u\in X^{x,++}_{i}$.
			Then $c_x(u)=1$ and $u\in X'_{i}$, and hence, $c(v)=1$ and $v\in A'_{i+2}\cup B'_{i+1}$. 
			That is, $v\in A^{x,++}_{i+2}\subseteq S$, which is contrary to (1). 
			We then assume that $u\in Y^{x,+}\cup X^{x,+}_{i+2}\cup X^{x,+}_{i-3}$. 
			Since $c(v)=c_x(u)=3$, $v\in A'_{i-3}\cup A'_{i-2}\cup A'_{i-1}\cup A'_i\cup B'_{i-3}\cup B'_{i-2}\cup B'_{i-1}\cup D'$. 
			Moreover, if $v\in A'_{i-2}\cup B'_{i-3}\cup B'_{i-1}\cup D'$, then $u\in Y^{x,+}$. 
			By the definition of $Y'$, there is a neighbor $s\in D'$ of $u$. 
			If $c(s)=3$, then $c_x$ is a type $(B,1)$ coloring with respect to $(C',i)$. 
			So $v\notin D'$. 
			If $c(s)=1$ and $s$ has no neighbors in $A'_{i-2}\cup B'_{i-1}$, then $c_x$ is a type $(B,2)$ coloring with respect to $(C',i)$. 
			Thus, we have $c(s)=1$ and $s$ has a neighbor in $A'_{i-2}\cup B'_{i-1}$. 
			Hence, if $u\in Y^{x,+}$, then $u$ has no neighbors in $A'_{i-2}\cup B'_{i-1}\cup B'_{i-3}$. 
			In a word, if $u\in Y^{x,+}\cup X^{x,+}_{i+2}\cup X^{x,+}_{i-3}$, $v\in A'_{i-3}\cup A'_{i-1}\cup A'_i\cup B'_{i-2}$. 
			If $v\in A'_{i-1}\cup B'_{i-2}$, then $v_{i-2}-v'_{i-1}-v-u-a_{i+2}^{x,+}-v_{i+2}-a'_{i+2}-x-a_{i-2}^x-v_{i-2}$ is an induced $C_9$ in $G$. 
			If $v\in A'_{i-3}$, then $v_{i-2}-v_{i-3}-v-u-a_{i+2}^{x,+}-v_{i+2}-a'_{i+2}-x-a_{i-2}^x-v_{i-2}$ is an induced $C_9$ in $G$.
			If $v\in A'_i$, then $c(v'_i)=1$ and hence, $a_{i+2}^{x,+}\in A'_{i+2}$. 
			Then $u-v-v'_i-v'_{i-1}-v_{i-2}-v_{i-3}-v_{i+3}-v_{i+2}-a_{i+2}^{x,+}-u$ is an induced $C_9$. 
			Thus, we may assume that $u\in A^{x,+}_{i+3}$. 
			Since $c(v)=3$ and by Claim \ref{edges in NC}, $v\in A'_{i-3}\cup A'_{i-1}\cup A'_{i}\cup B'_{i-2}\cup D'$. 
			If $v\in A'_{i-1}\cup B'_{i-2}$, then $v_{i-2}-v'_{i-1}-v-u-a_{i+2}^{x,+}-v_{i+2}-a'_{i+2}-x-a_{i-2}^x-v_{i-2}$ is an induced $C_9$ in $G$. 
			If $v\in A'_{i-3}$, then $v_{i-2}-v_{i-3}-v-u-a_{i+2}^{x,+}-v_{i+2}-a'_{i+2}-x-a_{i-2}^x-v_{i-2}$ is also an induced $C_9$ in $G$.
			If $v\in A'_i$, then since $c(v)=3$, $c(v'_i)=1$. 
			Hence, $a'_{i+2},a_{i+2}^{x,+}\in A'_5$. 
			Then $v_{i-3}-v_{i-2}-v'_{i-1}-v'_{i}-v-u-a_{i+2}^{x,+}-v_{i+2}-a'_{i+2}-x$ is an induced $P_{10}$. 
			If $v\in D'$, then $c_x$ is a type $(B,3)$ coloring with respect to $(C',i)$. 
			This completes the proof of (5). 
			
			\vspace{0.2cm}
			\noindent (6) $u\notin A^{x,+}_{i+2}$. 
			\vspace{0.2cm}
			
			Suppose that $u\in A^{x,+}_{i+2}$. 
			So $c(v)=c_x(v)=c_x(u)=2$. 
			Let $x_i^{x,+}\in X_i^{x,+}$ and $a_{i-2}^x\in A_{i-2}^x$ such that $x-a_{i-2}^x-x_i^{x,+}-u$ is an induced $P_4$ in $G$.
			Since $c(v)=2$ and by Claim \ref{edges in NC}, $v\in Y'\cup X'_{i}\cup X'_{i+2}\cup X'_{i-3}\cup A'_{i+3}$. 
			Hence, $v\in S$, a contradiction.
			This completes the proof of (6). 
			
			\vspace{0.2cm}
			\noindent (7) $u\notin X_i^{x,-}\cup X_i^{x,+}\cup X^x_{i+3}\cup X^x_{i-2}$. 
			\vspace{0.2cm}
			
			Suppose that $u\in X_i^{x,-}\cup X_i^{x,+}\cup X^x_{i+3}\cup X^x_{i-2}$. 
			We first assume that $u\in X_i^{x,+}\cup X^x_{i+3}\cup X^x_{i-2}$. 
			So $c(v)=c_x(v)=c_x(v)=1$. 
			Let $a_{i-2}^x\in A_{i-2}^x$ such that $x-a_{i-2}^x-v$ is an induced $P_3$ in $G$. 
			Since $c(v)=1$ and $v\notin S$, $v\in A'_i\cup A'_{i+1}\cup A'_{i+3}\cup B'_{i+2}$. 
			If $v\in A'_{i+1}\cup B'_{i+2}$, then $c_x$ is a type $(A,2)$ coloring with respect to $(C',i)$. 
			If $v\in A'_{i+3}$, then $c_x$ is a type $(A,3)$ coloring with respect to $(C',i)$. 
			If $v\in A'_i$, then since $c(v)=1$, $c_x(v'_i)=3$. 
			Hence, $a_{i-2}^x\in A'_{i-2}$. 
			Then $u-v-v'_{i}-v'_{i+1}-v_{i+2}-v_{i+3}-v_{i-3}-v_{i-2}-a_{i-2}^x-u$ is an induced $C_9$. 
			So we may assume that $u\in X_i^{x,-}$
			Hence, $c(v)=c_x(v)=c_x(u)=3$ and $v\in A'_{i-2}\cup B'_{i-1}$. 
			Thus, $v\in A_{i-2}^{x,-}\subseteq S$, contrary to (1). 
			This completes the proof of (7).  
			
			\vspace{0.2cm}
			\noindent (8) $u\notin A_{i-2}^x$. 
			\vspace{0.2cm} 
			
			Suppose that $u\in A_{i-2}^x$. 
			So $c(v)=c_x(v)=c_x(u)=2$. 
			Since $c(v)=2$, $v\notin X^x_{i+3}\cup X^x_{i-2}$ and by Claim \ref{edges in NC}, $v\in A'_{i-3}\cup X'_i\cup Y'$. 
			If $v\in A'_{i-3}$, then $c_x$ is a type $(A,4)$ coloring with respect to $(C',i)$. 
			If $v\in Y'$, then $c_x$ is a type $(B,4)$ coloring with respect to $(C',i)$. 
			So we may assume that $v\in X'_i$. 
			Since $v\notin X_i^{x,-}\cup X_i^{x,+}$, $S_1(v)\neq \emptyset$ and $S_2(v)\neq \emptyset$. 
			That is, either $T(C)\cap S_{1.1}(v)\neq \emptyset$ or $T(C)\cap S_{1.2}(v)\neq \emptyset$ or $T(C)\cap S_{1.3}(v)\neq \emptyset$. 
			
			We first assume that $T\cap S_{1.1}(v)\neq \emptyset$. 
			By the definition of $S_1(v)$, we have $v\in X'_i$, $T(C)\cap S_{1.1}(v)=\{v_{i+1}\}$ and $c(v_{i+1})=1$. 
			On the other hand, since $S_2(v)\cap T\neq \emptyset$, either $S_2(v)\cap T\cap S_{2.1}(v)\neq \emptyset$ or $S_2(v)\cap T\cap S_{2.2}(v)\neq \emptyset$. 
			If $T\cap S_{2.1}(v)\neq \emptyset$, then $T(C)\cap S_{2.1}(v))=\{v_{i-1}\}$. 
			Hence, $c(v_{i-1})=3$. 
			Since $v$ is adjacent to both $v_{i-1}$ and $v_{i+1}$, $L(v)=\{2\}$, $L(u)=\{3\}$. 
			Hence, $|L(x)|\leq 2$, a contradiction. 
			If $T(C)\cap S_{2.2}(v)\neq \emptyset$, then $T(C)\cap S_{2.2}(v))=\{v_{i}\}$ and hence, $c(v_{i})=2$ and there is a vertex $v'$ in $N(v_{i-2})\cap N(v_i)\cap N(v)$. 
			Since $c(v_{i-2})=1$ and $c(v_{i-2})=2$, then $L(v')=\{3\}$, $L(v)=\{2\}$, $L(u)=\{3\}$ and hence, $|L(x)|\leq 2$, a contradiction. 
			
			We then assume that $T\cap S_{1.2}(v)\neq \emptyset$. 
			By the definition of $S_{1.2}(v)$, we have $v\in X'_i$, $T\cap S_{1.2}(v)=\{v_{i}\}$, $c(v_{i})=2$ and there is vertex $v'$ in $N(v)\cap N(v_{i+2})\cap N(v_{i})$. 
			Hence, we have $L(v')=\{1\}$.
			If $T(C)\cap S_{2.1}(v)\neq \emptyset$, then $T(C)\cap S_{2.1}(v))=\{v_{i-1}\}$ and $L(v_{i-1})=\{3\}$. 
			So, $L(v)=\{2\}$, $L(u)=\{3\}$ and hence, $|L(x)|\leq 2$, a contradiction.  
			If $T(C)\cap S_{2.2}(v)\neq \emptyset$, then $T(C)\cap S_{2.2}(v))=\{v_{i}\}$ and hence, $L(v_{i})=\{2\}$ and there is a vertex $v'$ in $N(v_{i-2})\cap N(v_i)\cap N(v)$. 
			Since $L(v_{i-2})=\{1\}$ and $L(v_{i-2})=\{2\}$, then $L(v')=\{3\}$, $L(v)=\{2\}$, $L(u)=\{3\}$ and hence, $|L(x)|\leq 2$, a contradiction.
			
			So we may assume that $T(C)\cap S_{1.3}(v)\neq \emptyset$. 
			By the definition of $S_{1.3}(v)$, we have $v\in X'_i$, $T(C)\cap S_{1.3}(v)=\{v_{i+1}\}$, $L(v_{i+1})=\{1\}$ and there are vertices $v'\in A'_{i+2}\cup B'_{i+1}$ and $v''\in X'_i$ with $L(v')=\{1\}$, $L(v'')=\{2\}$ such that $x-u-v-v'-v''-v_{i+1}$ is an induced $P_6$ in $G$. 
			Since either $T(C)\cap S_{2.1}(v)\neq \emptyset$ or $T(C)\cap S_{2.2}(v)\neq \emptyset$, either $v_{i-1}\in S_2(v)$ with $L(v_{i-1})=\{3\}$ and hence, $L(v_{i})=\{2\}$, or $v_{i}\in S_2(v)$ with $L(v_{i})=\{2\}$. 
			Thus, $L(v_{i})=\{2\}$, and hence, $v_i\in X'_i$. 
			Then $c_x$ is a type $(A,1)$ coloring with respect to $(C',i)$. 
			This completes the proof of (8).
			By (1),(2),(3),(4),(5),(6),(7) and (8), this completes the proof of Claim \ref{proper}.
		\end{proof}

		Then $c_x$ is a 3-coloring of $G$ such that $x$ satisfies mono condition with respect to $(C',c_x)$ and for each $v\in V(C)\cup V(C')$, $c_x(v)=c(v)$. 
		By replacing the original coloring with newly derived one, we ultimately obtain the desired coloring of $G$. 
		Since in each step, at least one vertex in $A'_{i-2}\cup B'_{i-1}\cup B'_{i+1}\cup A'_{i+2}$ with color changed to 2 and if a vertex is changed to satisfy mono condition, then it will always satisfy mono condition in the later step, this progress will terminate. 
		This completes the proof of Lemma \ref{wonderful}. 	
	\end{proof}

    Now we are ready to prove Lemma \ref{one set}. 
    Note that for any 3-coloring $c$ of $G$, either every vertex in $X_i$ satisfies mono condition with respect to $(C,c)$ or there is a vertex $x$ that is not satisfy. 
    For the former case, we construct a restriction of $(G,L,\emptyset)$. 
    In the later case, $x$ and its two neighbors in $A_{i-2}\cup B_{i-1}$ and $A_{i+2}\cup B_{i-+1}$ join with $C$ construct a new induced odd cycle $C'$. 
    By guessing $c$ is a type $A$ or $B$ coloring with respect to $C'$, we construct two sets of restrictions of $(G,L,\emptyset)$,  respectively. 
    If $G$ is 3-colorable but has no type $A$ or $B$ coloring with respect to $C'$, then by Lemma \ref{wonderful}, there must be an coloring of $G$ such that every vertex in $\bar{X}_i$ satisfies mono condition with respect to $C'$. 
    Thus, we construct another set of restrictions of $(G,L,\emptyset)$. 
    Let $\mathcal{R}$ be the set of all restrictions we obtained. 
    Then $R$ is colorable if and only if $(G,L,\emptyset)$ is colorable. 
	\begin{proof}[Proof of Lemma \ref{one set}]
		If $L(v_{i-2})=L(v_{i+2})$, then every vertex in $X_i$ with list length 3 is reducible, contrary to the condition that $L$ is non-reducible for $C$. 
		So we may assume that $L(v_{i-2})\neq L(v_{i+2})$, say $L(v_{i-2})=\{1\}$ and $L(v_{i+2})=\{2\}$. 
		Let $R_1=(G\setminus \bar{X}_i,L,Z_1)$, where $\bar{X}_i$ is the set of those vertices in $X_i$ with list length 3 in $L$ and $Z_1=\bigcup_{x\in \bar{X}_i}\{N(x)\cap (A_{i-2}\cup B_{i-1}),N(x)\cap (A_{i+2}\cup B_{i+1})\}$. 
		Note that any coloring of $R_1$ can be extended to a coloring of $(G,L,\emptyset)$ in $O(|V(G)|)$ time.

		\vspace{0.2cm}
		\noindent (1) Construction of $\mathcal{R}_1$. 
		\vspace{0.2cm}

		For every $x\in \bar{X}_i$, $a_{i-2}\in N(x)\cap (A_{i-2}\cup B_{i-1})$, $a_{i+2}\in N(x)\cap (A_{i+2}\cup B_{i+1})$, we construct a palette $L'=L_{x,a_{i-2},a_{i+2}}$, depending on $x,a_{i-2}$ and $a_{i+2}$. 
		Define $L'$ by setting $L'(x)=L(v_{i-2})$, $L'(a_{i-2})=L'(a_{i+2})=L(v_{i-1})$ and leaving $L'(v)=L(v)$ for every $v\in V(G')\setminus \{x,a_{i-2},a_{i+2}\}$. 
		For the symmetry case, define $L'=L_{x,a_{i-2},a_{i+2}}$ by setting $L'(x)=L(v_{i+2})$, $L'(a_{i-2})=L'(a_{i+2})=L(v_{i-1})$ and leaving $L'(v)=L(v)$ for every $v\in V(G')\setminus \{x,a_{i-2},a_{i+2}\}$. 
		The set of $\mathcal{L}_1$ will be the set of all updated palettes from $L'$ obtained in this way. 
		So the number of palettes in $\mathcal{L}_1$ is $O(|V(G)|^3)$. 
		Let $\mathcal{R}_1=R_1\cup (G,\mathcal{L}_1,\emptyset)$. 
		Note that $(G,L,\emptyset)$ is colorable if and only if $R_1\cup (G,\mathcal{L}_1,\emptyset)$ is colorable. 
		The total time for constructing $\mathcal{R}_1$ thus amounts to $O(|V(G)|^3(|V(G)|+|E(G)|))$.

		\vspace{0.2cm}
		Let $(G,L',\emptyset)$ be a element of $(G,\mathcal{L}_1,\emptyset)$, where $L'$ depending on $x,a_{i-2}$ and $a_{i+2}$. 
		Set $C'=v_{i-3}-v_{i-2}-a_{i-2}-x-a_{i+2}-v_{i+2}-v_{i+3}-v_{i-3}$. 
		Note that $C'$ is also an induced $C_7$ in $G$ and any coloring of $(G,L',\emptyset)$ is a good coloring with respect to $(C',i)$. 
		
		\vspace{0.2cm}
		\noindent (2) Construction of $\mathcal{R}_{2.1}$. 
		\vspace{0.2cm}
		
		For every $T$ as claimed in a type $(A,j)$ coloring with respect to $(C',i)$, we construct a palette $L''=L_{T,j}$ depending on $T$ and $j$.
		Define $L''$ by setting $L''(v)$ equal to the color of $v$ admits as in a type $(A,j)$ coloring for every $v\in T$ and leaving $L''(v)=L'(v)$ for every remaining vertex $v$. 
		The set of $\mathcal{L}_{2.1}$ will be the set of all updated palettes from $L''$ obtained in this way such that no vertex in $X_i$ has list length 3 in the updated palette.    
		So the number of palettes in $\mathcal{L}_{2.1}$ is $O(|V(G)|^{11})$, since $|T|\leq 8$. 
		Let $\mathcal{R}_{2.1}=(G,\mathcal{L}_{2.1},\emptyset)$. 
		Note that $(G,L',\emptyset)$ admits a type $(A,j)$ coloring with respect to $(C',i)$ if and only if $(G,\mathcal{L}_{2.1},\emptyset)$ is colorable. 
		The total time for constructing $\mathcal{R}_{2.1}$ thus amounts to $O(|V(G)|^{11}(|V(G)|+|E(G)|))$.

		\vspace{0.2cm}
		\noindent (3) Construction of $\mathcal{R}_{2.2}$. 
		\vspace{0.2cm}
		
		For every $T$ as claimed in a type $(B,j)$ coloring with respect to $(C',i)$, we construct a palette $L''=L_{T,j}$ depending on $T$ and $j$.
		Define $L''$ by setting $L''(v)$ equal to the color of $v$ admits as in a type $(B,j)$ coloring for every $v\in T$ and leaving $L''(v)=L'(v)$ for every remaining vertex $v$. 
		The set of $\mathcal{L}_{2.2}$ will be the set of all updated palettes from $L''$ obtained in this way such that there are no 2$P_3$ with respect to the set of vertices in $X_i$ that have list length 3.   
		So the number of palettes in $\mathcal{L}_{2.2}$ is $O(|V(G)|^9)$, since $|T|\leq 6$. 
		Let $\mathcal{R}_{2.2}=(G,\mathcal{L}_{2.2},\emptyset)$. 
		Note that $(G,L',\emptyset)$ admits a type $(B,j)$ coloring with respect to $(C',i)$ if and only if $(G,\mathcal{L}_{2.2},\emptyset)$ is colorable. 
		The total time for constructing $\mathcal{R}_{2.2}$ thus amounts to $O(|V(G)|^9(|V(G)|+|E(G)|))$.

		\vspace{0.2cm}
		\noindent (4) Construction of $\mathcal{R}'_{2.2}$. 
		\vspace{0.2cm}
		
		For every $L''\in \mathcal{L}_{2.2}$, we construct a palette $L'''=L_{x'}$, depending on $x'$, where $x'$ is a vertex in the set of vertices that have list length 3 in $L''$. 
		Define $L'''_1$ by setting $L'''_1(x')=3$ and leaving $L'''_1(v)=L''(v)$ for every remaining vertex $v$. 
		The set of $\mathcal{L}_{2.2.1}$ will be the set of all updated palettes from $L'''_1$ obtained in this way.
		Since there are no 2$P_3$ with respect to the set of vertices in $X_i$ that have list length 3 in $L''$, no vertex in $X_i$ has list length 3 in every element of $\mathcal{L}_{2.2.1}$. 
		Define $L'''_2$ by setting $L'''_2(x')=L''(x')\setminus \{3\}$ and leaving $L'''_2(v)=L''(v)$ for every remaining vertex $v$.
		The set of $\mathcal{L}_{2.2.2}$ will be the set of all updated palettes from $L'''_2$ obtained in this way. 
		Thus, no vertex in $X_i$ has list length 3 in every element of $\mathcal{L}_{2.2.2}$. 
		Set $\mathcal{R}'_{2.2}=(G,\mathcal{L}_{2.2.1}\cup \mathcal{L}_{2.2.2},\emptyset)$. 
		Then $\mathcal{R}_{2.2}$ is colorable if and only if $\mathcal{R}'_{2.2}$ is colorable. 
		The size of $\mathcal{R}'_{2.2}$ is $O(|V(G)|^{10})$ and the total time for constructing $\mathcal{R}'_{2.2}$ thus amounts to $O(|V(G)|^{10}(|V(G)|+|E(G)|))$.

		\vspace{0.2cm}
		\noindent (5) Construction of $\mathcal{R}$. 
		\vspace{0.2cm}
		
		Let $R_2=(G\setminus \bar{X}'_i,L',Z_2)$, where $\bar{X}'_i$ is the set of those vertices in $X_i$ with list length 3 in $L'$ and $Z_2=\bigcup_{x\in \bar{X}'_i}\{N(x)\cap (A_{i-2}\cup B_{i-1}),N(x)\cap (A_{i+2}\cup B_{i+1})\}$. 
		Note that any coloring of $R_2$ can be extended to a coloring of $(G,L',\emptyset)$ in $O(|V(G)|)$ time. 
		Let $\mathcal{R}'$ the union of $R_2$, $\mathcal{R}_{2.1}$ and $\mathcal{R}'_{2.2}$ we obtain for every $L'\in \mathcal{L}_1$. 
		By Lemma \ref{wonderful}, $(G,\mathcal{L}_1,\emptyset)$ is colorable if and only if $\mathcal{R}'$ is colorable. 
		Let $\mathcal{R}=R_1\cup \mathcal{R}'$. 
		So, $(G,L,\emptyset)$ is colorable if and only if $\mathcal{R}$ is colorable, for each element of $\mathcal{R}$, $|L(u)|\leq 2$ for every $u\in X_i$, since the size of $\mathcal{R}$ is  $O(|V(G)|^{12})$, the total time for constructing $\mathcal{R}$ thus amounts to $O(|V(G)|^{11}(|V(G)|+|E(G)|))$ and any coloring of $\mathcal{R}$ can be extended to a coloring of $(G,L,\emptyset)$ in $O(|V(G)|)$ time.  
		This completes the proof of Lemma \ref{one set}. 
	\end{proof}

	\section{Proof of Main Theorem}
	For any induced odd cycle $C$ of $G$ and $\{i,j,k\}=\{1,2,3\}$, we call $C$ admit a  
	\begin{itemize}
		\item type \uppercase\expandafter{\romannumeral1} coloring if $C$ is colored as $i-j-k-i-k-j-k-i$; 
		\item type \uppercase\expandafter{\romannumeral2} coloring if $C$ is colored as $i-j-k-j-k-i-k-i$; 
		\item type \uppercase\expandafter{\romannumeral3} coloring if $C$ is colored as $i-j-k-j-k-j-k-i$.
	\end{itemize}
	Now, we are ready to prove our main Theorem. 
	\begin{proof}[Proof of Theorem \ref{Main}]
		It is sufficient to prove for the case when $G$ is connected. 
		By Lemma \ref{comparable} and Lemma \ref{cleaned}, it spends in $O(|V(G)|^7(|V(G)|+|E(G)|))$ time to reduce $G$ to $G'$ with no comparable pair and cleaned if $G$ is not bipartite or determine that $G$ is bipartite. 
		If $G$ is bipartite, a 2-coloring of $G$ can be obtained in $O(|V(G)|+|E(G)|)$ time. 
		So, we may assume that $G$ is not bipartite.

		\begin{claim}\label{C7 type}
			Let $c$ be a 3-coloring of $G$. 
			Then every induced $C_7$ in $G$ admits either a type \uppercase\expandafter{\romannumeral1}, type \uppercase\expandafter{\romannumeral2} or type \uppercase\expandafter{\romannumeral3} coloring. 
		\end{claim}
		\begin{proof}
			Let $C_7:=v_1-v_2-v_3-v_4-v_5-v_6-v_7-v_1$ and $\{i,j,k\}=\{1,2,3\}$. 
			Since $\chi(C_7)=3$, there is one element in $\{i,j,k\}$, say $k$, such that there are exactly three vertices in $C_7$ colored $k$. 
			Without loss of generality, we may assume that $c(v_3)=c(v_5)=c(v_7)=k$. 
			On the other hand, either there are exactly one vertex in $C_7$ colored $i$ and exactly three vertices in $C_7$ colored $j$ or there are exactly two vertices in $C_7$ colored $i$ and exactly two vertices in $C_7$ colored $j$. 
			If the former case happens, then we may assume that $c(v_1)=i$ and $c(v_2)=c(v_4)=c(v_6)=j$, and hence $C_7$ admits a type \uppercase\expandafter{\romannumeral3} coloring. 
			So we may assume that the latter case happens. 
			Then by symmetry, we may assume that $c(v_1)=i$ and $c(v_2)=j$. 
			It follows that either $c(v_4)=i$ and $c(v_6)=j$ or $c(v_4)=j$ and $c(v_6)=i$. 
			Then $C_7$ admits a type \uppercase\expandafter{\romannumeral1} or \uppercase\expandafter{\romannumeral2} coloring. 
			This completes the proof of Claim \ref{C7 type}. 
		\end{proof}

		For every induced odd cycle $C$, we construct a set $\mathcal{L}_1^C$ of all palettes that give $C$ a type \uppercase\expandafter{\romannumeral1} coloring and update. 
		The set $\mathcal{L}_1$  will be the set of all palettes we obtained this way. 
		We define $\mathcal{L}_2$, similarly, for all type \uppercase\expandafter{\romannumeral2} colorings. 
		
		\vspace{0.2cm}
		\noindent(1) Construct a set $\mathcal{R}_1$ of restrictions of $(G',\mathcal{L}_1,\emptyset)$ such that $(G',\mathcal{L}_1,\emptyset)$ is colorable if and only if $\mathcal{R}_1$ is colorable. 
		\vspace{0.2cm}
		
			Let $C:=v_1-v_2-v_3-v_4-v_5-v_6-v_7-v_1$ be an induced odd cycle of $G'$ and $L_1=L_{C}$, depending on $C$, be an element of $\mathcal{L}_1$. 
			Without loss of generality, we may assume that $L_1(v_1)=L_1(v_4)=\{1\}$, $L_1(v_2)=L_1(v_6)=\{2\}$, $L_1(v_3)=L_1(v_5)=L_1(v_7)=\{3\}$. 
			Let $\bar{X}_i$ be the set of those vertices in $X_{i}$ that have color list length 3 and are not reducible. 
			Then $\bar{X}_4\cup \bar{X}_5\cup \bar{X}_6=\emptyset$.
			So, we focus on the sets $\bar{X}_7,\bar{X}_1,\bar{X}_2$ and $\bar{X}_3$. 
			Note that by the colors of vertices in $C$, $\bar{X}_1$ and $\bar{X}_2$ are symmetrical, $\bar{X}_7$ and $\bar{X}_3$ are symmetrical and all vertices in $\bar{X}_1$ have no neighbors in $B_7\cup B_2$, all vertices in $\bar{X}_2$ have no neighbors in $B_1\cup B_3$, all vertices in $\bar{X}_3$ have no neighbors in $B_2$. 
			
			Suppose first that $\bar{X}_1\cup \bar{X}_2\neq \emptyset$. 
			By symmetry, we may assume that $\bar{X}_1\neq \emptyset$. 
			Let $x_1\in \bar{X}_1$ with neighbors $a_6$ in $A_6$ and $a_3$ in $A_3$. 
			If $\bar{X}_2\neq \emptyset$, let $x_2\in \bar{X}_2$ with neighbors $a_7$ in $A_7$ and $a_4$ in $A_4$, then $x_1-a_6-v_6-v_7-v_1-v_2-v_3-v_4-a_4-x_2$ is an induced $P_{10}$ in $G$. 
			If $\bar{X}_3\neq \emptyset$, let $x_3\in \bar{X}_3$ with neighbors $a_1$ in $A_1$, then $x_2-a_1-v_1-v_2-v_3-a_3-x_1-a_6-v_6-v_5$ is an induced $P_{10}$ in $G$. 
			It follows that $\bar{X}_2\cup \bar{X}_3=\emptyset$. 
			If $\bar{X}_7=\emptyset$, by Lemma \ref{one set}, we obtain a set $\mathcal{R}_1'$ of  restrictions of $(G',\mathcal{L}_1,\emptyset)$. 
			Then it spends us in $O(|V(G)|)$ time to reduce every element of $\mathcal{R}_1'$ to obtain $\mathcal{R}_1''$ such that every palette in every element of $\mathcal{R}_1'$ is non-reducible. 
			By applying Lemma \ref{Y}, we obtain a set $\mathcal{R}_1^1$ of restrictions of $\mathcal{R}_1''$ such that every vertex in every element of  $\mathcal{R}_1^1$ has list length at most 2. 
			Thus, $\mathcal{R}_1^1$ has size $O(|V(G)|^{18})$ and the total time for constructing $\mathcal{R}_1^1$ amounts to $O(|V(G)|^{18}(|V(G)|+|E(G)|))$.
			So, we may assume that $\bar{X}_7\neq \emptyset$. 
			
			Note that if $(G',L_1,\emptyset)$ is colorable, then either there is a coloring $c$ of $(G',L_1,\emptyset)$ such that there is a vertex $x\in \bar{X}_1$ with two neighbors $a_{6}\in A_{6}\cup B_7$ and $a_{3}\in A_3\cup B_2$ such that $c(a_6)=c(a_3)=1$ or for every coloring $c$ of $(G',L_1,\emptyset)$, every vertex in $\bar{X}_1$ satisfies mono condition with respect to $(C,c)$. 
			
			For every $a_{6}\in A_{6}\cup B_7$ and $a_{3}\in A_3\cup B_2$ such that $a_6$ and $a_3$ have a common neighbor $x$ in $\bar{X}_1$, we construct a palette $L'=L_{a_6,a_3}$, depending on $a_6$ and $a_{3}$. 
			Define $L'$ by setting $L'(a_6)=L'(a_{3})=\{1\}$ and leaving $L'(v)=L_1(v)$ for every $v\in V(G')\setminus \{x,a_{i-2},a_{i+2}\}$. 
			Let $L''$ be the palette updated from $L'$. 
			Note that $C':=x-a_3-v_3-v_4-v_5-v_6-a_6-x$ is also an induced odd cycle. 
			Then for every vertex $x'$ in $\bar{X}_7$ that is not in a non-trivial component of $G'\setminus N[C']$, at least one of its neighbors is adjacent to $a_3$ or $a_6$. 
			Since $L''$ is updated, $|L''(x')|\leq 2$. 
			It follows that one can obtain $\mathcal{R}_1^{2,1}$ by applying Lemma \ref{one set}(1) and Lemma \ref{one set} for $i=1$ such that every non-reducible vertex for $C$ in $X$ that is not in a non-trivial component of $G'\setminus N[C']$ has list length at most 2 in any element of $\mathcal{R}_1^{2,1}$. 
			By reducing every element of $\mathcal{R}_1^{2,1}$ such that every palette in every element of $\mathcal{R}_1^{2,1}$ is non-reducible and Lemma \ref{Y} for $C$ and $C'$, one can obtain a set $\mathcal{R}_1^{2,1,1}$ of restrictions of $(G',\mathcal{L}_1,\emptyset)$. 
			Then we apply Lemma \ref{one set} for $i=7$ to obtain a set $\mathcal{R}_1^{2,2}$ of restrictions of $(G',L_1,\emptyset)$. 
			For every element $(G',L',Z')$ in $\mathcal{R}_1^{2,2}$, let $(G'\setminus \bar{X}_1, L',Z'\cup Z)$ be a restriction of $(G',L',Z')$,  where $Z=\bigcup_{x\in \bar{X}_1}\{N(x)\cap (A_{6}\cup B_{7}),N(x)\cap (A_{3}\cup B_{2})\}$. 
			The set of $\mathcal{R}_{1}^{2,2,1}$ will be the set of all restrictions we obtain from this way. 
			Then by reducing every element of $\mathcal{R}_1^{2,2,1}$ such that every palette in every element of $\mathcal{R}_1^{2,1}$ is non-reducible and Lemma \ref{Y} for $C$, one can obtain a set $\mathcal{R}_1^{2,2,2}$ of restrictions of $(G',\mathcal{L}_1,\emptyset)$.
			Let $\mathcal{R}_1^{2}=\mathcal{R}_1^{2,1,1}\cup \mathcal{R}_1^{2,2,2}$. 
			It follows that every vertex in every element of $\mathcal{R}_1^2$ has list length at most 2. 
			Moreover, $\mathcal{R}_1^2$ has size $O(|V(G)|^{21})$ and the total time for constructing $\mathcal{R}_1^2$ amounts to $O(|V(G)|^{21}(|V(G)|+|E(G)|))$.

			Suppose then $\bar{X}_1\cup \bar{X}_2= \emptyset$. 
			If $\bar{X}_3= \emptyset$ or $\bar{X}_7= \emptyset$, then similarly, we obtain a set $\mathcal{R}_1^3$ of restrictions of $(G',\mathcal{L}_1,\emptyset)$ with size $O(|V(G)|^{18})$ in $O(|V(G)|^{11}(|V(G)|+|E(G)|))$ time such that every vertex in every element of  $\mathcal{R}_1^3$ has list length at most 2.  
			So, we then assume that both $\bar{X}_3$ and $\bar{X}_7$ are not empty. 
			Note that if $(G',L_1,\emptyset)$ is colorable, then either there is a vertex $x$ in $\bar{X}_3\cup \bar{X}_7$, say $\bar{X}_3$, has a neighbor $a_1$ such that $c(a_1)=3$ or every neighbors in $A_1$ of every vertex in $\bar{X}_3$ admits color 2 and every neighbors in $A_2$ of every vertex in $\bar{X}_7$ admits color 1. 
			
			Assume first that there is a vertex $x$ in $\bar{X}_3\cup \bar{X}_7$, say $\bar{X}_3$, has a neighbor $a_1$ such that $c(a_1)=3$. 
			Let $a_5$ be a neighbor of $x$ in $A_{5}\cup B_{4}$.  
			Then $C':=x-a_5-v_5-v_6-v_7-v_1-a_1-x$ is an induced odd cycle. 
			It follows that at least one neighbor $a_2\in A_2$ of every vertex $x'\in \bar{X}_7$ that is not in a non-trivial component of $G'\setminus N[C']$ is adjacent to $a_1$. 
			Thus, every vertex in $\bar{X}_7$ does not admit color 1 in $c$. 
			So, we may apply Lemma \ref{one set} for $i=3$ to obtain a set $\mathcal{R}_1^{4,1}$ of restrictions of $(G',\mathcal{L}_1,\emptyset)$. 
			For every element $(G',L',Z')$ of $\mathcal{R}_{1}^{4,1}$, define $(G',L'',Z')$ where $L''(v)=L'(v)\setminus \{1\}$ for every $v\in \bar{X}_7$ such that $v$ is not in a non-trivial component of $G'\setminus N[C']$ and $L''(v)=L'(v)$ for every remaining vertex $v$. 
			The set $\mathcal{R}_1^{4,1,1}$ will be the set of all restrictions we obtain from this way. 
			
			Let $R_1=(G',L',\emptyset)$ be the restriction of $(G',L_1,\emptyset)$, where $L'(v)=L_1(v)\setminus \{2\}$ for every $v\in \bar{X}_3$,$L'(v)=L_1(v)\setminus \{1\}$ for every $v\in \bar{X}_7$ and $L'(v)=L_1(v)$ for every remaining vertex $v$. 
			Let $\mathcal{R}_1^{4,2}=R_1\cup \mathcal{R}_1^{4,1,1}$. 
			Then by reducing every element of $\mathcal{R}_1^{4,2}$ such that every palette in every element of $\mathcal{R}_1^{4,2}$ is non-reducible and Lemma \ref{Y} for $C$ and $C'$, one can obtain a set $\mathcal{R}_1^{4}$ of restrictions of $(G',\mathcal{L}_1,\emptyset)$ such that every vertex in every element of $\mathcal{R}_1^4$ has list length at most 2. 
			Moreover, similar with $\mathcal{R}_1^1$, $\mathcal{R}_1^4$ has size $O(|V(G)|^{18})$ and the total time for constructing $\mathcal{R}_1^4$ amounts to $O(|V(G)|^{18}(|V(G)|+|E(G)|))$.

			Let $\mathcal{R}_1=\mathcal{R}_1^1\cup \mathcal{R}_1^2\cup \mathcal{R}_1^3\cup \mathcal{R}_1^4$. 
			Then $(G,\mathcal{L}_1,\emptyset)$ is colorable if and only if $\mathcal{R}_1$ is colorable, for each element of $\mathcal{R}_1$, $|L(v)|\leq 2$ for every $v\in X$ and $\mathcal{R}_1$ has size $O(|V(G)|^{21})$ and the total time for constructing $\mathcal{R}_1$ amounts to $O(|V(G)|^{21}(|V(G)|+|E(G)|))$. 
			This finishes the work of (1).
			
			\vspace{0.2cm}
			
			\noindent(2) Construct a set $\mathcal{R}_2$ of restrictions of $(G',\mathcal{L}_2,\emptyset)$ such that $(G',\mathcal{L}_2,\emptyset)$ is colorable if and only if $\mathcal{R}_1\cup \mathcal{R}_2$ is colorable. 
		    \vspace{0.2cm}
		    
		    By (1), it follows that for any coloring of $(G',\mathcal{L}_2,\emptyset)$ that is not a coloring of $(G',\mathcal{L}_1,\emptyset)$, there are no induced odd cycle admitting a type $\uppercase\expandafter{\romannumeral1}$ coloring.  
		    Let $C:=v_1-v_2-v_3-v_4-v_5-v_6-v_7-v_1$ be an induced odd cycle of $G'$ and $L_2=L_{C}$, depending on $C$, be an element of $\mathcal{L}_2$. 
		    Without loss of generality, we may assume that $L_2(v_1)=L_2(v_6)=\{1\}$, $L_2(v_2)=L_2(v_4)=\{2\}$ and $L_{2}(v_3)=L
		    _2(v_5)=L_2(v_7)=\{3\}$. 
		    Let $\bar{X}_i$ be the set of those vertices in $X_{i}$ that have color list length 3 and are not reducible. 
		    Then $\bar{X}_5=\emptyset$. 
		    Since no induced odd cycle admits a type $\uppercase\expandafter{\romannumeral1}$ coloring, for any coloring $c$ of $(G',\mathcal{L}_2,\emptyset)$ that is not a coloring of $(G',\mathcal{L}_1,\emptyset)$, $c(v)=2$ for every $v\in N(X_1)\cap (A_3\cup B_2)$, $c(v)=1$ for every $v\in N(X_2)\cap (A_7\cup B_1)$ and $c(v)=3$ for every $v\in (N(X_4)\cap (A_6\cup B_5))\cup (N(X_6)\cap (A_4\cup B_5))$. 
		    So, we only need to focus on $\bar{X}_3$ and $\bar{X}_7$. 
		    Note that in the construction in (1) when $\bar{X}_1=\bar{X}_2=\emptyset$, the lists of $v_4$ and $v_6$ are not used. 
		    So, by symmetry, we could obtain a set $\mathcal{R}_2^{1}$ of restrictions of $(G',\mathcal{L}_2,\emptyset)$ such that for every element $(G'',L',Z')$ of $\mathcal{R}_2^{1}$ and for every vertex $x\in \bar{X}_3\cup \bar{X}_7$ that is not in a non-trivial component of $G''\setminus N[C']$, where $C'$ is the corresponding induced odd cycle of $x$, $|L'(x)|\leq 2$. 
		    Then define $L''$ from $L'$ by setting $L''(v)=\{2\}$ for every $v\in N(X_1)\cap (A_3\cup B_2)$, $L''(v)=\{1\}$ for every $v\in N(X_2)\cap (A_7\cup B_1)$, $L''(v)=\{3\}$ for every $v\in (N(X_4)\cap (A_6\cup B_5))\cup (N(X_6)\cap (A_4\cup B_5))$ and leaving $L''(v)=L'(v)$ for every remaining vertex $v$. 
		    Let $L'''$ be the updated palette from $L''$. 
		    The set $\mathcal{R}_2^{1,1}$ will be the union of all restrictions $(G'',L''',Z')$ we obtained in this way. 
		    Then by reducing every element of $\mathcal{R}_2^{1,1}$ such that every palette in every element of $\mathcal{R}_2^{1,1}$ is non-reducible and Lemma \ref{Y} for $C$ and $C'$, one can obtain a set $\mathcal{R}_2$ of restrictions of $(G',\mathcal{L}_2,\emptyset)$ such that every vertex in every element of $\mathcal{R}_2$ has list length at most 2. 
		    Moreover, $(G',\mathcal{L}_2,\emptyset)$ is colorable if and only if $\mathcal{R}_1\cup \mathcal{R}_2$ is colorable, $\mathcal{R}_2$ has size $O(|V(G)|^{18})$ and the total time for constructing $\mathcal{R}_2$ amounts to $O(|V(G)|^{18}(|V(G)|+|E(G)|))$. 
		    This finishes the work of (2).
		    
		    \vspace{0.2cm}
		    
		    \noindent(3) Construct a set $\mathcal{R}_3$ of restrictions of $(G',L_0,\emptyset)$ such that $(G',L_0,\emptyset)$ is colorable if and only if $\mathcal{R}_1\cup \mathcal{R}_2\cup \mathcal{R}_3$ is colorable. 
		    \vspace{0.2cm}  
		    
		    Let $C:=v_1-v_2-v_3-v_4-v_5-v_6-v_7-v_1$ be an induced odd cycle of $G'$. 
		    For every possible type \uppercase\expandafter{\romannumeral3} coloring $c$ of $C$, define $L=L_{c}$, depending on $c$, by setting $L(v)=\{c(v)\}$ for every $v\in V(C)$ and $L(v)=\{1,2,3\}$ for every $v\in G'\setminus V(C)$. 
		    Let $L'$ be the updated palette from $L$. 
		    The set $\mathcal{R}_3^{1}$ will be the set of all restrictions $(G',L',\emptyset)$ we obtained in this way. 
		    Thus, $\mathcal{R}_3^1$ has size $O(1)$ and can be constructed in $O(1)$ time.     
		    
		    Let $(G',L',\emptyset)$ be an element of $\mathcal{R}_3^1$. 
		    Without loss of generality, we may assume that $L'(v_1)=\{1\}$, $L'(v_2)=L'(v_4)=L'(v_6)=\{2\}$ and $L'(v_3)=L'(v_5)=L'(v_7)=\{3\}$. 
		    By (1) and (2), it follows that for any coloring $c$ of $(G',\mathcal{L}_2,\emptyset)$ that is not a coloring of $(G',\mathcal{L}_1\cup \mathcal{L}_2,\emptyset)$, $c(v)=2$ for every $v\in N(X_3)\cap (A_5\cup B_4)$, $c(v)=3$ for every $v\in N(X_6)\cap (A_4\cup B_5)$ and for every $v\in X_1\cup X_2\cup X_7$, $x$ satisfies mono condition with respect to $c$. 
		    Moreover, since $L'(v_2)=L'(v_6)=2$ and $L'(v_3)=L'(v_4)=3$, every vertex $v\in X_4\cup X_5$ is reducible. 
		    Then let 
		    \begin{align}
		    	Z= & (\bigcup_{x\in X_1}\{N(x)\cap (A_6\cup B_7),N(x)\cap (A_3\cup B_2)\})\nonumber \\ \cup &(\bigcup_{x\in X_2}\{N(x)\cap (A_7\cup B_1),N(x)\cap (A_4\cup B_3)\})\nonumber \\\cup  &(\bigcup_{x\in X_7}\{N(x)\cap (A_5\cup B_6),N(x)\cap (A_2\cup B_1)\}),
		    	\nonumber
		    \end{align}
		    $L''(v)=\{2\}$ for every $v\in N(X_3)\cap (A_5\cup B_4)$, $L''(v)=\{3\}$ for every $v\in N(X_6)\cap (A_4\cup B_5)$ and leaving $L''(v)=L'(v)$ for every remaining vertex $v$. 
		    Define $R_{L'}=(G)$ by setting $(G'\setminus (X_1\cup X_2\cup X_7),L''',Z)$, where $L'''$ is the updated palette from $L''$. 
		    The set $\mathcal{R}_3^1$ be the set of all restrictions $R_{L'}$ we obtained in this way. 
		    Then by reducing every element of $\mathcal{R}_3^{1}$ such that every palette in every element of $\mathcal{R}_3^{1}$ is non-reducible and Lemma \ref{Y} for $C$, one can obtain a set $\mathcal{R}_3$ of restrictions of $(G',L_0,\emptyset)$ such that every vertex in every element of $\mathcal{R}_3$ has list length at most 2. 
		    Moreover, $(G',L_0,\emptyset)$ is colorable if and only if $\mathcal{R}_1\cup \mathcal{R}_2\cup \mathcal{R}_3$ is colorable, $\mathcal{R}_3$ has size $O(1)$ and the total time for constructing $\mathcal{R}_3$ amounts to $O(|V(G)|+|E(G)|)$.
		    This finishes the work of (3). 
		    
		    \vspace{0.2cm}
		    
		    Let $\mathcal{R}=\mathcal{R}_1\cup \mathcal{R}_2\cup \mathcal{R}_3$. 
		    It follows that $\mathcal{R}$ has size $O(|V(G)|^{21})$, the total time for constructing $\mathcal{R}$ amounts to $O(|V(G)|^{21}(|V(G)|+|E(G)|))$, $G$ is 3-colorable if and only if $\mathcal{R}$ is colorable and every vertex $v$ in every element of $\mathcal{R}$ has list length at most 2. 
		    Then by using Lemma \ref{improve 2sat}, we can check in $O(|V(G)|+|E(G)|)$ time whether an element of $\mathcal{R}$ is colorable and find one if it exists. And a coloring of $\mathcal{R}$ can be extended to a coloring of $G$ in $O(|V(G)|+|E(G)|)$ time by Lemmas~\ref{cleaned}, \ref{Y} and \ref{one set}.
		    Since $|\mathcal{R}|$ is $O(|V(G)|^{21})$, and $\mathcal{R}$ can be computed in $O(|V(G)|^{21}(|V(G)|+|E(G)|))$ time, the total running time amounts to $O(|V(G)|^{21}(|V(G)|+|E(G)|))$ time. 
		    This completes the proof.
	\end{proof}
%
%
%
%
%
%
%
	
	\section{Conclusion}
	In this paper, we give a polynomial-time algorithm that solves the 3-coloring problem for the family of $P_{10}$-free graphs whose induced odd cycles are all of length $7$ (we denote this family by $\mathcal{G}_{10,7}$). 
	On the one hand, we make progress towards determining the complexity of 3-coloring in $P_t$-free graphs, which is still open for $t\geq 8$. 
	On the other hand, we explored the class of graphs that have only one prescribed induced odd cycle length, which is a  generalization of bipartite graphs. It would be interesting to determine the complexity of 3-coloring in $\mathcal{G}_{t,7}$ for other integer $t\geq 10$ in this way. 
	Note that in the proof of this paper, we use the properties that $G$ has no comparable pair many times and change the colors of some vertices to find a new coloring in Section 5. 
	Since these properties may not hold, our proof fails for solving the list 3-coloring problem. 
	It would also be interesting to determine the complexity of list 3-coloring in $\mathcal{G}_{10,7}$.

	\section*{Acknowledgement}
	
	\section*{Declaration}
	
	\noindent$\textbf{Conflict~of~interest}$
	The authors declare that they have no known competing financial interests or personal relationships that could have appeared to influence the work reported in this paper.
	
	\noindent$\textbf{Data~availability}$
	Data sharing not applicable to this paper as no datasets were generated or analysed during the current study.


\begin{thebibliography}{99}
		\bibitem{AA23} A. R. Anr\'{i}quez and M. Stein,
		3-Colouring $P_t$-Free Graphs Without Short Odd Cycles,
		{\em Algorithmica} \textbf{85} (2023) 831--853. 
		
		\bibitem{AB79} B. Aspvall, M. F. Plass and R. E. Tarjan,
		A linear-time algorithm for texting the truth of certain quantified boolean formulas,
		{\em Information Processing Letters} \textbf{8} (1979) 121--123. 
		
		\bibitem{BA08} J. A. Bondy, U. S. R. Murty,
		Graph Theory, Springer, London (2008). 
		
		\bibitem{BF18} F. Bonomo, M. Chudnovsky, P. Maceli, O. Schaudt, M. Stein and M. Zhong,
		Three-coloring and list three-coloring of graphs without induced paths on seven vertices,
		{\em Combinatorica} \textbf{38}(4) (2018) 779--801. 
		
		\bibitem{CM06} M. Chudnovsky, N. Robertson, P. Seymour and R. Thomas,
		The strong perfect graph theorem, {\em Annals of Mathematics} \textbf{164} (2006) 51--229.
		
		\bibitem{CM18} M. Chudnovsky and J. Stacho,
		3-coloring subclasses of $P_8$-free graphs,
		{\em SIAM Journal on Discrete Mathematics} \textbf{32}(2) (2018) 1111--1138. 
		
		\bibitem{CM24} M. Chudnovsky, S. Sophie and M. Zhong,
		Four-Coloring $P_6$-Free Graphs. I. Extending an Excellent Precoloring,
		{\em SIAM Journal on Computing} \textbf{53}(1) (2024) 111--145. 
		
		\bibitem{CM241} M. Chudnovsky, S. Sophie and M. Zhong,
		Four-Coloring $P_6$-Free Graphs. II. Finding an Excellent Precoloring,
		{\em SIAM Journal on Computing} \textbf{53}(1) (2024) 146--187. 
		
		\bibitem{EK86} K. Edwards, 
		The complexity of colouring problems on dense graphs,
		{\em Theoretical Computer Science} \textbf{43} (1986) 337--343. 
		
		\bibitem{EP90} P. Erd\H{o}s,  
		Some of my favourite unsolved problems,
		In: A. Baker, et al.: A tribute to Paul Erd\H{o}s, pp. 467, Cambridge: Cambridge University Press 1990. 
		
		\bibitem{EP79} P. Erd\H{o}s, A. Rubin and H. Taylor,  
		Choosability in graphs,
		{\em Congressus Numerantium} \textbf{26} (1979) 125--157. 
		
		\bibitem{HP13} P. Hell and S. Huang,   
		Complexity of coloring graphs without paths and cycles,
		{\em Discrete Applied Mathematics} \textbf{216}(1) (2017) 211--232. 
		
		\bibitem{HC10} C. T. Ho\`{a}ng, M. Kami\'{n}ski, V. V. Lozin, J. Sawada and X. Shu,   
		Deciding $k$-colorability of $P_5$-free graphs in polynomial time,
		{\em Algorithmica} \textbf{57} (2010) 74--81.
		
		\bibitem{HS13} S. Huang,   
		Improved complexity results on $k$-coloring $P_t$-free graphs,
		in: {\em Proceedings of the International Symposium on Mathematical Foundations of Computer Science 2013}, 
		volume 7551 of {\em Lecture Notes in Computer Science}, 
		551--558, 2013.
		
		\bibitem{HI81} I. Holyer,   
		The NP-completeness of edge-coloring,
		{\em SIAM Journal on Computing} \textbf{10} (1981) 718--720. 
		
		\bibitem{GA92} A. Gy\'{a}rf\'{a}s,   
		Graphs with k odd cycle lengths,
		{\em Discrete Math} \textbf{103} (1992) 41--48. 
		
		\bibitem{KM07} M. Kami\'{n}ski and V. V. Lozin,  
		Coloring edges and vertices of graphs without short or long cycles,
		{\em Contributions to Discrete Mathematics} \textbf{2} (2007) 61--66. 
		
		\bibitem{KR72} R. Karp,  
		Reducibility among combinatorial problems, 
		in: R. Miller and J. Thatcher, editors, 
		{\em Complexity of Computer Computations}, 85--103, 
		Plenum Press, New York, 1972. 
		
		\bibitem{KD01} D. Kr\'{a}l, J. Kratochv\'{i}l, Zs. Tuza and G. J. Woeginger,  
		Complexity of coloring graphs without forbidden induced subgraphs, in: M. C. Golumbic, M. Stern, A. Levy, and G. Morgenstern, editors,
		{\em Proceedings of the International Workshop on Graph-Theoretic Concepts in Computer Science 2001} volume 2204 of {\em Lecture Notes in Computer Science}, 
		pages 254--262, 2001.  
		
		\bibitem{LD83} D. Leven and Z. Galil,  
		NP-completeness of finding the chromatic index of regular graphs,
		{\em Journal of Algorithms} \textbf{4} (1983) 35--44. 
		
		\bibitem{MP04} P. Mih\'{o}k and I. Schiermeyer,  
		Cycle lengths and chromatic number of graphs,
		{\em Discrete Mathematics} \textbf{286} (2004) 147--149. 
		
		\bibitem{RB01} B. Randerath and I. Schiermeyer,  
		Colouring graphs with prescribed induced cycle lengths,
		{\em Discussions Mathematicae Graph Theory} \textbf{21} (2001) 267--282, 
		(An extended abstract appeared in SODA 1999). 
		

		\bibitem{VV76} V. Vizing,  
		Coloring the vertices of a graph in prescribed colors,
		{\em Metody Diskretnogo Analiza} \textbf{29} (1976) 3--10.
	\end{thebibliography}
\end{document}